\documentclass{elsarticle}

\usepackage{amsfonts}

\usepackage{amsthm}
\usepackage{amssymb}
\usepackage{latexsym}
\usepackage{amsmath,amsfonts}
\usepackage{mathrsfs}
\usepackage{cases}
\usepackage{latexsym,bm}
\usepackage{indentfirst}
\usepackage{color}
\usepackage{ifpdf}

\usepackage{graphicx}
\usepackage{psfrag}
\usepackage{dsfont}
\usepackage{graphicx}
\usepackage{epstopdf}
\usepackage[
pdfauthor={Guangming Jing},
pdftitle={On Edge coloring of Multigraphs},
pdfstartview=XYZ,
bookmarks=true,
colorlinks=true,
linkcolor=blue,
urlcolor=blue,
citecolor=blue,
bookmarks=true,
linktocpage=true,
hyperindex=true
]{hyperref}

\DeclareMathOperator{\IE}{IE}
\DeclareMathOperator{\RE}{RE}

\newtheorem{THM}{\textbf{Theorem}}[section]

\newtheorem{DEF}{\textbf{Definition}}

\newtheorem{LEM}{\textbf{Lemma}}[section]

\newtheorem{conj}{\textbf{Conjecture}}
\newtheorem{COR}{\textbf{Corollary}}[section]

\newtheorem{REM}{\textbf{Remark}}

\newtheorem{CLA}{\textbf{Claim}}[section]
\newtheorem{case}{Case}
\newtheorem{subcase}{Case}
\numberwithin{subcase}{case}
\newtheorem{subsubcase}{Case}
\numberwithin{subsubcase}{subcase}

\newcommand{\TT}{\mathcal {T}}
\newcommand{\D}{\Delta}

\newcommand{\phibar}{\overline{\varphi}}
\newcommand{\phiv}{\varphi}
\newcommand{\mubar}{\overline{\mu}}
\newcommand{\pibar}{\overline{\pi}}
\newcommand{\sigmabar}{\overline{\sigma}}
\makeatletter
\def\ps@pprintTitle{%
	\let\@oddhead\@empty
	\let\@evenhead\@empty
	\let\@oddfoot\@empty
	\let\@evenfoot\@oddfoot
}
\makeatother

\begin{document}

\title{On Edge Coloring of Multigraphs}




\author{Guangming Jing\fnref{label1}}
\ead{jingguangming@gmail.com}
\fntext[label1]{The author is supported in part by NSF DMS-2348740.}
\affiliation[label1]{organization={Department of Mathematics and Statistics, University of Massachusetts Lowell},
	addressline={Lowell, Massachusetts},
	postcode={01854},
	country={USA}}


\begin{abstract}
 Let $\Delta(G)$ and $\chi'(G)$ be the maximum degree and chromatic index of a graph $G$, respectively.
 Gupta\,(1967), Goldberg\,(1973), Andersen\,(1977), and Seymour\,(1979) made the following conjecture:
Every multigraph $G$ satisfies $\chi'(G) \le \max\{ \Delta(G) + 1, \Gamma(G) \}$, where $\Gamma(G) = \max_{H \subseteq G, |V(H)|\geq 2} \left\lceil \frac{ |E(H)| }{ \lfloor \tfrac{1}{2} |V(H)| \rfloor} \right\rceil$ is the density of $G$. In this paper, we present a polynomial-time algorithm for coloring any multigraph with $\max\{ \Delta(G) + 1, \Gamma(G) \}$ colors, confirming the conjecture algorithmically. Since $\chi'(G)\geq \max\{ \Delta(G), \Gamma(G) \}$, this algorithm gives a proper edge coloring that uses at most one more color than the optimum. As determining the chromatic index of an arbitrary graph is $NP$-hard, the $\max\{ \Delta(G) + 1, \Gamma(G) \}$ bound is best possible for efficient proper edge coloring algorithms on general multigraphs, unless $P=NP$.  Chen, Hao, Yu, and Zang have also presented an algorithm using similar high-level ideas; the present approach establishes a complete proof.

The proofs of Theorems~\ref{tech} and~\ref{con} were also verified by an
AI-assisted audit using OpenAI's GPT-6 Pro model in ChatGPT;
all resulting comments and revisions were independently evaluated by
the author.
\end{abstract}


\maketitle
\section{Introduction}
A (proper) {\bf $k$-edge-coloring}\label{kedgecoloring} of $G$ is an assignment of $k$ colors, $1,2, \ldots, k$, to the edges of $G$ 
so that no two adjacent edges have the same color. By definition, the chromatic index $\chi'(G)$ of $G$ is the minimum $k$ 
for which $G$ has a $k$-edge-coloring.  We use $[k]$ to denote the color set $\{1,2, \ldots, k\}$,
and use ${\cal C}^k(G)$ to denote the set of all $k$-edge-colorings of $G$. Note that every $k$-edge-coloring
of $G$ is a mapping from $E$ to $[k]$.   Let $\D(G)$ and $\mu(G)$ be the maximum degree  and the multiplicity of a graph $G$, respectively.   Vizing~\cite{Vizing64} showed that $\D(G) \le \chi'(G) \le \D(G) + \mu(G)$. Simple graphs are consequently divided into two classes. Class I: graphs $G$ with  $\chi'(G) = \D(G)$; Class II:
graphs $G$ with $\chi'(G) = \D(G) +1$.

Since the edges of $G$ with the same color form a matching, we have $|E(H)|\leq\chi'(G)\lfloor |V(H)|/2\rfloor$ for any subgraph $H$ of $G$ with $|V(H)|\geq 2$. Therefore, we have $\chi'(G)\geq \max\{\D(G), \Gamma(G)\}$, where the {\bf density} of $G$ is defined by

 \[\Gamma(G):=\max_{H \subseteq G, |V(H)|\geq 2} \left\lceil \frac{ |E(H)| }{ \lfloor \frac{1}{2} |V(H)| \rfloor} \right\rceil.\]

On the other hand, Gupta\, (1967)~\cite{Gupta67}, Goldberg\,(1973)~\cite{Goldberg}, Andersen\,(1977)~\cite{Andersen} and Seymour\,(1979)~\cite{Seymour} independently made the following conjecture, commonly known as the Goldberg-Seymour conjecture:

\begin{conj}[Andersen, Goldberg, Gupta, Seymour]\label{con:GAS}
Every multigraph $G$ satisfies $\chi'(G) \le \max\{\D(G) +1, \Gamma(G) \}$.
\end{conj}

 Over the past five decades, this conjecture has been a subject of extensive research in the fields of graph theory, theoretical computer science and operations research, and has inspired a significant body of work, with contributions from many researchers; see~\cite{StiebSTF-Book} by Stiebitz et al.  and~\cite{gs} by Chen, Zang, and the author for a comprehensive account. In particular,~\cite{gs} confirms the Goldberg-Seymour conjecture. However, the approach used in~\cite{gs} is quite technical and artificial, and thus it is far from obtaining an efficient edge coloring algorithm. 
 
 The key to proving the Goldberg-Seymour conjecture is to find a ``strongly closed elementary set"; see the proof of Theorem~\ref{con} for more details. The Tashkinov tree method, introduced by Tashkinov \cite{Tashkinov-2000} in the early 2000s, is the perfect tool for this purpose. Therefore, almost all recent progress towards the conjecture is based on the extensions of Tashkinov trees. To properly extend a Tashkinov tree, two major questions need to be answered:
 \begin{enumerate}
 \item Where to extend. This often involves adding requirements on the tree sequence we extend from.
 \item How to mimic Tashkinov's idea to show that the extended structure is elementary. This involves (a) Kempe changes, and (b) proving that the structure and all extension requirements remain well defined after the Kempe changes.
 \end{enumerate}

 Question (2) was answered naturally in~\cite{2017paper} and~\cite{gs} by utilizing the idea of interchangeability. This idea shows that under certain conditions, for a color $\alpha$ missing at a vertex inside an extended Tashkinov tree $T_n$ (that is a Tashkinov-type tree sequence extended $n-1$ times) and a color $\delta$ appearing on the boundary of $T_n$, there is at most one $(\alpha,\delta)$-path sharing vertices with $T_n$. All the other $(\alpha,\delta)$-components sharing vertices with $T_n$ must be cycles. In this way, we can mimic Tashkinov's idea in proving the elementariness after extending $T_n$ by avoiding this unique $(\alpha,\delta)$-path sharing vertices with $T_n$.
 
 On the other hand, Question (1) was not fully answered in~\cite{gs}. Pick an arbitrary color $\alpha$ missing at an arbitrary vertex inside $T_n$ and a ``defective color" $\delta$ (defective means that there is more than one edge colored by $\delta$ that appears on the boundary of $T_n$). Ideally, we should be able to naturally extend $T_n$ with any $\delta$-colored defective boundary edge except for the one that is incident with the $(\alpha,\delta)$-exit vertex of $T_n$, which is the last vertex of the only $(\alpha,\delta)$-path (uniqueness from interchangeability) shared with $T_n$ along the direction starting inside $T_n$. One can refer to the later Figure~\ref{figure1} for a graphical explanation. Asplund and McDonald~\cite{AsplundMcD16} also suggested a similar type of extension. However, \cite{gs} cannot guarantee that the extension structure is well-defined after performing Kempe changes. To be more specific, after Kempe changes involving the color $\alpha$ or $\delta$, the $(\alpha,\delta)$-path intersecting $T_n$ might change its exit. So the entire proof collapses. To overcome this obstacle, \cite{gs} introduced a sophisticated way of extending a Tashkinov tree with a lot of ``maximum" rules while changing the coloring of $G$ along the extension process. As a result, it does not provide a polynomial-time algorithm
 for constructing the required edge coloring.
 
 To address this problem, in this paper we show that either such Kempe changes do not change the exit of the aforementioned $(\alpha,\delta)$-path intersecting $T_n$, or there is a way to reduce the ``size" (not the size of the vertex set) of a counterexample. See Lemmas~\ref{A2} to~\ref{A3.5} and their proofs for more details on this. So Question (1) is answered naturally, which leads to a polynomial-time edge coloring algorithm.

In particular, we confirm Conjecture~\ref{con:GAS} by finding a polynomial-time algorithm to properly edge color any (multi)graph $G$ with $\max\{\D(G) +1,  \Gamma(G) \}$ many colors. Since $\chi'(G)\geq \max\{\D(G), \Gamma(G)\}$, it is equivalent to saying that:
\begin{THM}\label{main}
	There exists a polynomial-time algorithm to find a proper edge coloring for any given multigraph $G$ with $\D(G)+1$ colors if $\D(G)\leq \chi'(G)\leq \D(G)+1$, and with exactly $\chi'(G)=\Gamma(G)$ colors if $\chi'(G)> \D(G)+1$.
	
\end{THM}	
Holyer~\cite{Holyer81} proved that it is $NP$-hard in general to determine $\chi'(G)$,  even when restricted to a simple cubic graph.  In this sense, the  $\max\{\D(G) +1,  \Gamma(G) \}$ bound is best possible for efficient edge coloring algorithms. Theorem~\ref{main} also confirms the following conjecture by Hochbaum, Nishizeki, and Shmoys~\cite{HochNS80}.

\begin{conj}[Hochbaum, Nishizeki, Shmoys]
	There exists a polynomial-time algorithm to find a proper edge coloring for any given multigraph $G$ with $\max\{\Delta(G)+1,\Gamma(G)\}$ many colors.
\end{conj}


It is worth noting that in July 2024 (the first draft of this paper was posted on arXiv in August 2023), Chen, Hao, Yu, and Zang~\cite{short} also posted a related result on arXiv using a similar approach. In their draft, Lemmas~\ref{A3}–\ref{A3.5} are combined into a single lemma without incorporating the techniques developed in Lemma~\ref{A3.2}. As explained in Remark~\ref{gap}, this omission leaves a critical case unresolved.

\section{Notation and terminology}
We will generally follow~\cite{StiebSTF-Book,gs} for notation and terminology.  Let $G=(V,E)$ be a multigraph.  For each $X \subseteq V$, let $G[X]$ denote the subgraph of $G$ induced by $X$, and let 
$G-X$ denote $G[V-X]$; we write $G-x$ for $G-\{x\}$. Moreover, we use $\partial(X)$ to denote the set of all 
edges with precisely one end in $X$, and write $\partial(x)$ for $\partial(\{x\})$. Edges in $\partial(X)$ are called {\bf boundary edges} of $X$. For each pair $x, y 
\in V$, let $E(x,y)$ denote the set of all edges between $x$ and $y$. As it is no longer appropriate to represent 
an edge $f$ between $x$ and $y$ by $xy$ in a multigraph, we write $f \in E(x,y)$ instead. For each subgraph $H$ of $G$, 
let $V(H)$ and $E(H)$ denote the vertex set and edge set of $H$, respectively, let $|H|=|V(H)|$, and let $G[H]=G[V(H)]$ and $\partial(H) =\partial(V(H))$.     

Let $e$ be an edge of $G$. A {\bf tree-sequence}\label{tsequence} with respect to $G$ and $e$ is a sequence 
$T=(y_0,e_1,y_1, \ldots, e_p, y_p)$ with $p\ge 1$, consisting of distinct edges $e_1,e_2, \ldots, e_p$ and 
distinct vertices $y_0,y_1, \ldots, y_p$, such that $e_1=e$ and each edge $e_j$ with $1\le j \le p$ is between 
$y_j$ and some $y_i$ with $0\le i <j$. Let $E(T)$ and $V(T)$ be the edge set and vertex set of a tree-sequence $T$, respectively. So $E(T)$ indeed induces a tree as a subgraph of $G$.  Given a tree-sequence $T=(y_0,e_1,y_1, \ldots, e_p, y_p)$, we can naturally 
associate a linear order $\prec$ with its vertices, such that $y_i \prec y_j$ if $i<j$. We
write $y_i \preceq y_j$ if $i\le j$.  This linear order will be used repeatedly in subsequent sections.
For each vertex $y_j$ of $T$ with $j \ge 1$, let $T(y_j)$ denote $(y_0,e_1,y_1, \ldots, e_j, y_j)$. For the first vertex \(v\) of \(T\), set \(T(v)=(v)\). Clearly, 
$T(y_j)$ is also a tree-sequence with respect to $G$ and $e$. We call $T(y_j)$ the {\bf segment}\label{segment} of $T$ induced
by $y_j$. Let $T_1$ and $T_2$ be two tree-sequences with respect to $G$ and $e$. We write $T_1 \subseteq T_2$ if $T_1$ is a segment of $T_2$, and write $T_1 \subset T_2$ if $T_1$ 
is a proper segment of $T_2$; that is, $T_1 \subseteq T_2$ and $T_1 \ne T_2$.

Let $\varphi$ be a $k$-edge-coloring of $G$. For each $\alpha \in [k]$, the edge set $E_{\varphi, \alpha}=\{e\in E:\, 
\varphi(e)=\alpha\}$ is called a {\bf color class}\label{colorclass}, which is a matching in $G$ if it is nonempty.  For any two distinct colors 
$\alpha$ and $\beta$ in $[k]$, let $H$ be the spanning subgraph of $G$ with $E(H)=E_{\varphi, \alpha} 
\cup E_{\varphi, \beta}$. Then each component of $H$ is either a path or an even cycle; we refer to such a component 
as an $(\alpha, \beta)$-{\bf chain} with respect to $\varphi$, and also call it an $(\alpha, \beta)$-{\bf path} \label{alphabetapath}
(resp. $(\alpha, \beta)$-{\bf cycle}) if it is a path (resp. cycle). Possibly a component of $H$ is an isolated
vertex. We use $P_v(\alpha, \beta, \varphi)$ to denote the unique $(\alpha, \beta)$-chain containing the vertex $v$. 
Clearly, for any two distinct vertices $u$ and $v$, $P_u(\alpha, \beta, \varphi)$ and $P_v(\alpha, \beta, \varphi)$ 
are either identical or vertex-disjoint. In this paper, $P_v(\alpha, \beta, \varphi)$ usually is an $(\alpha,\beta)$-path with $v$ as an end, unless specified otherwise. Let $C$ be an $(\alpha, \beta)$-chain with respect to $\varphi$, and let 
$\varphi'$ be the $k$-edge-coloring arising from $\varphi$ by interchanging $\alpha$ and $\beta$ on $C$. We say 
that $\varphi'$ is obtained from $\varphi$ by {\bf recoloring} $C$, and write $\varphi'=\varphi /C$. This 
operation is called a {\bf Kempe change}\label{kempe}. 

Let $F$ be an edge subset of $G$. As usual, $G-F$ stands for the multigraph obtained from $G$ by deleting all
edges in $F$; we write $G-f$ for $G-\{f\}$. Let $\varphi \in {\cal C}^k(G-F)$. For each $K \subseteq E$, 
define $\varphi(K)=\{\varphi(e):e\in K\setminus F\}$. For each $v \in V$, define 
\[\varphi(v)=\varphi(\partial(v)) \hskip 2mm {\rm and} \hskip 2mm \overline{\varphi}(v)=[k]-\varphi(v).\] 
We call $\varphi(v)$ the set of colors {\bf present} at $v$ and call $\overline{\varphi}(v)$ the set of colors {\bf missing}\label{missingcolor}
at $v$.  For each $X\subseteq V$, define the set of colors {\bf missing} at $X$ as
\[\overline{\varphi}(X)= \cup_{v\in X} \, \overline{\varphi}(v).\] 
We call $X$ {\bf elementary}\label{elementaryset} with respect to $\varphi$ if $\overline{\varphi}(u) \cap \overline{\varphi}(v)
=\emptyset$ for any two distinct vertices $u, v\in X$. We call $X$ {\bf closed}\label{closedsets} with respect to $\varphi$ if
$\varphi(\partial(X))\cap \overline{\varphi}(X)=\emptyset$; that is, no missing color of $X$ appears on the 
edges in $\partial(X)$.  Furthermore, we call $X$ {\bf strongly closed}\label{sclosed} with respect to $\varphi$ if $X$ is closed
with respect to $\varphi$ and $\varphi(e) \ne \varphi(f)$ for any two colored edges $e, f \in 
\partial(X)$ with $e\neq f$.  For each subgraph or tree-sequence $H$ of $G$, write $\overline{\varphi}(H)$ for $\overline{\varphi}(V(H))$,
and write ${\varphi}(H)$ for ${\varphi}(E(H))$. If more than one boundary edge of $H$ (edge in $\partial(H)$) is colored by $\alpha$, we call $\alpha$ a {\bf defective color}\label{defectivecolor} 
of $H$ with respect to $\varphi$, call each $\alpha$-colored boundary edge a {\bf defective edge}\label{defectiveedge} of $H$ with 
respect to $\varphi$, and call the end vertex of a defective edge in $H$ a {\bf defective vertex}\label{defectivevertex} of $H$ with 
respect to $\varphi$. A color $\alpha \in \overline{\varphi}(H)$ is called {\bf closed} in $H$ under $\varphi$ 
if $H$ has no $\alpha$-colored boundary edges.  For convenience, we say that $H$ is {\bf closed} (resp. {\bf strongly closed}) 
with respect to $\varphi$ if $V(H)$ is closed (resp. strongly closed) with respect to $\varphi$. Let $\alpha$ and $\beta$ be 
two colors that are not assigned to $\partial(H)$ under $\varphi$. We use $\varphi/(G-H, \alpha, \beta)$ to denote the 
coloring $\varphi'$ obtained from $\varphi$ by interchanging $\alpha$ and $\beta$ in $G-V(H)$. Since 
$\varphi$ belongs to ${\cal C}^k(G-F)$, so does $\varphi'$.

\section{Extended Tashkinov Tree (ETT)}

For simplicity, we will present our proof based on the concept of critical edges. An edge $e$ of a graph $G$ is {\bf critical} if $\chi'(G-e)<\chi'(G)$. Throughout this paper, by 
a $k$-{\bf triple}\label{ktriple} we mean a multigraph $G=(V,E)$ such that $\chi'(G)=k+1$ and $k \ge \Delta(G)+1$,  together with an 
uncolored critical edge $e\in E$ and a coloring $\varphi \in {\cal C}^k(G-e)$; we denote it by $(G,e, \varphi)$.  A graph $G$ is {\bf $k$-critical} if $\chi'(G)=k+1$ and every edge of $G$ is critical.
 
Let $(G,e, \varphi)$ be a $k$-triple. A {\bf Tashkinov tree}\label{tashkinovtree} with respect to $e$ and $\varphi$ is a 
tree-sequence $T=(y_0,e_1,y_1, \ldots, e_p, y_p)$ with respect to $G$ and $e$, such that 
for each edge $e_j$ with $2\le j \le p$, there is a vertex $y_i$ with $0 \le i <j$ satisfying
$\varphi(e_j) \in  \overline{\varphi}(y_i)$.

The following theorem is due to Tashkinov \cite{Tashkinov-2000}; its proof can also be found in Stiebitz et al. 
\cite{StiebSTF-Book}  (see Theorem 5.1 on page 116).

\begin{THM}\label{TashTree}
	Let $(G,e, \varphi)$ be a $k$-triple and let $T$ be a Tashkinov tree with respect to $e$ and $\varphi$.
	Then $V(T)$ is elementary with respect to $\varphi$. 
\end{THM}

To carry out a proof of Conjecture \ref{con:GAS}, we will need the following generalizations.

\begin{DEF} \label{TAA}
	{\rm Given a $k$-triple $(G,e, \varphi)$ and a tree-sequence $T$ with respect to $G$ and $e$, we say that a tree-sequence $(T,f,y)$ is obtained from $T$
		by a {\bf Tashkinov augmentation}\label{TA} (TA) under $\varphi$ if $\varphi(f)\in\phibar(T)$, one end $x$ of $f$ is contained in $T$, and 
		the other end $y$ of $f$ is outside $T$. A {\bf Tashkinov augmentation algorithm} (TAA)\label{TAAalgorithm} consists of a sequence of TAs under the same
		edge coloring. We call a tree-sequence $T'$ a {\bf closure}\label{closure} of $T$ under $\varphi$ if $T'$ arises from $T$ by 
		TAA and cannot grow further by TA under $\varphi$ (equivalently, $T'$ is closed).} 
\end{DEF}

So a Tashkinov tree with respect to $e$ and $\varphi$ is a tree-sequence obtained from $(y_0,e,y_1)$ by TAA, where
$y_0$ and $y_1$ are two ends of the uncolored edge $e$. We point out that, although there might be several ways to construct a closure of $T$ under $\varphi$, the 
vertex set of these closures is unique. Indeed, for two closures $T'_1=(y_0,e_1,y_1, \ldots, e_n, y_n)$ and $T'_2$ of $T$, if $V(T'_1)- V(T'_2)\neq\emptyset$, then for the smallest $i$ with $y_i\in V(T'_1)$ but not in $V(T'_2)$, we have $\varphi(e_i)\in\phibar(T'_2)$ and $e_i\in\partial(T'_2)$, a contradiction.

Let $C$ be an $(\alpha,\beta)$-chain under $\varphi$ and $T$ be a tree-sequence. A subpath $P_{ex}$ of $C$ with ends $u$ and $v$ is called an {\bf $(\alpha,\beta)$-exit path} of $T$
under $\varphi$ if $V(T) \cap V(P_{ex})=\{v\}$ and $\overline{\varphi}(u) \cap \{\alpha,\beta\} \ne \emptyset$; 
in this case, $v$ is called the corresponding $(\alpha,\beta)$-{\bf exit vertex} of $T$ and edge $f$ in $\partial(T)\cap E(P_{ex})$ is called the corresponding $(\alpha,\beta)$-{\bf exit edge} of $T$. Note that possibly $\overline{\varphi}(v) \cap \{\alpha,\beta\}=\emptyset$, but $C$ must be an $(\alpha,\beta)$-path if it contains an $(\alpha,\beta)$-exit path of $T$. A subpath $D$ of $C$ with ends $u$ and $v$ is called an {\bf $(\alpha,\beta)$-ear} of $T$ if $|V(D)|\geq 3$ and $V(T)\cap V(D)=\{u,v\}$; in this case, $u$ and $v$ are called {\bf roots} of $D$ in $T$. See Figure~\ref{figure1} for reference.

\begin{figure}[!h]
	\begin{center}
		\includegraphics[scale =.45
		]{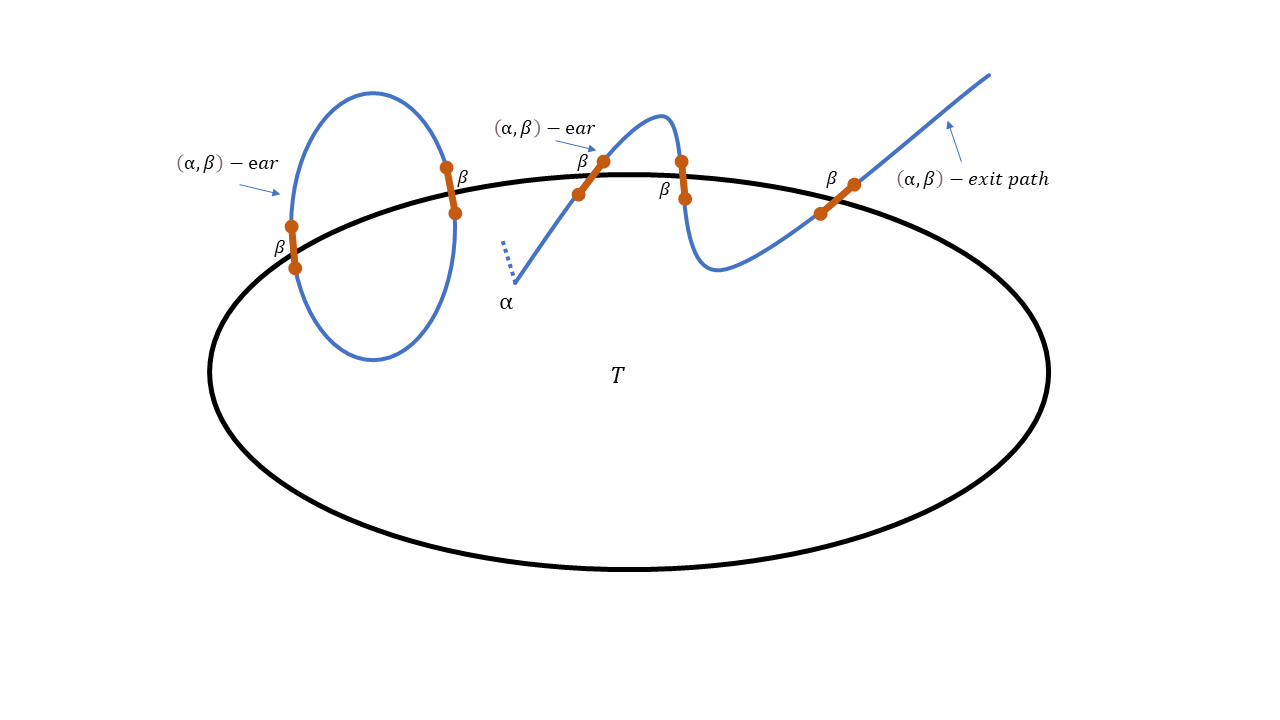}
		\caption{}
		\label{figure1}
	\end{center}
\end{figure}

Given a $k$-triple $(G,e,\varphi)$, a {\bf Tashkinov Series} is a series of tuples $(T_{n+1}, S_n, F_n, \Theta_n)$ for $n=0,1, 2 \ldots $ constructed inductively using the following algorithm, where $T_n+f_n$ stands for a tree-sequence 
augmented from $T_n$ by adding a boundary edge $f_n$ of $T_n$. It is worth pointing out that iteration $n$ includes iteration $1$ when $n=1$, but we still write iteration $1$ out for a better presentation.

\newpage
\noindent {\bf Algorithm 3.1}

\noindent {\bf Iteration 0.} Let $(T_1,  S_0, F_0, \Theta_0)$ be the initial tuple, such that  $T_1$ is a closure of $e$ under 
$\varphi$ (which is a closed Tashkinov tree with respect to $e$ and $\varphi$), $f_0=e$, $F_0=\{f_0\}$, and $S_0 =\Theta_0=\emptyset$.  

\noindent {\bf Iteration 1.}
\begin{itemize}
\item If $T_1$ has a $(\gamma_1,\delta_1)$-ear $Q$ where $\delta_1$ is a defective color of $T_1$ under $\varphi$ and $\gamma_1\in\phibar(T_1)$, we construct $(T_2, S_1, F_1, \Theta_1)$ as follows:\footnote{We shall see later that Theorem~\ref{tech} (A1) and (A2) guarantee the existence of such a $(\gamma_1,\delta_1)$-ear $Q$ if $T_1$ is not strongly closed under $\varphi$. In fact, for any color $\gamma_1\in\phibar(T_1)$, all the $\delta_1$-colored defective edges of $T_1$ are contained in some  $(\gamma_1,\delta_1)$-ears except for one which is contained in a $(\gamma_1,\delta_1)$-exit path. See more details in the proof of Theorem~\ref{con}.} pick an arbitrary $\delta_1$-colored edge $f_1\in\partial(T_1)$ (boundary edge of $T_1$) that is contained in the $(\gamma_1,\delta_1)$-ear $Q$, set $F_1=\{f_1\}$, $S_1=\{\gamma_1,\delta_1\}$, and $\Theta_1=\IE$ (Here $\IE$ stands for {\bf Initial Extension} which is the extension type). Let $T_2$ be a closure of $T_1+f_1$ under $\varphi$.
\item Else, stop.
\end{itemize}

\noindent {\bf Iteration} ${\bm n}$ $(n\geq 1)$. 
\begin{itemize}

	\item If there exists a subscript $h \le n-1$ such that $\Theta_h =\IE$ and $S_h =\{\delta_h, \gamma_h\}$, and there is
	a $(\gamma_h, \delta_h)$-ear $O$ of $T_h$ containing a vertex outside $V(T_n)$, we construct $(T_{n+1}, S_n, F_n, \Theta_n)$ as follows: pick an arbitrary edge $f_n$ in $O \cap \partial(T_n)$ such that $O$ contains a path $L$
	connecting $f_n$ and $V(T_h)$ with $V(L) \subseteq V(T_n)$.  Set $F_n=\{f_n\}$, $S_n=\{\gamma_n,\delta_n\}$ with $\gamma_n=\gamma_h$ and $\delta_n=\delta_h$, and $\Theta_n=\RE$. Let $T_{n+1}$ be a closure of $T_n+f_n$ under $\varphi$. This type of extension is called a {\bf Revisiting Extension ($\RE$)}.

	\item Else,  if $T_n$ has a $(\gamma_n,\delta_n)$-ear $Q$ where $\delta_n$ is a defective color of $T_n$ under $\varphi$ and $\gamma_n\in\phibar(T_n)$, we construct $(T_{n+1}, S_n, F_n, \Theta_n)$ as follows: pick an arbitrary edge $f_n\in\partial(T_n)\cap Q$, set $F_n=\{f_n\}$, $S_n=\{\gamma_n,\delta_n\}$, and $\Theta_n=\IE$. Let $T_{n+1}$ be a closure of $T_n+f_n$ under $\varphi$. This type of extension is called an {\bf Initial Extension ($\IE$)}.
	\item Else, stop.
	\end{itemize}

	In the above algorithm, for each $1\leq i\leq n$, $f_i, \delta_i$, and $\gamma_i$ are called the {\bf connecting edge}, {\bf connecting color}, and {\bf companion color} of $T_i$, respectively. We say $T_i$ is {\bf $\RE$ finished} if for each $h\leq i-1$ with $\Theta_{h}=\IE$, all the $(\gamma_h,\delta_h)$-ears of $T_h$ are contained in $G[V(T_i)]$, that is, $T_i$ does not satisfy the requirement for an $\RE$ extension. So $T_1$ is $\RE$ finished. By convention, we define $T_0=\emptyset$ and also say $T_0$ is $\RE$ finished. 
\begin{REM}	
	It is worth pointing out that since $RE$ extensions always have higher priority than $\IE$ extensions, for each $T_i$ which we will perform an $\IE$ extension on, $T_i$ must be $\RE$ finished. In fact, each time after an $\IE$ extension, we perform a series of $\RE$ extensions until it is $\RE$ finished, before performing the next $\IE$ extension. This mechanism ensures that, before the next $\IE$ extension, all $(\gamma_n,\delta_n)$-ears of $T_n$ are eventually contained after an $\IE$ extension of $T_n$.
\end{REM}	
	  For each $T_i$ and extension type $\Theta_{i-1}$ with $i\geq 1$, we call the index $h$ with $0\leq h\leq i-1$ the {\bf initial index} of $T_i$ and $\Theta_{i-1}$, if $h$ is the largest index with $\Theta_h=\IE$ when $i\geq 2$, and if $h=0$ when $i=1$.  So $T_h$ is $\RE$ finished and $\Theta_j=\RE$ for each $h<j<i-1$.  Define $D_n=\cup_{1\leq i\leq n} S_i-\phibar(T_n)$. 
In the remainder of this paper, we reserve notations $T_{i}, S_i, F_i, \Theta_i, D_i, \gamma_i,$ and $\delta_i$ for Tashkinov series built from Algorithm 3.1.

\begin{DEF} \label{DEF-ETT}
  Let $(G,e, \varphi)$ be a $k$-triple and let $\TT=\{(T_i, S_{i-1}, F_{i-1}, \Theta_{i-1}): 
  1\le i \le n+1\}$  be a Tashkinov series constructed from $(G,e, \varphi)$ following Algorithm 3.1.  A tree-sequence $T$ is called an {\em extended Tashkinov 
  	tree} (ETT)\label{ETT} constructed from $\TT$ if $T_n\subset T\subseteq T_{n+1}$. $\TT$ is called the corresponding Tashkinov series of $T$.
The unique nonnegative integer $n$ such that $T_n\subsetneq T\subseteq T_{n+1}$ is called the {\bf rung number} of $T$ and denoted by $n(T)$. The sequence $T_0\subset T_1\subset \dots \subset T_n \subset T$ is called the {\bf ladder} of $T$. 
\end{DEF}

The idea of extending from the ear is natural. In fact, the author first got this idea in 2017 from~\cite{AsplundMcD16} by Asplund and McDonald; see Theorem 2 and Figure 2 in their paper. But due to difficulties proving elementariness and keeping the structure stable, \cite{gs} used a very technical method to extend a Tashkinov tree instead.

Clearly each $T_i$ in a Tashkinov series is a closed ETT.
Note that if $T$ is an ETT then any segment $T'$ of $T$ is also an ETT.  
Let $T$ be an ETT under $\varphi$ and
$C$ be a subset of $[k]$. We say
that an edge $f$ of $G$ is {\bf incident} to $T$ if at least one end of $f$ is contained in $T$; this definition
applies to edges of $T$ as well.  Since our proof consists of a sophisticated sequence of Kempe changes, 
the concept of stable coloring introduced below will be employed to preserve some important coloring properties 
of $T$, such as, among others, the color on each edge and the set of colors missing at each vertex. Usually, 
$C$ is the set of colors assigned to $E(T)$ but not missing at any vertex of $T$. 

To be specific, a coloring $\pi \in {\cal C}^k(G-e)$ is called a $(T, C, \varphi)$-{\bf stable coloring}\label{tcstable} 
if the following two conditions are satisfied:
\begin{itemize}
	\vspace{-2mm}
	
	\item[$(i)$] $\overline{\pi} (v) = \overline{\varphi}(v)$ for any $v\in V(T)$; and
	\vspace{-2mm}
	\item[$(ii)$] $\{f\ |\ \varphi(f)=\alpha\ and\ f\ is\ incident\ to\ T\}=\{f\ |\ \pi(f)=\alpha\ and\ f\ is\ incident\ to\ T\}$ for each color $\alpha\in\phibar(T)\cup C$.
\end{itemize}

\vspace{-2mm}
By convention, $\pi(e)=\varphi(e)=\emptyset$. 
From the definition, we see that the following statements hold for a $(T, C, \varphi)$-stable 
coloring $\pi$: 

$\bullet$ being $(T, C, \cdot )$-stable is an equivalence relation on ${\cal C}^k(G-e)$;

$\bullet$ if $T' \subseteq T$ and $\overline{\varphi}(T')\cup C' \subseteq \overline{\varphi}(T)\cup C$, then $\pi$ is also $(T', C', \varphi)$-stable;

$\bullet$ if a color $\alpha \in \overline{\varphi}(T)$ is closed in $T$ under $\varphi$, then it is also closed in $T$ under $\pi$; and

$\bullet$ if ${\varphi}(T) \subseteq \overline{\varphi}(T)\cup C$, then $\pi(f) = \varphi(f)$ for all edges $f\in E(T)$.

Let  $T_n$ be an ETT from a Tashkinov series with $(n-1)$ rungs under $\varphi$ and $\pi$ be a $(T_n, D_{n-1}, \phiv)$-stable coloring.  It is worth pointing out that although $\pi$ is $(T_n, D_{n-1}, \phiv)$-stable, $T_n$ might not be an ETT under $\pi$, because a $(\gamma_i,\delta_i)$-ear of $T_i$ under $\varphi$ could become a $(\gamma_i,\delta_i)$-exit path of $T_i$ under $\pi$ for some $i\leq n$.
Therefore, we introduce the following concept.
\begin{DEF}\label{mod}
Let $\TT=\{ (T_i, S_{i-1}, F_{i-1}, \Theta_{i-1})\ : \
1 \le i \le n+1\}$ be a Tashkinov series.  A $(T_{n+1},D_{n},\varphi)$-stable coloring $\pi$ is called {\bf $T_{n+1}$ mod $\phiv$} if  $\TT=\{ (T_i, S_{i-1}, F_{i-1}, \Theta_{i-1})\ : \
1 \le i \le n+1\}$ remains a Tashkinov series under $\pi$; and a $(T_{n},D_{n},\varphi)$-stable coloring $\phi$ is called {\bf $T_{n}+f_n$ mod $\phiv$} if $\TT=\{ (T^*_i, S_{i-1}, F_{i-1}, \Theta_{i-1})\ : \
1 \le i \le n+1\}$ is a Tashkinov series under $\phi$, where $T^*_i=T_i$ for each $1\leq i\leq n$, and $T^*_{n+1}$ is a closure of $T_n+f_n$ under $\phi$. By convention, a {\bf $T_0+f_0$ mod $\varphi$} coloring is defined as any $k$-edge-coloring of $G-e$, $T_0+f_0$ is just a Tashkinov tree with the uncolored edge $e$ and its two ends, and a closure of $T_0+f_0$ is a closure of $e$.  

\end{DEF}

Let $T$  be an ETT under $\varphi$ with $n$ rungs and corresponding Tashkinov series $\TT=\{ (T_i, S_{i-1}, $ $F_{i-1}, \Theta_{i-1})\ : \
1 \le i \le n+1\}$. Clearly, if $\pi$ is both $(T, D_{n},\varphi)$-stable and $T_n+f_n$ mod $\varphi$, then $T$ or any tree-sequence obtained from $T_n+f_n$ by TAA under $\pi$ is also an ETT with the same ladder, connecting edges, connecting colors, companion colors, and extension types as under $\varphi$. Moreover, mod colorings form equivalence classes. Let $v$ be the last vertex of $T$. Denote by $T-v$ the segment of $T$ ending at the vertex before $v$ along $\prec$. Let $T$ be an ETT with $n$ rungs and $T_n\subset T$ under $\varphi$. We call a $(T_{n},D_n,\varphi)$-stable coloring $\pi$  {\bf $(T,D_n,\varphi)$-weakly stable} if $\varphi(f)=\pi(f)$ for each $f\in E(T)$ and $\phibar(u)=\pibar(u)$ for each $u\in V(T-v)$.

\begin{LEM}\label{wstable}
Let $T$ be an ETT with ladder $T_0\subset T_1 \subset \dots \subset T_n \subset T$ with respect to a $k$-triple $(G,e,\varphi)$, and $v$ be the last vertex of $T$. If a $(T,D_n,\varphi)$-weakly stable coloring $\pi$ is also $T_n+f_n$  mod $\varphi$ ($n=0$ may happen), then $T$ remains an ETT under $\pi$.
\end{LEM}
	
\begin{proof}
	
	Let $T$, $v$ be described as above, $f$ be the last edge of $T$, and $\pi$ be a $(T,D_n,\varphi)$-weakly stable coloring which is also $T_n+f_n$ mod $\varphi$. 
	If $f=f_n$, then $T=T_n+f_n$ is an ETT under $\pi$ because $\pi$ is $T_n+f_n$ mod $\varphi$. So we assume $f\neq f_n$, and therefore $\varphi(f)\in\phibar(T-v)$ following Algorithm 3.1. Since $\pi$ is $(T,D_n,\varphi)$-weakly stable, we see that $\varphi(f)=\pi(f)\in\pibar(T-v)$, so $T$ can still be obtained by TAA from $T_n+f_n$ under $\pi$ the same way as under $\varphi$. Because $\pi$ is also $T_n+f_n$ mod $\varphi$, $T$ remains an ETT under $\pi$ as desired. 
\end{proof}
In view of the lemma above, it is easy to see that if $\pi$ is $(T-v,D_n,\varphi)$-stable, then $\pi$ is also $(T, D_n,\varphi)$-weakly stable, because under $\varphi$, the last edge of $T$ is either colored by a missing color of $V(T-v)$, or it is colored by a color in $D_n$. So we have the following corollary.
\begin{COR}\label{vstable}
Let $T$ be an ETT with ladder $T_0\subset T_1 \subset \dots \subset T_n \subset T$ with respect to a $k$-triple $(G,e,\varphi)$, and $v$ be the last vertex of $T$. If a $(T-v,D_n,\varphi)$-stable coloring $\pi$ is also $T_n+f_n$  mod $\varphi$ ($n=0$ may happen), then $T$ remains an ETT under $\pi$.

\end{COR}

The following lemma gives a lower bound for the number of missing colors in a closed Tashkinov tree.
\begin{LEM}\label{size}
	Let $T_1$ be a closed Tashkinov tree with respect to a $k$-triple $(G,e,\varphi)$. Then $|\phibar(T_1)|\geq 5$.
	
\end{LEM}	

\begin{proof}
	Since $T_1$ is a Tashkinov tree, $T_1$ is elementary by Theorem~\ref{TashTree}. Moreover, because $T_1$ is closed, none of the colors in $\phibar(T_1)$ is used on edges of $\partial(T_1)$. Let $\alpha\in\phibar(T_1)$. So edges colored by $\alpha$ form a matching for all but the vertex missing color $\alpha$ in $V(T_1)$. Consequently, $|V(T_1)|$ is odd. Since $e\in T_1$, $|V(T_1)|\geq 3$. Because $k>\Delta$ and the edge $e$ is uncolored, both ends of $e$ must have at least two missing colors. Thus $|\phibar(T_1)|\geq 5$ by the elementariness of $T_1$.
\end{proof}

Let $T$ be an ETT under a $k$-triple $(G,e,\varphi)$ and $\alpha,\beta$ be two colors in $[k]$.   We say a cycle or a path $P$ {\bf intersects} $T$ (resp. a subgraph $H$ of $G$) if $V(P)\cap V(T)\neq\emptyset$ (resp. ($V(P)\cap V(H)\neq\emptyset$)). If there is at most one $(\alpha,\beta)$-path intersecting $T$ under $\varphi$, then we say $\alpha$ and $\beta$ are   {\bf $T$-interchangeable} or {\bf interchangeable} for $T$. Note that $P_v(\alpha,\beta,\varphi)=\{v\}$ is also an  $(\alpha,\beta)$-path intersecting $T$ if $\alpha,\beta\in\phibar(v)$ for some $v\in V(T)$. Clearly, if $\alpha$ and $\beta$ are $T$-interchangeable, then all the $(\alpha,\beta)$-chains intersecting $T$ are cycles except for at most one.
\begin{DEF}
	A closed ETT $T$ with $n(T)=n$ under $\varphi$  has the {\bf interchangeability property} if any two colors are $T$-interchangeable provided one of them is in $\phibar(T)$ when $T$ is $\RE$ finished, and if any two colors are $T$-interchangeable provided one of them is in $\phibar(T)-\{\gamma_n\}$ when $T$ is not $\RE$ finished, where $\gamma_n=\gamma_h$ with $h$ being the initial index of $T$. 
\end{DEF}
We are now ready to present the following technical theorem.
\begin{THM}\label{tech}
	Let $n$ be a nonnegative integer and $(G, e, \phiv)$ be a $k$-triple with $k \ge \D +1$.  
If $T$ is an ETT with ladder $T_0\subset \dots \subset T_n \subset T\subseteq T_{n+1}$ under $\varphi$, then the following hold.

{\flushleft \bf A1}:
$T$ is elementary under $\phiv$.

{\flushleft \bf A2}: $T_{n+1}$ has the interchangeability property. Consequently, if $T_{n+1}$ is $\RE$ finished, then any $(T_{n+1},D_{n},\varphi)$-stable coloring $\varphi'$ is a $T_{n+1}$ mod $\varphi$ coloring, and $T_{n+1}$ remains $\RE$ finished under $\varphi'$.

{\flushleft \bf A2.5}: Suppose $T_{n+1}$ could be extended further to $T_{n+2}$ with $F_{n+1}=\{f_{n+1}\}$, $S_{n+1}=\{\gamma_{n+1},\delta_{n+1}\}$, and initial index $h$.  If $\varphi'$ is a $(T_{n+1},D_{n+1},\varphi)$-stable coloring such that  $T_{h}$ has the same $(\gamma_{h},\delta_{h})$-exit vertex under $\varphi'$ as under $\varphi$, then $\varphi'$ is a $T_{n+1}+f_{n+1}$ mod $\varphi$ coloring.

{\flushleft \bf A3}: Assume $T_{n+1}$ is $\RE$ finished, $\delta$ is a defective color of $T_{n+1}$ under $\varphi$, and $\gamma\in\phibar(T_{n+1})$. Let $\alpha,\beta$ be two colors such that $\alpha\in\phibar(T_{n+1})\setminus\{\gamma\}$ and $\beta$ is arbitrary. For any $(\alpha,\beta)$-path $P$, if $P$ does not intersect $T_{n+1}$, then $T_{n+1}$ has the same $(\gamma,\delta)$-exit vertex under $\varphi/P$ as under $\varphi$.

{\flushleft \bf A3.5:} Assume $T_{n+1}$ could be extended further with $F_{n+1}=\{f_{n+1}\}$ and $S_{n+1}=\{\gamma_{n+1},\delta_{n+1}\}$.  Let $\alpha,\beta$ be two colors such that $\alpha\in\phibar(T_{n+1})\setminus\{\gamma_{n+1}\}$ and $\beta$ is arbitrary. For any $(\alpha,\beta)$-path $P$, if $P$ does not intersect $T_{n+1}$, then the coloring $\varphi'=\varphi/P$ is a $T_{n+1}+f_{n+1}$ mod $\varphi$ coloring.

\end{THM}
 
For the remainder of this section, let us present a proof of Conjecture~\ref{con:GAS} using Theorem~\ref{tech}.
\begin{THM}\label{con}
	Let $G$ be a graph with $\chi'(G)\geq \D(G)+2$. Then $\chi'(G)= \Gamma(G)$.
\end{THM}
\begin{proof}
	Assume that $\chi'(G)\geq \D(G)+2$. Then $G$ must contain a $k$-critical subgraph $H$ with $k=\chi'(G)-1 \geq\D(G)+1\geq \D(H)+1$. Let $(H,e,\varphi)$ be a $k$-triple, and construct a Tashkinov series  $\TT=\{ (T_i, S_{i-1}, F_{i-1}, \Theta_{i-1})\ : \ 1 \le i \le n+1\}$ with maximum $n$ following Algorithm 3.1. We claim that $T_{n+1}$ is elementary and strongly closed. Indeed, $T_{n+1}$ is elementary by Theorem~\ref{tech} (A1). Moreover, it is closed as it is a closure of $T_{n}+f_n$ under $\varphi$. Since $n$ is maximum, we cannot keep extending $T_{n+1}$ by an $\RE$ extension, so $T_{n+1}$ is $\RE$ finished. Suppose, on the contrary, that $T_{n+1}$ is not strongly closed. Then there exists a defective color $\delta_{n+1}$ of $T_{n+1}$. Then $\delta_{n+1}\notin\phibar(T_{n+1})$ because $T_{n+1}$ is closed. Pick an arbitrary color $\gamma_{n+1}\in\phibar(T_{n+1})$. By Theorem~\ref{tech} (A2), $\gamma_{n+1}$ and $\delta_{n+1}$ are $T_{n+1}$-interchangeable, so $P=P_v(\gamma_{n+1},\delta_{n+1},\varphi)$ is the only $(\gamma_{n+1},\delta_{n+1})$-path intersecting $T_{n+1}$, where $v$ is the only vertex in $T_{n+1}$ with $\gamma_{n+1}\in\phibar(v)$. As the end $v$ of $P$ is in $T_{n+1}$, we only have one $(\gamma_{n+1},\delta_{n+1})$-exit vertex and one $(\gamma_{n+1},\delta_{n+1})$-exit edge of $T_{n+1}$. Therefore, a $\delta_{n+1}$-colored boundary edge $f_{n+1}$ other than the  $(\gamma_{n+1},\delta_{n+1})$-exit edge is contained in some $(\gamma_{n+1},\delta_{n+1})$-ear of $T_{n+1}$. Therefore, we could extend $T_{n+1}$ using $f_{n+1}$ with an $\IE$ extension, a contradiction to $n$ being maximum.
	
	So $T_{n+1}$ is elementary and strongly closed. Now for each color $\alpha$ in $\phibar(T_{n+1})$, its color class forms a matching for all vertices in $T_{n+1}$ but the one missing the color $\alpha$. So this gives $|\phibar(T_{n+1})|\cdot (|V(T_{n+1})|-1)/2$  many edges in $H[V(T_{n+1})]$ and $|V(T_{n+1})|$ is odd. On the other hand, any color $\beta$ in $[k]-\phibar(T_{n+1})$ appears exactly once on $\partial(T_{n+1})$, so its color class also forms a matching for all vertices in $T_{n+1}$ but the one incident with the $\beta$-colored boundary edge. Recall that $e$ is uncolored. So together we have $k \cdot (|V(T_{n+1})|-1)/2+1$ many edges in $H[V(T_{n+1})]$, and therefore $\Gamma(H)\geq \frac{k \cdot (|V(T_{n+1})|-1)/2+1}{(|V(T_{n+1})|-1)/2}>k$. Thus $k+1\leq \Gamma(H)$. Since $ \Gamma(G)$ is a lower bound for $\chi'(G)$, we have  $\Gamma(G) \leq \chi'(G)=k+1\leq\Gamma(H)\leq \Gamma(G)$, forcing $\chi'(G)=\Gamma(G)$.
\end{proof}	
	\section{Proof of Theorem~\ref{tech}}
We will prove Theorem~\ref{tech} by induction on $n=n(T)$, the number of rungs. Note that when $n=0$, (A1) holds by Theorem~\ref{TashTree}.  For the inductive step we assume Theorem~\ref{tech} holds for rungs smaller than $n$ and prove it for $n$. At each inductive step, we first obtain (A1) from Lemma~\ref{tec},
and then prove (A2), (A2.5), (A3), and (A3.5) in this order. Note that in the inductive step, by (A1) and (A2), $T_h$ has at most one $(\gamma_h, \delta_h)$-exit vertex where $h$ is the initial index of $T_{n+1}$, and that is why we say ``the same $(\gamma,\delta)$-exit vertex" in (A2.5) and (A3). 
\begin{LEM}(A2)\label{A2}
	Let $n$ be a nonnegative integer. Suppose (A1) holds for  all  ETTs with at most $n$ rungs, and (A2), (A2.5), and (A3) hold for all ETTs with at most $(n-1)$ rungs. If $T_{n+1}$ is a closed ETT with ladder $T_0\subset \dots \subset T_n \subset T_{n+1}$, then $T_{n+1}$ has the interchangeability property. Furthermore, if $T_{n+1}$ is $\RE$ finished, then any $(T_{n+1},D_{n},\varphi)$-stable coloring $\varphi'$ is a $T_{n+1}$ mod $\varphi$ coloring, and $T_{n+1}$ remains $\RE$ finished under $\varphi'$.

\end{LEM}
\begin{proof}
	Let $T_{n+1}$ be a closed ETT with ladder $T_0\subset \dots \subset T_n \subset T_{n+1}$ and $h$ be the initial index of $T_{n+1}$.
	 We first prove the ``Furthermore" part. Let $\varphi'$ be a $(T_{n+1},D_{n},\varphi)$-stable coloring and assume that $T_{n+1}$ is $\RE$ finished. If $n=0$, then $T_1$ remains a closed Tashkinov tree under $\varphi'$ and we have as desired. So we assume $1\leq h\leq n$. Since $T_{h}$ is elementary under $\varphi$ by (A1), there is exactly one vertex $v\in V(T_{h})$ with $\gamma_h\in\phibar(v)$ and $\delta_h\notin\phibar(T_h)$ as $T_h$ is closed under $\varphi$ (see Algorithm 3.1).  Therefore, by (A2) on $T_h$ which is $\RE$ finished and has $h-1<n$ rungs, there is exactly one $(\gamma_h,\delta_h)$-path intersecting $T_h$ under $\varphi$. Let $u$ and $f$ be the corresponding $(\gamma_h,\delta_h)$-exit vertex and $(\gamma_h,\delta_h)$-exit edge, respectively. Then under $\varphi$, all the $\delta_h$-colored boundary edges of $T_h$ are contained in some $(\gamma_h,\delta_h)$-ear of $T_h$, except for $f$. Since $T_{n+1}$ is $\RE$ finished, all the $(\gamma_h,\delta_h)$-ears of $T_h$ are contained in $G[V(T_{n+1})]$ (see Algorithm 3.1), and therefore they are colored the same under $\varphi'$ as $\varphi$, because $\varphi'$ is $(T_{n+1},D_{n},\varphi)$-stable and $\{\gamma_h,\delta_h\}=S_h\subseteq\phibar(T_{h})\cup D_{h}\subseteq\phibar(T_{n+1})\cup D_n$. This forces $u$ and $f$ to be the only $(\gamma_h,\delta_h)$-exit vertex and $(\gamma_h,\delta_h)$-exit edge of $T_h$ under $\varphi'$, respectively, and all the $(\gamma_h,\delta_h)$-ears of $T_h$ under $\varphi'$ are contained in  $G[V(T_{n+1})]$. So by (A2.5) with $T_n$ that has $(n-1)$ rungs in place of $T_{n+1}$, the $(T_{n+1},D_{n},\varphi)$-stable coloring $\varphi'$ is also $T_n+f_n$ mod $\varphi$. Thus any tree-sequence obtained from $T_n+f_n$ is an ETT under $\varphi'$ with $n$ rungs. Since $\varphi'$ is $(T_{n+1},D_n,\varphi)$-stable and $T_{n+1}$ is closed under $\varphi$, we see that $T_{n+1}$ remains a closure of $T_n+f_n$ under $\varphi'$, so $\varphi'$ is $T_{n+1}$ mod $\varphi$. Moreover, since all the $(\gamma_h,\delta_h)$-ears of $T_h$ under $\varphi'$ are contained in  $G[V(T_{n+1})]$, $T_{n+1}$ is also $\RE$ finished under $\varphi'$.

 Let $\alpha,\beta$ be two colors such that $\alpha\in\phibar(T_{n+1})$, and furthermore $\alpha\neq\gamma_n$ if $T_{n+1}$ is not $\RE$ finished. We now show that $T_{n+1}$ has the interchangeability property by assuming otherwise that there are two $(\alpha,\beta)$-paths $Q_1$ and $Q_2$ intersecting $T_{n+1}$. By (A1), $T_{n+1}$ is elementary with respect to $\varphi$. So  
	$T_{n+1}$ contains at most two vertices $v$ with $\overline{\varphi}(v) \cap \{\alpha, \beta\} \ne \emptyset$, 
	which in turn implies that at least two ends of $Q_1$ and $Q_2$ are outside $T_{n+1}$. Since $T_{n+1}$ is closed, $\beta\notin\phibar(T_{n+1})$.
	Thus we further deduce that at least three ends of $Q_1$ and $Q_2$ are outside $T_{n+1}$. Traversing $Q_1$ and $Q_2$ 
	from these ends respectively, we can find at least three $(\alpha, \beta)$-exit paths 
	$P_1,P_2,P_3$ for $T_{n+1}$ under $\varphi$ with corresponding exit vertices $b_1,b_2$ and $b_3$. For $i=1,2,3$, let $a_i$  be the end of $P_i$ outside of $T_{n+1}$, 
	and $f_i$ be the boundary edge of $T_{n+1}$ in $P_i$. Renaming subscripts if necessary, we may assume that $b_1\prec b_2 
	\prec b_3$ along $T_{n+1}$. Define $Q(T_{n+1})=0$ (when $n\geq 1$) if $\delta_h\in\phibar(T_{n+1}(b_2))$, and $Q(T_{n+1})=1$ otherwise. By convention, we define $Q(T_1)=0$ for any ETT $T_1$ with $0$ rungs. We call the tuple $(\varphi,T_{n+1}, \alpha, \beta, P_1,P_2,P_3)$ a {\em counterexample} and use ${\cal K}$ 
	to denote the set of all such counterexamples.
	
	Let $(\varphi, T_{n+1}, \alpha, \beta, P_1,P_2,P_3)$ be a counterexample in ${\cal K}$ such that:
	\begin{itemize}
		
		\item [(a)] $Q(T_{n+1})$ is minimum, and
		\item [(b)] $|P_1|+|P_2|+|P_3|$ is minimum subject to (a).
	\end{itemize}	


For any color $\gamma\in\phibar(T_{n+1})$ with $\gamma\neq\gamma_h$, let $\varphi_{\gamma}=\varphi/(G-T_{n+1},\alpha,\gamma)$. We claim that

(1) $\varphi_{\gamma}$ is $T_{n+1}$ mod $\varphi$, and therefore $T_{n+1}$ is an ETT with $n$ rungs under $\varphi_{\gamma}$.

Indeed, if $T_{n+1}$ is $\RE$ finished, then we have as claimed by the ``Furthermore" part we proved earlier. Now if $T_{n+1}$ is not $\RE$ finished, we have $\alpha\neq\gamma_h$ by assumption and $\delta_n\notin\phibar(T_{n+1})$. So both $\gamma$ and $\alpha$ are different from $\gamma_h=\gamma_n$ and $\delta_n$, and therefore $T_{h}$ has the same $(\gamma_{h},\delta_{h})$-exit vertex under $\varphi_{\gamma}$ as $\varphi$. So $\varphi_{\gamma}$ is $T_n+f_n$ mod $\varphi$ by (A2.5) because $\varphi_{\gamma}$ is also $(T_{n}, D_{n},\varphi)$-stable. Thus any closure of $T_n+f_n$ is an ETT with $n$ rungs under $\varphi_{\gamma}$. Finally since  $\varphi_{\gamma}$ is $(T_{n+1}, D_{n},\varphi)$-stable, we have as claimed as $T_{n+1}$ remains a closure of $T_n+f_n$ under $\varphi_{\gamma}$. 
Furthermore, we have

(2) $b_2,b_3\notin V(T_{n})$, and therefore 	$b_2,b_3\notin V(T_{h})$.

Note that (2) holds trivially if $n=0$, so we may assume $n\geq 1$, and therefore $h\geq 1$. Since $|\phibar(T_1)|\geq 5$ by Lemma~\ref{size}, there's a color $\gamma\in\phibar(T_{n})$ with $\gamma\neq \gamma_h$. So by (1), $\varphi_{\gamma}=\varphi/(G-T_{n+1},\alpha,\gamma)$ is $T_{n+1}$ mod $\varphi$ and $T_{n+1}$ is an ETT with $n$ rungs under $\varphi_{\gamma}$. Thus $T_{n}$ remains an ETT with $(n-1)$ rungs under $\varphi_{\gamma}$. Since $\gamma\neq \gamma_h$, we see that there is at most one $(\gamma,\beta)$-path intersecting $T_n$ by (A2).  Because $P_1$, $P_2$, and $P_3$ are $(\gamma, \beta)$-exit paths of $T_{n+1}$ under $\varphi_{\gamma}$ with exit vertices $b_1\prec b_2 \prec b_3$,  (2) is established. When \(n=0\), use Case I below without references to the
connecting colors; the mod-coloring assertion follows from
Definition~\ref{mod}. Depending on whether $\beta$ is $\delta_h$ or not, we consider the following two cases. 

{\flushleft {\bf Case I:}}  $\beta\neq\delta_h$.

 Let $\gamma$ be a color in $\phibar(b_3)$. Since $T_{n+1}$ is elementary by (A1) and $\gamma_h\in\phibar(T_{h})$, by (2), we have
 
 (3) $\gamma$ is different from $\gamma_n=\gamma_h$. 
 
Let $\varphi_{\gamma}=\varphi/(G-T_{n+1},\alpha,\gamma)$.

Note that $P_3$ is a $(\gamma,\beta)$-path under $\varphi_{\gamma}$ with $\gamma\in\overline{\varphi}_{\gamma}(b_3)$. So $f_i$ is colored 
	by $\beta$ under both $\varphi$ and $\varphi_{\gamma}$ for $i=1,2,3$.  
Consider $\sigma=\varphi_{\gamma}/P_3$. Note that $\beta \in \overline{\sigma}(b_3)$, and $f_i$ is colored 
	by $\beta$ under $\sigma$ for $i=1,2$.  
	Let $T'$  be a closure of $T_{n+1}(b_3)$ under $\sigma$.  Since $\beta \in \overline{\sigma}(b_3)$, by TAA,
	
	(4) both $f_1$ and $f_2$ are contained in $G[V(T')]$.
	
 We claim that 
	
	(5) $\sigma$ is $T_{n}+f_n$ mod $\varphi_{\gamma}$, and $T'$ is an ETT with $n$ rungs under $\sigma$.
	
	 Note that $T'$ can be obtained by TAA from $T_{n}+f_n$. So by (1) and (3), $T'$ is an ETT with $n$ rungs under $\sigma$ once we show that $\sigma$ is $T_{n}+f_n$ mod $\varphi_{\gamma}$. Indeed, because $\beta\notin\phibar(T_{n+1})$ and $T_{n+1}$ is closed under $\varphi$, we have $\beta\neq\gamma_h$. As $\gamma_h\in\phibar(T_h)$ and $T_{n+1}$ is elementary under $\varphi$ by (A1), we have $\gamma\neq\gamma_h$ by (2). So both $\gamma$ and $\beta$ are different from $\gamma_h$ . Recall that we have $\beta\neq\delta_h$ in the current case. So if $\gamma\neq\delta_h$, then $T_{h}$ has the same $(\gamma_{h},\delta_{h})$-exit vertex under $\sigma$ as $\varphi_{\gamma}$, and thus $\sigma$ is $T_{n}+f_n$ mod $\varphi_{\gamma}$ by (A2.5) for $T_n$.  If $\gamma=\delta_h$, then all the $(\gamma_h,\delta_h)$-ears of $T_h$ under $\varphi$ are contained in $G[V(T_{n+1})]$ by TAA. Thus by the construction of $\varphi_{\gamma}$ and $\sigma$, they are colored the same under $\sigma$ as under $\varphi$. Since $T_h$ only has one $(\gamma_h,\delta_h)$-exit path under $\varphi$ by (A2),  $T_{h}$ still has the same $(\gamma_{h},\delta_{h})$-exit vertex under $\sigma$ as under $\varphi_{\gamma}$, and thus (5) is again established by (A2.5) for $T_n$.

	Observe that none of $a_1,a_2,a_3$ is contained in $T'$, for otherwise, let $a_i \in V(T')$ for some $i$
	with $1\le i \le 3$. By (A1) and (5), we see that $T'$ is elementary under $\sigma$. Since $\{\beta,\gamma\}\cap \overline{\sigma}(a_i) \ne \emptyset$ and $\beta \in \overline{\sigma}(b_3)$, 
	we obtain $\gamma \in \overline{\sigma}(a_i)$. Hence from TAA we see that $P_1,P_2,P_3$ and $a_1, a_2, a_3$ are all entirely contained in $G[V(T')]$, 
	which in turn implies  $V(T')$ is not 
	elementary with respect to $\sigma$, a contradiction. So each $P_i$ contains a subpath $L_i$ with $i=1,2,3$, which is an exit path of  $T'$
	under $\sigma$. So $f_1$ is not contained in $L_1$ by (4), and thus we obtain $|L_1|+|L_2|+|L_3|<|P_1|+|P_2|+|P_3|$. Moreover, we observe that $Q(T')\leq Q(T_{n+1})$.
	Thus the existence of the counterexample $(\sigma, T', \beta, \gamma, L_1,L_2,L_3)$ with $\beta\neq\gamma_n$ (see (2)) violates the minimality 
	assumption on $(\varphi, T_{n+1}, \alpha, \beta, P_1,P_2,P_3)$.
	
{\flushleft {\bf Case II:}} $\beta=\delta_h$.

In this case, we have $h\geq 1$. Since $T_{n+1}$ is elementary (by (A1)) and closed under $\varphi$, we see that $\delta_h\notin\phibar(T_{n+1})$. So $Q(T_{n+1})=1$.  In this case we will get a contradiction to the minimality of $Q(T_{n+1})$. For this purpose, we let $x_e$ be the first vertex in $T_{n+1}$ along $\prec$, and pick a color $\gamma\in\phibar(x_e)$ with $\gamma\neq\gamma_h$; we have such a choice because $|\phibar(x_e)|\geq 2$ as $e$ is uncolored.  By (1) and (A2) on $T_n$ under $\varphi_{\gamma}$, one of $P_2$ and $P_3$, say $P_j$ with $j\in\{2,3\}$, belongs to a $(\gamma,\beta)$-path $P$ which is disjoint from $T_n$. Let $\eta\in\phibar(b_j)$ and consider $\varphi_{\eta}=\varphi/(G-T_{n+1},\alpha,\eta)$. Since $\eta\neq\gamma_h$ following (2) and the elementariness of $T_{n+1}$ under $\varphi$ by (A1), in view of (1), we have 

(6) $\varphi_{\eta}$ is $T_{n+1}$ mod $\varphi$, and therefore $T_{n+1}$ is an ETT with $n$ rungs under $\varphi_{\eta}$.

Observe that $P_j$ is an $(\eta,\beta)$-path under $\varphi_{\eta}$ with $\eta\in\overline{\varphi}_{\eta}(b_j)$. As in the previous case, we let $\sigma=\varphi_{\eta}/P_j$ and $T'$ be a closure of $T_{n+1}(b_j)$ under $\sigma$. We claim that 

(7) $T_h$ has the same $(\gamma_h,\delta_h)$-exit vertex under $\sigma$ as $\varphi_{\eta}$, and therefore $\sigma$ is  $T_{n}+f_n$ mod $\varphi_{\eta}$ (by (A2.5) on $T_n$).

Note that following (7), $T'$ is elementary by (A1) and (6), and therefore, as before, neither $a_1$ nor $a_j$ is in $T'$ (see the paragraph before Case II). So under $\sigma$, $P_1$ and $P_j$ guarantee two $(\beta,\eta)$-paths intersecting $T'$, and they provide three $(\beta,\eta)$-exit paths $L_1, L_2$ and $L_3$ of $T'$, where $L_1$, $L_2$, and $L_3$ produce at least two $(\beta,\eta)$-exit vertices in $T'$ not before $b_j$ along $\prec$ of $T'$ (one of them might be $b_j$ itself). Since $\delta_h=\beta\in\overline{\sigma}(b_j)$,  $(\sigma, T', \beta, \eta, L_1,L_2,L_3)$ provides a counterexample with $Q(T')=0$, a contradiction.

If \(T_{n+1}\) is RE finished, all the
\((\gamma_h,\delta_h)\)-ears of \(T_h\) under
\(\varphi_\eta\) are contained in \(G[V(T_{n+1})]\)
and are unchanged by the switch on \(P_j\), so (7) holds.
Thus assume otherwise, and hence \(\alpha\ne\gamma_h\). We now prove (7) to finish the proof of Lemma~\ref{A2} by assuming otherwise that $T_h$ has a different $(\gamma_h,\delta_h)$-exit vertex $v_1$ under $\sigma$ from the $(\gamma_h,\delta_h)$-exit vertex $v_2$ under $\varphi_{\eta}$. Note that $\sigma$ is $(T_n, D_n,\varphi_{\eta})$-stable as $P_j$ is disjoint from $T_n$ (see (2)). So $\sigma$ is $T_h$ mod $\varphi$ and $T_h$ has exactly one $(\gamma_h,\delta_h)$-exit vertex $v_1$ under $\sigma$ by (A2) for $T_h$ with $h\leq n$. We first assume that $v_1\prec v_2$ along $\prec$ of $T_{n+1}$. In this case, under $\varphi_{\eta}$, $P_j$ intersects both the $(\gamma_h,\delta_h)$-exit path $P_{ex}$ and a $(\gamma_h,\delta_h)$-ear $D$ with a root $v_1$ of $T_h$. Since $\gamma\neq\gamma_h$ and $\eta\neq\gamma_h$ (see the sentence before (6)), $P_{ex}$ is still a $(\gamma_h,\delta_h)$-exit path and $D$ is still a $(\gamma_h,\delta_h)$-ear  of $T_{h}$ under $\varphi_{\gamma}$. Let $n'$ be the smallest index such that $v_2\in T_{n'}$ and $T_{n'}$ is $\RE$ finished under $\varphi_{\gamma}$. Then $n'\leq n$ and $v_2\notin T_{h'}$, where $h'$ is the initial index of $T_{n'}$. Pick \(\rho\in\bar\varphi_\gamma(v_2)\), and let
\(\omega=\varphi_\gamma/(G-T_h,\gamma_h,\rho)\) if
\(\rho\ne\gamma_h\), and \(\omega=\varphi_\gamma\) otherwise.
By elementariness,
\(\{\gamma_h,\rho\}\cap\{\gamma,\delta_h\}=\emptyset\).
Both interchanged colors are closed for \(T_h\), so \(\omega\)
is stable on \(T_{n'}\), and \(T_{n'}\) remains RE finished
by (A2). The path \(P\) is unchanged, while \(P_{\rm ex}\)
and \(D\) become a \((\rho,\delta_h)\)-exit path and ear of
\(T_h\). Since \(\rho\in\bar\omega(v_2)\), (A2) on \(T_{n'}\)
gives a \((\rho,\delta_h)\)-ear \(D'\) containing \(D\).
Now Lemma~\ref{A3.0} under \(\omega\), with \(T_{n'}\) in place of
\(T_{n+1}\), \(v_2,v_1\) in place of \(x,y\), and
\(\gamma,\rho,\delta_h\) in place of \(\alpha,\gamma,\delta\),
gives a contradiction.


So we may assume $v_2\prec v_1$. Now in this case, under $\sigma$, we may assume $P_j$ intersects both the $(\gamma_h,\delta_h)$-exit path $P_{ex}$ of $T_h$ with exit $v_1$ and a $(\gamma_h,\delta_h)$-ear $D$ with root $v_2$. Since $T_{n+1}$ is elementary under $\varphi$, $\delta_h=\beta\notin\phibar(T_{n+1})$, and $\gamma,\eta\neq\gamma_h$ (see the sentence before (6)), we observe that:

(8) $\sigma$ is $(T_h, D_h, \varphi)$-stable, $T_{n+1}$ is elementary under $\sigma$, $\sigma(P)=\{\gamma,\eta,\delta_h\}$, all the $\eta$-colored edges of $P$ are not in $G[V(T_{n+1})]$, all the $\gamma$-colored edges of $P$ are in $G[V(T_{n+1})]$, $\delta_h\in\overline{\sigma}(b_j)$, two ends of $P$ miss either $\delta_h$ or $\eta$, and exactly one boundary edge of $T_{n+1}$ is colored by $\eta$ which is $f_j$  under $\sigma$.

Let $n'$ be the smallest index such that $v_1\in T_{n'}$ and $T_{n'}$ is $\RE$ finished, and let $h'$ be the initial index of $T_{n'}$. So $n'\leq h$, and $v_1\notin T_{h'}$. Let $\alpha'\in\phibar(v_1)=\overline{\sigma}(v_1)$ and $\tau\in\phibar(x_e)$ with $\tau\neq\alpha'$ and $\tau\neq\gamma_{h'}$; we have such choices because $T_{n+1}$ is elementary under $\varphi$, $v_1$ is not the first vertex in $T_{n+1}$ along $\prec$, and $|\phibar(x_e)|\geq 2$.  So $\alpha',\tau\notin \{\delta_h, \eta\}$ and $\alpha'\neq\gamma$ by the choice of $\alpha',\tau$, (2), $\delta_h\notin\phibar(T_{n+1})$, and the elementariness of $T_{n+1}$ under $\varphi$.  Moreover, under $\sigma$ and $\varphi$, $T_{h}$ is closed for both $\alpha'$ and $\gamma_h$, and $T_{n+1}$ and $T_h$ are closed for both $\gamma$ and $\tau$ ($\gamma=\tau$ may happen). Let $\mu=\sigma/(G-T_{h},\gamma,\tau)$.  Let $\mu_1=\mu/(G-T_{h},\alpha',\gamma_h)$ if $\tau\neq\gamma_h$, and $\mu_1=\mu/(G-T_{h},\alpha',\gamma)$ otherwise.  Furthermore, because $P$ is disjoint from $T_n$, following (8), we have the following under $\mu_1$: 

(9) $\mu_1$ is $(T_{n'}, D_{n'}, \varphi)$-stable, $T_{n+1}$ is elementary under $\mu_1$, $P_{ex}$ is an $(\alpha',\delta_h)$-exit path of both $T_{n'}$ and $T_h$ with $\alpha'\in\overline{\mu}_1(v_1)$, $D$ is an $(\alpha',\delta_h)$-ear of $T_h$, $\mu_1(P)=\{\tau,\eta,\delta_h\}$, all the $\eta$-colored edges of $P$ are not in $G[V(T_{n+1})]$, all the $\tau$-colored edges of $P$ are in $G[V(T_{n+1})]$, $\delta_h\in\overline{\mu}_1(b_j)$, two ends of $P$ miss either $\delta_h$ or $\eta$, exactly one boundary edge of $T_{n+1}$ is colored by $\eta$ which is $f_j$  under $\mu_1$, and $T_{n+1}$ is closed for $\tau$ under $\mu_1$.

Because $T_{n+1}$ is elementary under $\mu_1$, exactly one boundary edge of $T_{n+1}$ is colored by $\eta$ which is $f_j$  under $\mu_1$, and $T_{n+1}$ is closed for $\tau$ under $\mu_1$, we see that there is exactly one $(\tau,\eta)$-path $P_{x_e}(\eta,\tau,\mu_1)$ intersecting $T_{n+1}$ under $\mu_1$ with $(\tau,\eta)$-exit $b_j$. Moreover, the subpath of $P_{x_e}(\eta,\tau,\mu_1)$ from $x_e$ to $b_j$ is contained in $G[V(T_{n+1})]$, and all the other $(\tau,\eta)$-chains intersecting $T_{n+1}$ are cycles also contained in $G[V(T_{n+1})]$.  So following (9), all the $\eta$-colored edges of $P$ are contained either in $(\tau,\eta)$-chains not intersecting $T_{n+1}$ or in $P_{x_e}(\eta,\tau,\mu_1)$. Consequently, following (9), because all the $\tau$-colored edges of $P$ are in $G[V(T_{n+1})]$ and no $\eta$-colored edges of $P$ are in $G[V(T_{n+1})]$, the coloring $\mu_2$ obtained from $\mu_1$ by switching all the $\eta$-colored edges of $P$ not in $P_{x_e}(\eta,\tau,\mu_1)$ to $\tau$ by Kempe changes satisfies the following:

\begin{itemize}
	\item [(o)] $\mu_2$ is $(T_{n'}, D_{n'}, \varphi)$-stable, $P_{ex}$ is an $(\alpha',\delta_h)$-exit path of both $T_{n'}$ and $T_h$ with $\alpha'\in\overline{\mu}_2(v_1)$, and $D$ is an $(\alpha',\delta_h)$-ear of $T_h$ under $\mu_2$.
	\item [(i)] $P$ shares edges with both $P_{ex}$ and $D$, and $P$ is vertex disjoint from $T_{n'}$.
\item [(ii)] $\mu_2(P)=\{\tau,\eta,\delta_h\}$
\item [(iii)] $\delta_h\in\overline{\mu}_2(b_j)$, $\overline{\mu}_2(u)\cap\{\eta,\tau,\delta_h\}\neq\emptyset$, and $\overline{\mu}_2(v)\cap\{\eta,\tau,\delta_h\}\neq\emptyset$, where $u$ and $v$ are the two ends of $P$.
\item [(iv)] All the $\eta$-colored edges of $P$ belong to $P_{x_e}(\eta,\tau,\mu_2)=P_{x_e}(\eta,\tau,\mu_1)$ containing the vertex $b_j\in V(P)$ with one end $x_e\in T_{n'}(v_1)$ missing $\tau$ and the other end $r'$ outside of $T_{n'}$, and all of these $\eta$-colored edges of $P$ are in the subpath from $b_j$ to $r'$ of $P_{x_e}(\eta,\tau,\mu_2)$, and this subpath does not intersect $T_{n'}$.
\end{itemize}

Because $T_{n'}$ is $\RE$ finished, following (A2) on $T_{n'}$ and (o), we see that $T_{n'}$ is an ETT with $(n'-1)$ rungs under $\mu_2$, and $v_2$ is a root of an $(\alpha',\delta_h)$-ear $D'$ of $T_{n'}$. Since $v_2$ is a root of the $(\alpha',\delta_h)$-ear $D$ of $T_h$ with $h\geq n'$, $D$ is a subpath of $D'$, and $P$ shares edges with both $P_{ex}$ and $D'$ by (i). Moreover, since $\tau\in\phibar(x_e)$, we have $\tau\neq\delta_{h'}$. So $\tau\notin S_{h'}$. Thus with (o), (i), (ii), (iii) and (iv), we have a contradiction to the later introduced technical Lemma~\ref{A3.2} under $\mu_2$ with $T_{n'}$ ($n'\leq n$) in place of $T_{n+1}$, $D'$ in place of $D$, $\alpha'$ in place of $\alpha$, $\delta_h$ in place of $\delta$, $x_e$ in place of $r$, $v_1$ in place of $x$, $v_2$ in place of $y$, and $b_j$ in place of $w$.
\end{proof}
\begin{REM}\label{gap}
	Because almost all properties remain valid when we switch $\alpha$ and $\gamma$ for edges in $G - V(T_{n+1})$, for two colors $\alpha$ and $\gamma$ missing at vertices of $T_{n+1}$, it is tempting to combine Lemmas~\ref{A3}–\ref{A3.5} into a single lemma and apply it directly in the proof of Lemma~\ref{A2}. Indeed, Lemma 3.2 in~\cite{short} is analogous to our Lemmas~\ref{A3}–\ref{A3.5}. However, their proof is incomplete due to the absence of the key techniques developed in Lemma~\ref{A3.2}.
	
	Specifically, in the second paragraph after Claim 1 of their Lemma 3.2, they define $\varphi^*=\varphi/(G-U,\alpha_1,\alpha)$. They then claim that $R$, an $(\alpha_1,\alpha_2)$-chain under $\varphi$ which is an entire $(\alpha_1,\alpha_2)$-path (they define $(\alpha,\beta)$-chain as an entire $(\alpha,\beta)$-component that is a path), becomes an entire $(\alpha,\alpha_2)$-path under $\varphi^*$. This claim does not hold when $R$ and $U$ share an endpoint $v$ of $R$ with $\alpha_1\in\phibar(v)$ and $\alpha\notin\phibar(v)$ (a common situation since $U$ is elementary in their Lemma 3.2). In this case, $v$ is no longer an endpoint of the $(\alpha,\alpha_2)$-path containing $R$ under $\varphi^*=\varphi/(G-U,\alpha_1,\alpha)$. Consequently, their assumption that $w\in V(T_n(x))$ fails (see paragraph 6 after Claim 1 in their Lemma 3.2, which corresponds to (3) in the proof of our Lemma~\ref{A3}).
	
	Furthermore, when applying their Lemma 3.2 in the proof of their Lemma 3.3 (which is analogous to our Lemma~\ref{A2} together with Claim~\ref{bstep1}), the argument fails when applied to $P_x$ in paragraph 6 of their Lemma 3.3. This corresponds precisely to Case II in the proof of our Lemma~\ref{A2}, where Lemma~\ref{A3.2} is essential to resolve the issue. Consequently, our proof of Lemma~\ref{A2} must be more technical in order to properly apply Lemma~\ref{A3.2}.
	
	In fact, Lemma~\ref{A3.2} is essential to our approach. Its proof requires extending both the techniques and the algorithm (see the paragraph after $(f_i)$ in the proof of Lemma~\ref{A3.1}) to their full strength. In contrast, they use this algorithm only once in their proof of Lemma 3.2 (see the algorithm introduced before their Claim 2), which explains why their paper is shorter but also why their argument is incomplete.
\end{REM}

\begin{LEM}(A2.5)\label{stable}
	Let $n$ be a nonnegative integer. Suppose (A1) and (A2) hold for  all  ETTs with at most $n$ rungs, and $(A2.5)$ holds for  all  ETTs with at most $(n-1)$ rungs.  Let $T_{n+1}$ be a closed ETT with ladder $T_0\subset \dots \subset T_n \subset T_{n+1}$. Suppose $T_{n+1}$ could be extended further to $T_{n+2}$ with $F_{n+1}=\{f_{n+1}\}$, $S_{n+1}=\{\gamma_{n+1},\delta_{n+1}\}$, and initial index $h$.  If $\varphi'$ is a $(T_{n+1},D_{n+1},\varphi)$-stable coloring such that $T_{h}$ has the same $(\gamma_{h},\delta_{h})$-exit vertex  under $\varphi'$ as under $\varphi$, then $\varphi'$ is a $T_{n+1}+f_{n+1}$ mod $\varphi$ coloring.

\end{LEM}
	\begin{proof}
		Let $T_{n+1}$, $F_{n+1}$, $S_{n+1}$, $h$, and $\varphi'$ be as described above.
		We first consider the case when extension type $\Theta_{n+1}=\IE$. Then $T_{n+1}$ is $\RE$ finished, $h=n+1$, $\delta_{n+1}$ is a defective color of $T_{n+1}$, and $f_{n+1}$ belongs to a $(\gamma_{n+1},\delta_{n+1})$-ear of $T_{n+1}$ under $\varphi$. Moreover, since $T_{n+1}$ is elementary by (A1), $T_{n+1}$ has exactly one $(\gamma_{n+1},\delta_{n+1})$-exit vertex under $\varphi$ by (A2).  Since $\varphi'$ is $(T_{n+1},D_{n+1},\varphi)$-stable,  $T_{n+1}$ is closed under $\varphi'$ and $\varphi'$ is $(T_{n+1},D_{n},\varphi)$-stable. So $\varphi'$ is a $T_{n+1}$ mod $\varphi$ coloring by (A2), and therefore 
		 
		 (1) $T_{n+1}$ is a closed ETT under $\varphi'$ with the same ladder, $F_i$, $S_i$, and $\Theta_i$ for $0\leq i\leq n$, and $T_{n+1}$ is still $\RE$ finished under $\varphi'$.

		  

		Recall that $f_{n+1}$ belongs to a $(\gamma_{n+1},\delta_{n+1})$-ear of $T_{n+1}$ under $\varphi$. Since $\varphi'$ is a $(T_{n+1},D_{n+1},\varphi)$-stable coloring such that $T_{n+1}$ has the same $(\gamma_{n+1},\delta_{n+1})$-exit vertex under $\varphi'$ as $\varphi$, the connecting edge $f_{n+1}$ still belongs to  a $(\gamma_{n+1},\delta_{n+1})$-ear of $T_{n+1}$ under $\varphi'$. Since $T_{n+1}$ is still $\RE$ finished under $\varphi'$ by (1), we may still choose $F_{n+1}=\{f_{n+1}\}$ and $S_{n+1}=\{\gamma_{n+1},\delta_{n+1}\}$ to extend $T_{n+1}$ by an $\IE$ extension under $\varphi'$, so  the coloring $\varphi'$ is a $T_{n+1}+f_{n+1}$ mod $\varphi$ coloring following (1).
		
		We then consider the case when extension type $\Theta_{n+1}=\RE$ with initial index $h\leq n$. In this case, $n\geq 1$ and $h\geq 1$. Since $T_h$ is elementary by (A1),  $T_h$ has exactly one $(\gamma_h,\delta_h)$-exit vertex under $\varphi$ by (A2). By the assumption on $\varphi'$, the ETT $T_h$ has the same unique $(\gamma_h,\delta_h)$-exit vertex under $\varphi'$ as under $\varphi$.  Since $\varphi'$ is also a $(T_{n},D_{n},\varphi)$-stable coloring, we see that $\varphi'$ is a $T_n+f_n$ mod $\varphi$ coloring by (A2.5) with $T_n$ that has $(n-1)$ rungs in place of $T_{n+1}$. As $\varphi'$ is also $(T_{n+1},D_{n+1},\varphi)$-stable,
		
		(2) $T_{n+1}$ is a closed ETT under $\varphi'$ with the same ladder, $F_i$, $S_i$, and $\Theta_i$ for $0\leq i\leq n$.
		
		Furthermore, $f_{n+1}$ is still colored by $\delta_h=\delta_{n+1}$. Because $\Theta_{n+1}=\RE$, the connecting edge $f_{n+1}$ belongs to a $(\gamma_h,\delta_h)$-ear $O$ of $T_h$ such that $O$ contains a path $L$ connecting $f_{n+1}$ and $V(T_h)$ with $V(L)\subseteq V(T_{n+1})$ under $\varphi$. Since $\varphi'$ is $(T_{n+1},D_{n+1},\varphi)$-stable, $L$ remains a path contained entirely in either a $(\gamma_h,\delta_h)$-ear or a $(\gamma_h,\delta_h)$-exit path of $T_h$ under $\varphi'$. Since $T_h$ has the same unique $(\gamma_h,\delta_h)$-exit vertex under $\varphi'$ as under $\varphi$, the path $L$ must be contained entirely in a $(\gamma_h,\delta_h)$-ear $O^*$ of $T_h$ under $\varphi'$ that connects $f_{n+1}$ to $V(T_h)$ (here $O^*$ might be different from $O$). So we may still choose $F_{n+1}=\{f_{n+1}\}$ and $S_{n+1}=\{\gamma_{n+1},\delta_{n+1}\}$ to extend $T_{n+1}$ by an $\RE$ extension under $\varphi'$, and therefore  the coloring $\varphi'$ is a $T_{n+1}+f_{n+1}$ mod $\varphi$ coloring following (2).
	\end{proof}

\begin{LEM}(A3)\label{A3}
Let $n$ be a nonnegative integer. Suppose (A1), (A2), and (A2.5) hold for  all  ETTs with at most $n$ rungs, and (A3) holds for all ETTs with at most $(n-1)$ rungs.  Let $T_{n+1}$ be a closed ETT with ladder $T_0\subset \dots \subset T_n \subset T_{n+1}$. Assume $T_{n+1}$ is $\RE$ finished, $\delta$ is a defective color of $T_{n+1}$ under $\varphi$, and $\gamma\in\phibar(T_{n+1})$. Let $\alpha,\beta$ be two colors such that $\alpha\in\phibar(T_{n+1})\setminus\{\gamma\}$ and $\beta$ is arbitrary. For any $(\alpha,\beta)$-path $P$, if $P$ does not intersect $T_{n+1}$, then $T_{n+1}$ has the same $(\gamma,\delta)$-exit vertex under $\varphi/P$ as under $\varphi$.
\end{LEM}

\begin{proof}
	Let $T_{n+1}$, $\gamma$, $\delta$, $\alpha$, $\beta$, $P$, $\varphi$, and $\varphi'$ be as described above. Since $T_{n+1}$ is closed under $\varphi$, we have

	(1) $\delta\notin\phibar(T_{n+1})$.
	
	Assume otherwise that $T_{n+1}$ has a different $(\gamma,\delta)$-exit vertex under $\varphi/P$ than under $\varphi$. By (A2) for $T_{n+1}$, there is only one $(\gamma, \delta)$-path intersecting $T_{n+1}$ under $\varphi$, so $T_{n+1}$ has exactly one  $(\gamma, \delta)$-exit vertex, say $x$, contained in a $(\gamma, \delta)$-exit path $P_{ex}$. Since $P$ does not intersect $T_{n+1}$ and $T_{n+1}$ is $\RE$ finished, the coloring $\varphi'=\varphi/P$ is $(T_{n+1},D_{n},\varphi)$-stable and therefore is a $T_{n+1}$ mod $\varphi$ coloring by (A2). So $T_{n+1}$ is a closed ETT under $\varphi'$ with $n$ rungs. Therefore, by (A2) for $T_{n+1}$ under $\varphi'$, $T_{n+1}$ also has exactly one  $(\gamma, \delta)$-exit vertex under $\varphi'$, say $y$, and $y$ is different from $x$. Thus $y$ is a root of a $(\gamma,\delta)$-ear $D$ of $T_{n+1}$ under $\varphi$. Because interchanging $\alpha$ and $\beta$ along $P$ changes the $(\gamma, \delta)$-exit vertex of $T_{n+1}$, $\alpha\neq\gamma$ (by the assumption on $\alpha$), and $\alpha\neq\delta$ by (1), we have 
	
	(2) $P$ shares edges with both $D$ and $P_{ex}$, $\alpha\notin\{\gamma,\delta\}$, and $\beta\in\{\gamma,\delta\}$.
	
	 We may assume that
	
	(3) $y\prec x$ along $T_{n+1}$. Consequently, $|\phibar(T_{n+1}(x))-\phibar(x)|\geq 2$ and $|\phibar(T_{n+1}(x))|\geq 4$.
	
  Since $P$ is disjoint from $T_{n+1}$, the coloring $\varphi'$ is $(T_{n+1}, D_n,\varphi)$-stable and therefore is $T_{n+1}$ mod $\varphi$ by (A2). So $T_{n+1}$ is an $\RE$ finished closed ETT under $\varphi'$ with the same ladder, $F_i$, $S_i$, and $\Theta_i$ for $0\leq i\leq n$. As $\varphi=\varphi'/P$,  we just consider $T_{n+1}$ under $\varphi'$ as a counterexample with $(\gamma, \delta)$-exit vertex $y$ instead, so we have $y\prec x$. Consequently, since both ends of the uncolored edge $e$ have at least two missing colors, and $T_{n+1}$ is elementary by (A1), we have $|\phibar(T_{n+1}(x))-\phibar(x)|\geq 2$ and $|\phibar(T_{n+1}(x))|\geq 4$.
  
 Let $h$ be the initial index of $T_{n+1}$. We may furthermore assume that
  
  (4)  $x\notin V(T_h)$.

   If $n=0$, then by definition $h=0$ and $T_0=\emptyset$, so $x\notin V(T_h)$.  We then consider the case that $n\geq 1$ and $x\in V(T_h)$. Note that $h\geq 1$ as $n\geq 1$ and $T_h$ is also $\RE$ finished, and $y\in V(T_h)$ by (3). Moreover, by (3) there are two colors $\alpha',\gamma'\in\phibar(T_h)$. Now since $T_{n+1}$ is closed and $\RE$ finished, following (A2), by first switching colors on edges colored by $\gamma$ with $\gamma'$ outside of $T_{n+1}$, and then switching the edges colored by $\alpha'$ with $\alpha$ outside $T_{n+1}$ if $\alpha\notin\phibar(T_h)$, we have a smaller counterexample to (A3) by considering the $\RE$ finished ETT $T_h$ instead of $T_{n+1}$.
   
  
 Now with (2), (3), and (4), we have a contradiction by the following technical Lemma~\ref{A3.0}.

\end{proof}
For a path $P$ with vertices $u,v\in V(P)$, denote by $P[u,v]$ the subpath of $P$ from $u$ to $v$.
\begin{LEM}\label{A3.0}
	
	Let $n$ be a nonnegative integer. Suppose (A1), (A2), and (A2.5) hold for  all  ETTs with at most $n$ rungs. Let $T_{n+1}$ be a closed ETT with ladder $T_0\subset \dots \subset T_n \subset T_{n+1}$.  Assume $T_{n+1}$ is $\RE$ finished and has initial index $h$, $\delta$ is a defective color of $T_{n+1}$ under $\varphi$, $\gamma\in\phibar(T_{n+1})$, $x$ is the $(\gamma,\delta)$-exit vertex of $T_{n+1}$ with exit path $P_{ex}$ and $x\notin V(T_h)$. Suppose $T_{n+1}$ also has a $(\gamma,\delta)$-ear $D$ with one root $y$ such that $y\prec x$ along $T_{n+1}$.  Let $\alpha,\beta$ be two colors such that  $\alpha\in\phibar(T_{n+1})\setminus\{\gamma\}$ and $\beta$ is either $\gamma$ or $\delta$. For any   $(\alpha,\beta)$-path $P$, if $P$ does not intersect $T_{n+1}$, then $P$ does not share vertices with both $P_{ex}$ and $D$.
	
\end{LEM}	
\begin{proof}
	
		Let everything be as described in Lemma~\ref{A3.0}, and suppose the $(\alpha,\beta)$-path $P$ shares vertices with both $P_{ex}$ and $D$.  We claim that we may assume:
		
		(1) $\gamma\in\phibar(x)$, $\alpha\in\phibar(T_{n+1}(x))$, and there exists a color $\pi\in\phibar(T_{n+1}(x))$ such that $\pi\notin (\{\alpha,\gamma\}\cup S_h)$.
		
		Since $y\prec x$ along $T_{n+1}$, both ends of the uncolored edge $e$ are in $T_{n+1}(x)$. Since $\varphi$ has at least $\Delta(G)+1$ many colors, we have $|\phibar(T_{n+1}(x))|\geq 4$. Moreover, when $h\geq 1$, we have $|\phibar(T_1)|\geq 5$ by Lemma~\ref{size}. Furthermore, since $x\notin T_h$ by the assumption in the statement of the current lemma, there exist colors $\gamma'\in\phibar(x)$, $\alpha'\in\phibar(T_{n+1}(x))$, and a color $\pi\in\phibar(T_{n+1}(x))$ such that $\pi\notin (\{\alpha',\gamma'\}\cup S_h)$. Now since $T_{n+1}$ is closed and $\RE$ finished, following (A2), by switching colors on edges colored by $\alpha$ with $\alpha'$ and edges colored by $\gamma'$ with $\gamma$ outside of $T_{n+1}$ if necessary, we have as claimed.

		Our goal is to apply the following technical Lemma~\ref{A3.1} with $\gamma$ in place of $\alpha$. Now by (1), $P_{ex}$ is the entire $(\gamma,\delta)$-path. Moreover, since $T_{n+1}$ is $\RE$ finished, all the $(\gamma_h,\delta_h)$-ears intersecting $T_{h}$ when $h\geq 1$ are contained in $G[V(T_{n+1})]$. As a result, if $\delta=\delta_h$ with $h\geq 1$, then $P_{ex}$ is disjoint from any $(\gamma_h,\delta_h)$-ear intersecting $T_{h}$ . Assume the two ends of $P$ are $u$ and $v$. Since $P$ shares vertices with both $P_{ex}$ and $D$,  without loss of generality, we may assume the first shared vertex between $P_{ex}$ and $P$ along $P_{ex}$ from $x$ is $u'$, the first shared vertex between $D$ and $P$ along $D$ from $y$ is $v'$, and $u'$ is closer to $u$ than $v'$ along $P$.	So now 	$P_1=P_{ex}[x,u']\cup P[u',u]$ and $P_2=D[y,v']\cup P[v',v]$ are two edge disjoint paths from $T_{n+1}(x)$ to $\{u,v\}$ satisfying all conditions $(i)$ and $(ii)$ of 	Lemma~\ref{A3.1}. Thus applying Lemma~\ref{A3.1}, we have both $(II)$ and $(III)$. However, on one side, since $j\leq n+1$ where $j$ is the smallest index with $x\in T_j$, by $(II)$ and (A1), the resulting $T$ is elementary under $\mu$. On the other hand, because $u$ and $v$ are different vertices, $T$  is not elementary by $(III)$. This contradiction finishes the proof of Lemma~\ref{A3.0}.
\end{proof}
\begin{LEM}\label{A3.1}
	Let $n$ be a nonnegative integer. Suppose (A1), (A2), and (A2.5) hold for  all  ETTs with at most $n$ rungs. Let $T_{n+1}$ be a closed ETT with ladder $T_0\subset \dots \subset T_n \subset T_{n+1}$.  Assume $T_{n+1}$ has initial index $h$, $\delta$ is a defective color of $T_{n+1}$ under $\varphi$, and $\alpha\in\phibar(x)$ where $x$ is the $(\alpha,\delta)$-exit vertex of $T_{n+1}$ with exit path $P_{ex}$. Moreover, if $\delta=\delta_h$ when $h\geq 1$, then $P_{ex}$ does not intersect any $(\gamma_h,\delta_h)$-ear of $T_h$ under $\varphi$. Let $\pi$ be a color with $\pi\in\phibar(z)$ and $z\in V(T_{n+1}(x))$ such that $\pi\notin (\{\alpha\}\cup S_h)$ ($S_h=\emptyset$ when $h=0$). Let $P_1$ and $P_2$ be two edge disjoint paths from $T_{n+1}(x)$ to  a vertex set $V$ outside of $T_{n+1}$   satisfying all of the following conditions:
	\begin{itemize}
		\item [(i)] $\varphi(P_i)\subseteq(\phibar(T_{n+1}(x))\cup\{\delta\})$, $\alpha\notin \varphi(P_i\cap G[V(T_{n+1})])$ and $\pi\notin\varphi(P_i-G[V(T_{n+1})])$ for $i\in\{1,2\}$.
		\item [(ii)] $\phibar(v)\cap(\phibar(T_{n+1}(x))\cup\{\delta\})\neq\emptyset$ for each $v\in V$.
	
	\end{itemize}
 If $x\notin V(T_h)$ and $j$ is the smallest index where $x\in T_j$, then there exists a $(T_{j}(x)-x,D_{j-1},\varphi)$-stable coloring $\mu$ and a closure $T$ of $T_j(x)$ under $\mu$ such that:
 
 \begin{itemize}
 	\item [(I)] $(V(P_1)\cup V(P_2))\subseteq V(T)$.
 	\item [(II)] $T$ is an ETT with the same $(j-1)$ rungs under $\mu$ as $T_j$ under $\varphi$.
 	\item [(III)] If $P_1$ and $P_2$ have different ends in $V$, then $T$ is not elementary.
 \end{itemize}
\end{LEM}	

\begin{proof}
	Let everything be as described in Lemma~\ref{A3.1}, and suppose $x\notin V(T_h)$ with $j$ being the smallest index where $x\in T_j$. So $j-1\geq h$. Note that $P_1$ and $P_2$ together contain at least two $\delta$-colored boundary edges of $T_{n+1}$ under $\varphi$. Moreover, $\delta\notin\phibar(T_{n+1})$ as $\delta$ is a defective color of $T_{n+1}$ and $T_{n+1}$ is closed under $\varphi$. So by assumptions of the current lemma, $\gamma_h\in\phibar(T_h)$, and the elementariness of $T_{n+1}$ by (A1), we have
	
	(1) $x\notin V(T_h)$, colors $\pi,\alpha$, $\delta$, and $\gamma_h$ (when $h\geq 1$) are all different, and $\pi\notin S_h$.
	
	We will recursively construct a series of colorings $\varphi_i$ with $i\geq 1$ until we reach the desired $\mu$. We first let $\varphi_1'=\varphi/P_x(\alpha,\delta,\varphi)$, and then construct $\varphi_1$ from $\varphi_1'$ by switching all the $\alpha$-colored edges of $P_1\cup P_2$ not in $P_z(\pi,\alpha,\varphi_1')$ to $\pi$ by Kempe changes. Let $T^1$ be a closure of $T_j(x)$ under $\varphi_1$. We claim that $\varphi_1$ satisfies the following:
	
	\begin{itemize}
		
		\item [$(a_1)$] $\varphi_1$ is $(T_j(x)-x, D_{j-1},\varphi)$-stable, and if $h\geq 1$, then all the $(\gamma_h,\delta_h)$-ears of $T_h$ under $\varphi$ are the same as under $\varphi_1$.
		
		\item [$(b_1)$] $T^1$ is an elementary ETT with the same $(j-1)$ rungs under $\varphi_1$ as $T_j$ under $\varphi$.
		
		\item [$(c_1)$] $\varphi_1(P_1\cup P_2)\subseteq\phibar_1(T^1(x))\cup\{\alpha\}$.

		\item [$(d_1)$] $\phibar_1(v)\cap(\phibar_1(T^1(x))\cup\{\alpha\})\neq\emptyset$ for each vertex $v\in V$.
		
		\item [$(e_1)$] All the $\alpha$-colored edges of $P_1\cup P_2$ under $\varphi_1$ belong to $P_z(\pi,\alpha,\varphi_1)$ after the vertex $x$ along the direction from $z$.
		
		\item [$(f_1)$] If  $(V(P_1)\cup V(P_2))- V(T^1)\neq\emptyset$, then $\alpha\notin\phibar_1(T^1)$ and $T^1$ has a $(\pi,\alpha)$-exit vertex $x_2$ in $P_z(\pi,\alpha,\varphi_1)$ with $x\prec x_2$ along $T^1$, and the corresponding $(\pi,\alpha)$-exit path does not share a vertex with $T_{n+1}$.
	\end{itemize}

By (A1) and (A2) on $T_h$ under $\varphi$, we have

(2) if $h\geq 1$, under $\varphi$ all the $\delta_h=\delta_n$-colored defective edges of $T_h$ are contained in some $(\gamma_h,\delta_h)$-ear of $T_h$ except for the $(\gamma_h,\delta_h)$-exit edge. 

Clearly $\varphi_1'=\varphi/P_x(\alpha,\delta,\varphi)$ is $(T_j(x)-x, D_{j-1},\varphi)$-stable, because $P_x(\alpha,\delta,\varphi)\cap T_{n+1}=\{x\}$. As $T_{n+1}$ is closed for both $\alpha$ and $\pi$ under $\varphi$, all the $(\pi,\alpha)$-chains intersecting $T_{n+1}$ under $\varphi_1'$ are contained in $G[V(T_{n+1})]$ except for $P_z(\pi,\alpha,\varphi_1')$ that contains $x$, which is the only $(\pi,\alpha)$-exit of $T_{n+1}$ under $\varphi_1'$, and  all vertices in $P_z(\pi,\alpha,\varphi_1')$ from $z$ to $x$ are contained in $T_{n+1}$.  Moreover, by (i) in the statement of the current lemma, (1), and  $P_x(\alpha,\delta,\varphi)\cap T_{n+1}=\{x\}$, we see that $\pi$ is not used on edges of $P_1\cup P_2-G[V(T_{n+1})]$ and $\alpha$ is not used on edges of $(P_1\cup P_2)\cap G[V(T_{n+1})]$ under $\varphi_1'$. Therefore, since all vertices in $P_z(\pi,\alpha,\varphi_1')$ from $z$ to $x$ are contained in $T_{n+1}$, none of the $\alpha$-colored edges of $P_1\cup P_2$ under $\varphi_1'$ is in the subpath of $P_z(\pi,\alpha,\varphi_1')$ from $z$ to $x$. Consequently, we have $(e_1)$ because we switch all the $\alpha$-colored edges of $P_1\cup P_2$ not in $P_z(\pi,\alpha,\varphi_1')$ to $\pi$ by Kempe changes, and none of these $(\pi,\alpha)$-chains  intersects $T_{n+1}$ under $\varphi_1'$. It also follows that

(3) $\varphi_1$ is also $(T_j(x)-x,D_{j-1},\varphi)$-stable.

To prove $(a_1)$, by (3), we may assume $h\geq 1$. Since $\alpha\neq\gamma_h$, $\delta\neq\gamma_h$ and $\pi\notin S_h$ following (1), by (3) and the construction of $\varphi_1$, we have $(a_1)$ if $\delta\neq\delta_h$ and $\alpha\neq\delta_h$. If $\alpha=\delta_h$, then all the $(\gamma_h,\delta_h)$-ears of $T_h$ under $\varphi$ are contained in $G[V(T_{n+1})]$ by TAA, so still we have $(a_1)$ from (3) as $P_x(\alpha,\delta,\varphi)\cap T_{n+1}=\{x\}$ and none of the $(\pi,\alpha)$-chains we switched by Kempe changes from $\varphi_1'$ to $\varphi_1$ intersects $T_{n+1}$ (see the paragraph above (3)). So we assume $\delta=\delta_h$. In this case, we have $\alpha\neq\delta_h$, and by the statement of the current lemma (if $\delta=\delta_h$ when $h\geq 1$, then $P_{ex}$ does not intersect any $(\gamma_h,\delta_h)$-ear of $T_h$ under $\varphi$), we see that all the $(\gamma_h,\delta_h)$-ears of $T_h$ under $\varphi$ are colored the same as under $\varphi_1'$. Since  $\alpha\neq\delta_h$ and $\pi\notin S_h$, we see that all the $(\gamma_h,\delta_h)$-ears of $T_h$ under $\varphi_{1}$ are the same as under $\varphi_{1}'$. Thus we have $(a_1)$ following (3).


By $(a_1)$ and (2), we see that $T_h$ has the same $(\gamma_h,\delta_h)$-exit vertex under $\varphi_1$ as under $\varphi$ if $h\geq 1$.  Since $\varphi_1$ is $(T_j(x)-x, D_{j-1},\varphi)$-stable and $x\notin T_{j-1}$, we see that $\varphi_1$ is also $(T_{j-1}, D_{j-1},\varphi)$-stable, and thus is $T_{j-1}+f_{j-1}$ mod $\varphi$ following (A2.5), and therefore $T_j(x)$ is an ETT with $(j-1)$ rungs under $\varphi_1$ by Corollary~\ref{vstable}. Since $T^1$ is a closure of $T_j(x)$ under $\varphi_1$ and $j-1\leq n$, we have $(b_1)$ by (A1) on $T^1$.

$(c_1)$ and $(d_1)$ follow from (1), assumptions $(i)$ and $(ii)$, and the construction of $\varphi_1$.

To prove $(f_1)$, we assume that $(V(P_1)\cup V(P_2))- V(T^1)\neq\emptyset$. If some vertex $v$ in $V$ is contained in $V(T^1)$, then $\alpha\in\phibar_1(v)$ by the fact that $V\cap V(T_{j}(x))=\emptyset$, $(a_1)$, $(b_1)$ and $(d_1)$. Since both $P_1$ and $P_2$ have an end in $V(T_{n+1}(x))$,  by $(c_1)$ and TAA, we have $(V(P_1)\cup V(P_2))\subseteq V(T^1)$, a contradiction to our recent assumption. So both $P_1$ and $P_2$ have an end outside of $T^1$, and therefore the edge disjoint paths $P_1$ and $P_2$ give us (at least) two $\alpha$-colored boundary edges $f$ and $g$ of $T^1$ following $(c_1)$. Moreover, $\alpha\notin\phibar_1(T^1)$. So by $(e_1)$, one of $f$ and $g$ is incident to a vertex in $T^1$ other than $x$ and it is after $x$ along $P_z(\pi,\alpha,\varphi_1)$ from the direction starting at $z$. Since $\alpha\notin\phibar_1(T^1)$ and $T^1$ is elementary by $(b_1)$, we see that $T^1$ has a $(\pi,\alpha)$-exit vertex $x_2$ in $P_z(\pi,\alpha,\varphi_1)$ such that $x_2$ is after $x$ along  $P_z(\pi,\alpha,\varphi_1)$ from the direction starting at $z$. Because $P_z(\pi,\alpha,\varphi_1)$ and $T_{n+1}$ only share vertices on the subpath of $P_z(\pi,\alpha,\varphi_1)$ from $z$ to $x$ (which is also $P_z(\pi,\alpha,\varphi)$) by the construction of $\varphi_1$ ($T_{n+1}$ is closed for $\pi$ and it only has one boundary edge colored by $\alpha$ which is the edge incident to $x$ under $\varphi_1$), we see that $x_2\notin V(T_{n+1})$ and the corresponding $(\pi,\alpha)$-exit path of $T^1$ under $\varphi_1$ does not share a vertex with $T_{n+1}$. So $x\prec x_2$ along $T^1$ and $(f_1)$ is established. 

Now we have as claimed. If $(V(P_1)\cup V(P_2))\subseteq V(T^1)$, we stop with $\mu=\varphi_{1}$ and $T=T^1$. We then have (I), (II), and $\mu$ is $(T_j(x)-x,D_{j-1},\varphi)$-stable by $(V(P_1)\cup V(P_2))\subseteq V(T^1)$, $(b_1)$ and $(a_1)$. If  $P_1$ and $P_2$ have different ends in $V$, then both ends of $P_1$ and $P_2$ in $V$ are contained in $T^1$, and therefore $T^1$ is not elementary by $(d_1)$ and the fact that $V\cap V(T_j(x))=\emptyset$. Thus we have (III). Note that actually $P_1$ and $P_2$ can't have different ends in this case, as otherwise we have a contradiction to $(b_1)$. If $(V(P_1)\cup V(P_2))- V(T^1)\neq\emptyset$, we will in the following fashion recursively construct $\varphi_i$ with $i\geq 2$ from $\varphi_{i-1}$ and $x_i$ with an arbitrary color $\alpha_i\in\phibar_{i-1}(x_i)$,  and have the following assertions, as above, with $T^i$ being a closure of $T^{i-1}(x_i)$ under $\varphi_i$. Note that $\varphi_1$ also satisfies the claim below with $\alpha_1=\alpha$, $x_1=x$, $T^0=T_{j}$, and $\varphi_0=\varphi$.

\begin{itemize}
	
	\item [$(a_i)$] $\varphi_i$ is $(T^{i-1}(x_i)-x_i, D_{j-1},\varphi_{i-1})$-stable, and if $h\geq 1$, then all the $(\gamma_h,\delta_h)$-ears of $T_h$ under $\varphi_{i-1}$ are the same as under $\varphi_i$.
	
	\item [$(b_i)$] $T^i$ is an elementary ETT with the same $(j-1)$ rungs under $\varphi_i$ as $T_j$ under $\varphi$.
	
	\item [$(c_i)$] $\varphi_i(P_1\cup P_2)\subseteq\phibar_i(T^i(x_i))\cup\{\alpha_i\}$.

	\item [$(d_i)$] $\phibar_i(v)\cap(\phibar_i(T^i(x_{i}))\cup\{\alpha_i\})\neq\emptyset$ for each vertex $v\in V$.
	
	\item [$(e_i)$] All the $\alpha_i$-colored edges of $P_1\cup P_2$ under $\varphi_i$ belong to $P_z(\pi,\alpha_i,\varphi_i)$ after the vertex $x_i$ along the direction from $z$.
	
	\item [$(f_i)$] If  $(V(P_1)\cup V(P_2))- V(T^i)\neq\emptyset$, then $\alpha_i\notin\phibar_i(T^i)$ and $T^i$ has a $(\pi,\alpha_i)$-exit vertex $x_{i+1}$ in $P_z(\pi,\alpha_i,\varphi_i)$ with $x_i\prec x_{i+1}$ along $T^i$, and the corresponding $(\pi,\alpha_i)$-exit path does not share a vertex with $T^{i-1}$.
\end{itemize}

In the same induction, we also have
\(V\cap V(T^i(x_i))=\emptyset\). This holds when \(i=1\).
If a vertex of \(V\) belongs to \(T^i\), then by \((d_i)\)
and elementariness it misses \(\alpha_i\), and \((c_i)\) and
TAA imply \(V(P_1)\cup V(P_2)\subseteq V(T^i)\).
Thus, if the recursion continues, \(V\cap V(T^i)=\emptyset\),
which gives the assertion at the next stage.

Here is the detailed construction. Suppose $\varphi_{i-1}$ ($i\geq 2$) is constructed with the above claims $(a_{i-1})-(f_{i-1})$.  We first let $\varphi'=\varphi_{i-1}/(G-T^{i-1},\alpha_i,\pi)$, and then $\varphi_i'=\varphi'/P_{x_i}(\alpha_i,\alpha_{i-1},\varphi')$. Finally,   we construct $\varphi_i$ from $\varphi_i'$ by switching all the $\alpha_i$-colored edges of $P_1\cup P_2$ not in $P_z(\pi,\alpha_i,\varphi_i')$ to $\pi$ by Kempe changes. Let $T^i$ be a closure of $T^{i-1}(x_i)$ under $\varphi_i$. 

We will then prove the above claim inductively. By $(f_{k})$ for all $1\leq k\leq i-1$ and (1), we see that $x\prec x_2\prec ...\prec x_i$ along $T^{i-1}$, and therefore $x_i\notin T_h\subseteq T_{j-1}$ by (1) and $(b_{i-1})$. Following (1), $(a_{i-1})$, $(b_{i-1})$, $(c_{i-1})$, and $(d_{i-1})$, we have 

(4)  $\pi,\alpha_{i}$, $\gamma_h$ (when $h\geq 1$), and $\alpha_{i-1}$ are pairwise distinct, and $\pi\notin  S_h$.

Clearly $\varphi'$ is $(T^{i-1},D_{j-1},\varphi_{i-1})$-stable as $T^{i-1}$ is closed for both $\alpha_i$ and $\pi$ under $\varphi_{i-1}$ and $\varphi'=\varphi_{i-1}/(G-T^{i-1},\alpha_i,\pi)$. By $(f_{i-1})$, we see that $P_{x_i}(\alpha_i,\alpha_{i-1},\varphi')\cap T^{i-1}=\{x_i\}$. So $\varphi_i'$ is  $(T^{i-1}(x_i)-x_i,D_{j-1},\varphi_{i-1})$-stable. As $T^{i-1}$ is closed for both $\alpha_i$ and $\pi$ under $\varphi_{i-1}$, all the $(\pi,\alpha_i)$-chains intersecting $T^{i-1}$ under $\varphi_i'$ are contained in $G[V(T^{i-1})]$ except for $P_z(\pi,\alpha_i,\varphi_i')$ that contains $x_i$, which is the only $(\pi,\alpha_i)$-exit of $T^{i-1}$ under $\varphi_i'$, and all vertices in $P_z(\pi,\alpha_i,\varphi_i')$ from $z$ to $x_i$ are contained in $T^{i-1}$.  Moreover, since $\alpha_i\neq\alpha_{i-1}$ by (4), $x_{i-1}\prec x_{i}$ by $(f_{i-1})$, and $T^{i-1}$ is elementary by $(b_{i-1})$, we have $\alpha_i\notin (\phibar_{i-1}(T^{i-1}(x_{i-1}))\cup\{\alpha_{i-1}\})$. Hence $\alpha_i\notin\varphi_{i-1}(P_1\cup P_2)$ by $(c_{i-1})$. Thus $\pi$ is not used on edges of $P_1\cup P_2-G[V(T^{i-1})]$ and $\alpha_i$ is not used on edges of $(P_1\cup P_2)\cap G[V(T^{i-1})]$ under $\varphi'$ and $\varphi_i'$. Therefore, since all vertices in $P_z(\pi,\alpha_i,\varphi_i')$ from $z$ to $x_i$ are contained in $T^{i-1}$, none of the $\alpha_i$-colored edges of $P_1\cup P_2$ under $\varphi_i'$ is in the subpath of $P_z(\pi,\alpha_i,\varphi_i')$ from $z$ to $x_i$. Consequently, we have $(e_i)$ as we switch all the $\alpha_i$-colored edges of $P_1\cup P_2$ not in $P_z(\pi,\alpha_i,\varphi_i')$ to $\pi$ by Kempe changes, and none of these $(\pi,\alpha_i)$-chains  intersect $T^{i-1}$ under $\varphi_i'$. Furthermore,

(5) $\varphi_i$ is also $(T^{i-1}(x_i)-x_i,D_{j-1},\varphi_{i-1})$-stable.

To prove $(a_i)$, by (5), we may assume $h\geq 1$. Since $\alpha_i\neq\gamma_h$, $\alpha_{i-1}\neq\gamma_h$ and $\pi\notin S_h$ following (4), by $(a_{i-1})$, (5) and the construction of $\varphi_i$, we have $(a_i)$ if $\alpha_i\neq\delta_h$ and $\alpha_{i-1}\neq\delta_h$. If $\alpha_i=\delta_h$, then all the $(\gamma_h,\delta_h)$-ears of $T_h$ under $\varphi_{i-1}$ are contained in $G[V(T^{i-1})]$ by TAA, so still we have $(a_i)$ from $(a_{i-1})$ and (5) as $P_{x_i}(\alpha_i,\alpha_{i-1},\varphi')\cap T^{i-1}=\{x_i\}$ and  none of the $(\pi,\alpha_i)$-chains we switched by Kempe changes from $\varphi_i'$ to $\varphi_i$ intersects $T^{i-1}$ (see the paragraph above (5)). So we may assume $\alpha_{i-1}=\delta_h$. In this case, we have $\alpha_{i}\neq\delta_h$ and all the $(\gamma_h,\delta_h)$-ears of $T_h$ under $\varphi_{i-1}$ are the same as under $\varphi'$. Moreover, all the $(\gamma_h,\delta_h)$-ears of $T_h$ are contained in $G[V(T^{i-2})]$ under $\varphi_{i-2}$ as $\alpha_{i-1}=\delta_h\in\phibar_{i-2}(T^{i-2})$, and they stay the same as under $\varphi_{i-1}$ by $(a_{i-1})$. Since  $P_{x_i}(\alpha_i,\alpha_{i-1},\varphi')$ is the same path as the $(\pi,\alpha_{i-1})$-exit path of $T^{i-1}$ under $\varphi_{i-1}$, it does not share a vertex with $T^{i-2}$ by $(f_{i-1})$. So all the $(\gamma_h,\delta_h)$-ears of $T_h$ under $\varphi_{i-1}$ stay the same as under $\varphi_i'$. Finally, since  $\alpha_{i}\neq\delta_h$ ($\alpha_{i-1}=\delta_h$) and $\pi\notin S_h$, we see that all the $(\gamma_h,\delta_h)$-ears of $T_h$ under $\varphi_{i}$ are the same as under $\varphi_{i}'$. Thus we have $(a_i)$ by $(a_{i-1})$ and (5).

By (A2)  and $(b_{i-1})$, we have

(6) if $h\geq 1$, then all the $\delta_h=\delta_n$-colored defective edges of $T_h$ under $\varphi_{i-1}$ are contained in some $(\gamma_h,\delta_h)$-ear of $T_h$ except for the $(\gamma_h,\delta_h)$-exit edge. 

By $(a_i)$ and (6), we see that $T_h$ has the same $(\gamma_h,\delta_h)$-exit vertex under $\varphi_i$ as under $\varphi_{i-1}$ if $h\geq 1$.  By $(f_{k})$ for all $1\leq k\leq i-1$, we have $x_i\notin T_{j-1}$.  Since $\varphi_i$ is $(T^{i-1}(x_i)-x_i, D_{j-1},\varphi_{i-1})$-stable and $x_i\notin T_{j-1}$, we see that $\varphi_i$ is also $(T_{j-1}, D_{j-1},\varphi_{i-1})$-stable ($T^{i-1}$ has the same rung number under $\varphi_{i-1}$ as $T_j$ under $\varphi$ by $(b_{i-1})$), and thus is a $T_{j-1}+f_{j-1}$ mod $\varphi_{i-1}$ coloring by (A2.5). Therefore, $T^{i-1}(x_i)$ is an ETT with the same $(j-1)$ rungs under $\varphi_i$ as $T^{i-1}$ under $\varphi_{i-1}$ following Corollary~\ref{vstable}. Since $T^i$ is a closure of $T^{i-1}(x_i)$ under $\varphi_i$ and $j-1\leq n$, we have $(b_i)$ by (A1) on $T^i$ and $(b_{i-1})$.

The remaining assertions follow as in the proof for \(\varphi_1\), with the indices changed; we omit the details.

Now with the above claim, we see that we can always recursively construct $\varphi_{i+1}$ if $(V(P_1)\cup V(P_2))- V(T^i)\neq\emptyset$  with a new vertex $x_{i+1}$ with $x_1\prec x_2\prec... \prec x_{i+1}$ along $T^i$ by $(f_{k})$ for all $1\leq k\leq i$ . But since we have a finite graph, at a certain point we have to stop at some $\varphi_i$ with $V(P_1)\cup V(P_2)\subseteq V(T^i)$. We denote the corresponding coloring $\varphi_i$ by $\mu$ and $T^i$ by $T$, respectively. Then (I) and (II) are satisfied by $\mu$ and $T$ following $(b_i)$, and $\mu$ is $(T_{j}(x)-x,D_{j-1},\varphi)$-stable by $(a_k)$ for all $1\leq k\leq i$. If  $P_1$ and $P_2$ have different ends in $V$, then both ends of $P_1$ and $P_2$ in $V$ are contained in $T^i$, and therefore $T^i$ is not elementary by $(d_i)$ and the fact that $V\cap V(T^i(x_i))=\emptyset$. Thus we have (III). Note that this is a contradiction to $(b_i)$, and this contradiction is the key to the proof of Lemma~\ref{A3}. 
\end{proof}
\begin{LEM}\label{A3.2}
	Let $n$ be a nonnegative integer. Suppose (A1), (A2), and (A2.5) hold for  all  ETTs with at most $n$ rungs. Let $T_{n+1}$ be a closed ETT with ladder $T_0\subset \dots \subset T_n \subset T_{n+1}$.  Assume $T_{n+1}$ is $\RE$ finished with initial index $h$, $\delta$ is a defective color of $T_{n+1}$ under $\varphi$, and $\alpha\in\phibar(x)$ where $x$ is the $(\alpha,\delta)$-exit vertex of $T_{n+1}$ with exit path $P_{ex}$ and $x\notin V(T_h)$. Suppose there's an $(\alpha,\delta)$-ear $D$ of $T_{n+1}$ with one root  $y$ such that $y\prec x$ along $T_{n+1}$. If $\tau\in\phibar(T_{n+1}(x))$ with $\tau\notin S_h$ and $\eta$ is an arbitrary color where $\eta\notin\phibar(T_{n+1}(x))$ and colors $\alpha,\eta,\delta$ and $\tau$ are all different, then there does not exist a path $P$ disjoint from $T_{n+1}$ with two ends $u$, $v$ and a vertex $w\in  V(P)$ with $w\notin\{u,v\}$  satisfying all of the following conditions:
	\begin{itemize}
		\item [(i)] $P$ shares edges with both $P_{ex}$ and $D$, and $P$ is vertex disjoint from $T_{n+1}$.
		\item [(ii)] $\varphi(P)\subseteq\{\tau,\eta,\delta\}$
		\item [(iii)] $\delta\in\phibar(w)$, $\phibar(u)\cap\{\eta,\tau,\delta\}\neq\emptyset$, and $\phibar(v)\cap\{\eta,\tau,\delta\}\neq\emptyset$.
		\item [(iv)] All the $\eta$-colored edges of $P$ belong to an $(\eta,\tau)$-path $P_1=P_r(\eta,\tau,\varphi)$ containing the vertex $w\in V(P)$ with one end $r\in T_{n+1}(x)$ and $\tau\in\phibar(r)$, and all of these $\eta$-colored edges of $P$ are in the subpath $P_1[w,r']$ of $P_1$, where $r'$ is the other end of $P_1$ and $P_1[w,r']$ does not intersect $T_{n+1}$.
	\end{itemize}
\end{LEM}	
\begin{proof}
Let everything be as described in Lemma~\ref{A3.2}, and assume otherwise there is a path $P$ disjoint from $T_{n+1}$ with two ends $u$, $v$ and a vertex $w\in  V(P)$ with $w\notin\{u,v\}$  satisfying conditions (i)-(iv). By the assumption in the statement of the current lemma, we have    

(1) $x\notin V(T_h)$.
 

By (A1) and (A2) on $T_h$ under $\varphi$, we have

(2) if $h\geq 1$, under $\varphi$ all the $\delta_h=\delta_n$-colored defective edges of $T_h$ are contained in some $(\gamma_h,\delta_h)$-ear of $T_h$ except for the $(\gamma_h,\delta_h)$-exit edge. 

Following (i) in the structure of $P$ (see the last paragraph above Lemma~\ref{A3.1}), we see that 

(3) There are two edge disjoint paths $P_2$ and $P_3$ in $P\cup P_{ex}\cup D$ from $\{x,y\}$ to $u$ and $v$, respectively. 

Since $y\prec x$, we see that the uncolored edge $e\in T_{n+1}(x)$. So $|\phibar(T_{n+1}(x))|\geq 4$. Following (1) and the assumption of the current lemma,  we have

(4) There exists a color $\pi\in\phibar(z)$ with $z\in V(T_{n+1}(x))$ such that $\pi,\alpha,\eta,\delta$,$\gamma_h$ (when $h\geq 1$) and $\tau$ are all different, and $\pi,\tau\notin S_h$.

This is because we can pick $\pi$ different from $\alpha$, $\tau$, and $\eta$ as $|\phibar(T_{n+1}(x))|\geq 4$ and $\eta\notin\phibar(T_{n+1}(x))$, and $\pi\neq\delta$ as $\delta\notin\phibar(T_{n+1})$ ($\delta$ is a defective color of $T_{n+1}$ and $T_{n+1}$ is closed). Now if $h\geq 1$, then $T_1\subset T_{n+1}$, so because  $|\phibar(T_1)|\geq 5$ by Lemma~\ref{size}, we can also require $\pi\notin S_h$.

As in the proof of Lemma~\ref{A3.1}, we will construct a series of colorings $\varphi_i$. We first let $\varphi_1'=\varphi/P_x(\alpha,\delta,\varphi)=\varphi/P_{ex}$, and then construct $\varphi_1$ from $\varphi_1'$ by switching all the $\alpha$-colored edges of $P\cup P_{ex}\cup D$ not in $P_z(\pi,\alpha,\varphi_1')$ to $\pi$ by Kempe changes. Let $j$ be the smallest index such that $x\in T_j$. So $j-1\geq h$ by (1). Let $T^1$ be a closure of $T_j(x)$ under $\varphi_1$. We claim that $\varphi_1$ satisfies the following:

\begin{itemize}
	
	\item [$(a_1)$] $\varphi_1$ is $(T_j(x)-x, D_{j-1},\varphi)$-stable, and if $h\geq 1$, then all the $(\gamma_h,\delta_h)$-ears of $T_h$ under $\varphi$ are the same as under $\varphi_1$.

\item [$(b_1)$] $T^1$ is an elementary ETT with the same $(j-1)$ rungs under $\varphi_1$ as $T_j$ under $\varphi$.

\item [$(c_1)$] $\varphi_1(P \cup P_{ex}\cup D)\subseteq\phibar_1(T^1(x))\cup\{\alpha,\eta\}$, and the set of $\eta$-colored edges and the set of $\tau$-colored edges are the same under $\varphi_1$ as under $\varphi$.

\item [$(d_1)$] $\phibar_1(w)\cap(\phibar_1(T^1(x))\cup\{\alpha\})\neq\emptyset$, $\phibar_1(u)\cap(\phibar_1(T^1(x))\cup\{\eta,\alpha\})\neq\emptyset$, and $\phibar_1(v)\cap(\phibar_1(T^1(x))\cup\{\eta,\alpha\})\neq\emptyset$.

\item [$(e_1)$] All the $\alpha$-colored edges of $P\cup P_{ex}\cup D$ under $\varphi_1$ belong to $P_z(\pi,\alpha,\varphi_1)$ after the vertex $x$ along the direction from $z$.

\item [$(f_1)$] If $\eta\notin\phibar_1(T^1)$ and $P_1$ has an $(\eta,\tau)$-exit vertex of $T^1$ in $P_1[r,w]-\{w\}$, then $\alpha\notin\phibar_1(T^1)$ and $T^1$ has a $(\pi,\alpha)$-exit vertex $x_2$ in $P_z(\pi,\alpha,\varphi_1)$ with $x\prec x_2$ along $T^1$, and the corresponding $(\pi,\alpha)$-exit path does not share a vertex with $T_{n+1}$.
\end{itemize}

Clearly $\varphi_1'=\varphi/P_x(\alpha,\delta,\varphi)$ is $(T_j(x)-x, D_{j-1},\varphi)$-stable, because $P_x(\alpha,\delta,\varphi)\cap T_{n+1}=\{x\}$. Moreover, since $T_{n+1}$ is $\RE$ finished, all the $(\gamma_h,\delta_h)$-ears of $T_h$ under $\varphi$ are contained in $G[V(T_{n+1})]$ following Algorithm 3.1. So all the $(\gamma_h,\delta_h)$-ears of $T_h$ under $\varphi$ are colored the same under $\varphi_1'$ if $h\geq 1$. As $T_{n+1}$ is closed for both $\alpha$ and $\pi$ under $\varphi$, all the $(\pi,\alpha)$-chains intersecting $T_{n+1}$ under $\varphi_1'$ are contained in $G[V(T_{n+1})]$ except for $P_z(\pi,\alpha,\varphi_1')$ that contains $x$, which is the only $(\pi,\alpha)$-exit of $T_{n+1}$ under $\varphi_1'$, and all vertices in $P_z(\pi,\alpha,\varphi_1')$ from $z$ to $x$ are contained in $T_{n+1}$.  So none of the $(\pi,\alpha)$-chains we switched by Kempe changes from $\varphi_1'$ to $\varphi_1$ intersect $T_{n+1}$ and none of the $\alpha$-colored edges of $P\cup P_{ex}\cup D$ under $\varphi_1'$ is in the subpath of $P_z(\pi,\alpha,\varphi_1')$ from $z$ to $x$. Thus we have $(a_1)$. Moreover, since $\pi\notin\varphi(P\cup P_{ex} \cup D)$ by the assumptions in the statement of the current Lemma, we have $(e_1)$.

By $(a_1)$ and (2), we see that $T_h$ has the same $(\gamma_h,\delta_h)$-exit vertex under $\varphi_1$ as under $\varphi$ if $h\geq 1$.  Since $\varphi_1$ is $(T_j(x)-x, D_{j-1},\varphi)$-stable and $x\notin T_{j-1}$, we see that $\varphi_1$ is also $(T_{j-1}, D_{j-1},\varphi)$-stable, and thus is $T_{j-1}+f_{j-1}$ mod $\varphi$ following (A2.5), and therefore $T_j(x)$ is an ETT with $j-1$ rungs under $\varphi_1$ following Corollary~\ref{vstable}. Since $T^1$ is a closure of $T_j(x)$ under $\varphi_1$ and $j-1\leq n$, we have $(b_1)$ by (A1) on $T^1$.

$(c_1)$ and $(d_1)$ follow from the construction of $\varphi_1$, (4), and assumptions in the current lemma.

Note that if $\eta\notin\phibar_1(T^1)$, then because $T^1$ is elementary by $(b_1)$,  we see that $P_1$ must have an $(\eta,\tau)$-exit vertex of $T^1$ following $(c_1)$. To prove $(f_1)$, we may assume that $\eta\notin\phibar_1(T^1)$ and $P_1$ has an $(\eta,\tau)$-exit vertex of $T^1$ in $P_1[r,w]-\{w\}$. 
Since $\varphi(P_{ex}\cup D)\subseteq\{\alpha,\delta\}$, by $(iv)$ of $P$, (4), and $(c_1)$, all the $\eta$-colored edges of $P_2$ and $P_3$ (see (3)) are in $P_1$ after $w$ along the direction from $r$. Because the $(\eta,\tau)$-exit vertex of $T^1$ is in $P_1[r,w]-\{w\}$, we see that

(5) None of the $\eta$-colored edges of $P_2$ and $P_3$ are in $E(G[V(T^1)])\cup\partial(T^1)$.

Following $(c_1)$, we have $\varphi_1(P_2\cup P_3) \subseteq \phibar_1(T^1)\cup\{\alpha,\eta\}$. 
  We first assume otherwise that  $\alpha\in\phibar_1(T^1)$. Then by (5) and TAA, we see that no edges of $P_2$ and $P_3$ are colored by $\eta$ and $u,v\in V(T^1)$. Since $T^1$ is elementary and $T^1(x)=T_j(x)\subseteq T_{n+1}$ by $(b_1)$, we have $\eta\in(\phibar_1(u)\cup\phibar_1(v))\subseteq \phibar_1(T^1)$ by $(d_1)$ and the fact that $u,v\notin V(T_{n+1})$, a contradiction to the assumption that $\eta\notin\phibar_1(T^1)$. So  $\alpha\notin\phibar_1(T^1)$, and therefore $T^1$ has an $(\alpha,\pi)$-exit vertex in $P_z(\pi,\alpha,\varphi_1)$, as $T^1$ is elementary by $(b_1)$. Moreover,  neither $u$ nor $v$ is in $T^1$, as otherwise we similarly have $\alpha\in\phibar_1(T^1)$ by $(d_1)$ and the assumption that $\eta\notin\phibar_1(T^1)$. So following (5) and TAA, the two edge disjoint paths $P_2$ and $P_3$ contain two $\alpha$-colored boundary edges $f$ and $g$ of $T^1$. Since all the $\alpha$-colored edges of $P_2$ and $P_3$ belong to $P_z(\pi,\alpha,\varphi_1)$ after the vertex $x$ along the direction from $z$ by $(e_1)$, one of $f$ and $g$ is incident to a vertex in $T^1$ other than $x$ and it is after $x$ along $P_z(\pi,\alpha,\varphi_1)$ along the direction from $z$. Since $T^1$ has an $(\alpha,\pi)$-exit vertex in $P_z(\pi,\alpha,\varphi_1)$, we see that $T^1$ has a $(\pi,\alpha)$-exit vertex $x_2$ in $P_z(\pi,\alpha,\varphi_1)$ where $x_2$ is after $x$ along $P_z(\pi,\alpha,\varphi_1)$ from the direction starting at $z$. Because $P_z(\pi,\alpha,\varphi_1)$ and $T_{n+1}$ only share vertices on the subpath of $P_z(\pi,\alpha,\varphi_1)$ from $z$ to $x$ (which is $P_z(\pi,\alpha,\varphi)$) by the construction of $\varphi_1$ ($T_{n+1}$ is closed for $\pi$ and it only has one boundary edge colored by $\alpha$ which is incident to $x$ under $\varphi_1$), we see that $x_2\notin T_{n+1}$ and the corresponding $(\pi,\alpha)$-exit path of $T^1$ under $\varphi_1$ does not share a vertex with $T_{n+1}$. So $x\prec x_2$ along $T^1$ and $(f_1)$ is established. 
  
  Now we have as claimed.  Suppose $\varphi_{i-1}$ ($i\geq 2$) is constructed with claims $(a_{i-1})$ to $(f_{i-1})$ below (note that $\varphi_1$ also satisfies the claims below with $\alpha_1=\alpha$, $x_1=x$, $T^0=T_{j}$, and $\varphi_0=\varphi$). If $\eta\in\phibar_{i-1}(T^{i-1})$ or $P_1$ does not contain an $(\eta,\tau)$-exit vertex of $T^{i-1}$ in $P_1[r,w]-\{w\}$, we stop and let $\mu=\varphi_{i-1}$. Otherwise, we recursively construct $\varphi_i$ from $\varphi_{i-1}$ and $x_i$ with an arbitrary color $\alpha_i\in\phibar_{i-1}(x_i)$,  and have the following claims. Here is the detailed construction: We first let $\varphi'=\varphi_{i-1}/(G-T^{i-1},\alpha_i,\pi)$, and then let $\varphi_i'=\varphi'/P_{x_i}(\alpha_i,\alpha_{i-1},\varphi')$. Finally,   we construct $\varphi_i$ from $\varphi_i'$ by switching all the $\alpha_i$-colored edges of $P\cup P_{ex}\cup D$ not in $P_z(\pi,\alpha_i,\varphi_i')$ to $\pi$ by Kempe changes. Let $T^i$ be a closure of $T^{i-1}(x_i)$ under $\varphi_i$.

\begin{itemize}
	
	\item [$(a_i)$] $\varphi_i$ is $(T^{i-1}(x_i)-x_i, D_{j-1},\varphi_{i-1})$-stable, and if $h\geq 1$, then all the $(\gamma_h,\delta_h)$-ears of $T_h$ under $\varphi_{i-1}$ are the same as under $\varphi_i$.
	
	\item [$(b_i)$] $T^i$ is an elementary ETT with the same $(j-1)$ rungs under $\varphi_i$ as $T_j$ under $\varphi$.
	
	\item [$(c_i)$] $\varphi_i(P \cup P_{ex}\cup D)\subseteq\phibar_i(T^i(x_{i}))\cup\{\alpha_i,\eta\}$, and the set of $\eta$-colored edges and the set of $\tau$-colored edges are the same under $\varphi_i$ as under $\varphi$.

	\item [$(d_i)$] $\phibar_i(w)\cap(\phibar_i(T^i(x_{i}))\cup\{\alpha_i\})\neq\emptyset$, $\phibar_i(u)\cap(\phibar_i(T^i(x_i))\cup\{\eta,\alpha_i\})\neq\emptyset$, and $\phibar_i(v)\cap(\phibar_i(T^i(x_i))\cup\{\eta,\alpha_i\})\neq\emptyset$.
	
	\item [$(e_i)$] All the $\alpha_i$-colored edges of $P\cup P_{ex}\cup D$ under $\varphi_i$ belong to $P_z(\pi,\alpha_i,\varphi_i)$ after the vertex $x_i$ along the direction from $z$.
	
	\item [$(f_i)$] If $\eta\notin\phibar_i(T^i)$ and $P_1$ has an $(\eta,\tau)$-exit vertex of $T^i$ in $P_1[r,w]-\{w\}$, then $\alpha_i\notin\phibar_i(T^i)$ and $T^i$ has a $(\pi,\alpha_i)$-exit vertex $x_{i+1}$ in $P_z(\pi,\alpha_i,\varphi_i)$ with $x_i\prec x_{i+1}$ along $T^i$, and the corresponding $(\pi,\alpha_i)$-exit path does not share a vertex with $T^{i-1}$.
\end{itemize}
 
As in the previous lemma, we also have
\(\{u,v,w\}\cap V(T^i(x_i))=\emptyset\).
Indeed, this holds initially. If the recursion continues, \(w\)
lies beyond the \((\eta,\tau)\)-exit, while
\(\eta,\alpha_i\notin\phibar_i(T^i)\).
Thus \((d_i)\) and elementariness exclude \(u,v\) from \(T^i\),
giving the assertion at the next stage.

Now we prove the above claims inductively. By $(f_{k})$ for all $1\leq k\leq i-1$ and (1), we see that $x= x_1\prec...\prec x_i$ along $T^{i-1}$, and therefore $x_i\notin T_h$ by (1) and $(b_{i-1})$. Following (4), $(a_{i-1})$, $(b_{i-1})$, $(c_{i-1})$, and the assumption of $\eta\notin\phibar_{i-1}(T^{i-1})$ ($i\geq 2$) and the $(\eta,\tau)$-exit vertex of $T^{i-1}$ being in $P_1[r,w]-\{w\}$, we have 

(6) $\alpha_i\notin\{\eta,\alpha_{i-1}\}$, and thus $\pi,\alpha_{i},\eta,\alpha_{i-1}$, $\gamma_h$ (when $h\geq 1$),  and $\tau$ are all different, and $\pi,\tau\notin  S_h$.

Clearly $\varphi'$ is $(T^{i-1},D_{j-1},\varphi_{i-1})$-stable as $T^{i-1}$ is closed for both $\alpha_i$ and $\pi$ under $\varphi_{i-1}$ and $\varphi'=\varphi_{i-1}/(G-T^{i-1},\alpha_i,\pi)$. By $(f_{i-1})$, we see that $P_{x_i}(\alpha_i,\alpha_{i-1},\varphi')\cap T^{i-1}=\{x_i\}$. So $\varphi_i'$ is  $(T^{i-1}(x_i)-x_i,D_{j-1},\varphi_{i-1})$-stable. As $T^{i-1}$ is closed for both $\alpha_i$ and $\pi$ under $\varphi_{i-1}$, all the $(\pi,\alpha_i)$-chains intersecting $T^{i-1}$ under $\varphi_i'$ are contained in $G[V(T^{i-1})]$ except for $P_z(\pi,\alpha_i,\varphi_i')$ that contains $x_i$, which is the only $(\pi,\alpha_i)$-exit of $T^{i-1}$ under $\varphi_i'$, and all vertices in $P_z(\pi,\alpha_i,\varphi_i')$ from $z$ to $x_i$ are contained in $T^{i-1}$. Moreover, since $\alpha_i\notin\{\eta,\alpha_{i-1}\}$ by (6), $x_{i-1}\prec x_{i}$ by $(f_{i-1})$, and $T^{i-1}$ is elementary by $(b_{i-1})$, we have $\alpha_i\notin (\phibar_{i-1}(T^{i-1}(x_{i-1}))\cup\{\alpha_{i-1},\eta\})$. So $\alpha_i\notin\varphi_{i-1}(P\cup P_{ex}\cup D)$ by $(c_{i-1})$.  Thus $\pi$ is not used on edges of $P\cup P_{ex}\cup D-G[V(T^{i-1})]$ and $\alpha_i$ is not used on edges of $(P\cup P_{ex}\cup D)\cap G[V(T^{i-1})]$ under $\varphi'$ and $\varphi_i'$.  Therefore, since all vertices in $P_z(\pi,\alpha_i,\varphi_i')$ from $z$ to $x_i$ are contained in $T^{i-1}$, none of the $\alpha_i$-colored edges of $P\cup P_{ex}\cup D$ under $\varphi_i'$ is in the subpath of $P_z(\pi,\alpha_i,\varphi_i')$ from $z$ to $x_i$. Consequently, we have $(e_i)$ as we switch all the $\alpha_i$-colored edges of $P\cup P_{ex}\cup D$ not in $P_z(\pi,\alpha_i,\varphi_i')$ to $\pi$ by Kempe changes, and none of these $(\pi,\alpha_i)$-chains  intersects $T^{i-1}$ under $\varphi_i'$. Furthermore,  

(7) $\varphi_i$ is also $(T^{i-1}(x_i)-x_i,D_{j-1},\varphi_{i-1})$-stable.

To prove $(a_i)$, by (7), we may assume $h\geq 1$. Since $\alpha_i\neq\gamma_h$, $\alpha_{i-1}\neq\gamma_h$ and $\pi\notin S_h$ (see (6)), by $(a_{i-1})$, (7) and the construction of $\varphi_i$, we have $(a_i)$ if $\alpha_i\neq\delta_h$ and $\alpha_{i-1}\neq\delta_h$. If $\alpha_i=\delta_h$, then all the $(\gamma_h,\delta_h)$-ears of $T_h$ under $\varphi_{i-1}$ are contained in $G[V(T^{i-1})]$ by TAA, so still we have $(a_i)$ from $(a_{i-1})$ and (7) as $P_{x_i}(\alpha_i,\alpha_{i-1},\varphi')\cap T^{i-1}=\{x_i\}$ and none of the $(\pi,\alpha_i)$-chains we switched by Kempe changes from $\varphi_i'$ to $\varphi_i$ intersects $T^{i-1}$ (see the paragraph above (7)). So we assume $\alpha_{i-1}=\delta_h$. In this case, we have $\alpha_{i}\neq\delta_h$ and all the $(\gamma_h,\delta_h)$-ears of $T_h$ under $\varphi_{i-1}$ are the same as under $\varphi'$. Moreover, all the $(\gamma_h,\delta_h)$-ears of $T_h$ are contained in $G[V(T^{i-2})]$ under $\varphi_{i-2}$ as $\alpha_{i-1}=\delta_h\in\phibar_{i-2}(T^{i-2})$, and they are the same as under $\varphi_{i-1}$ by $(a_{i-1})$. Since  $P_{x_i}(\alpha_i,\alpha_{i-1},\varphi')$ is the same path as the $(\pi,\alpha_{i-1})$-exit path of $T^{i-1}$ under $\varphi_{i-1}$, it does not share a vertex with $T^{i-2}$ by $(f_{i-1})$. So all the $(\gamma_h,\delta_h)$-ears of $T_h$ under $\varphi_{i-1}$ are the same as under $\varphi_i'$. Finally, since  $\alpha_{i}\neq\delta_h$ ($\alpha_{i-1}=\delta_h$) and $\pi\notin S_h$, we see that all the $(\gamma_h,\delta_h)$-ears of $T_h$ under $\varphi_{i}$ are the same as under $\varphi_{i}'$. Thus we have $(a_i)$ by $(a_{i-1})$ and (7).

By (A2) and $(b_{i-1})$ under $\varphi_{i-1}$ , we have

(8) if $h\geq 1$, then all the $\delta_h=\delta_n$-colored defective edges of $T_h$ under $\varphi_{i-1}$ are contained in some $(\gamma_h,\delta_h)$-ear of $T_h$ except for the $(\gamma_h,\delta_h)$-exit edge.

By $(a_i)$ and (8), we see that $T_h$ has the same $(\gamma_h,\delta_h)$-exit vertex under $\varphi_i$ as under $\varphi_{i-1}$ if $h\geq 1$.  Note that $T^{i-1}$ has the same rung number under $\varphi_{i-1}$ as $T_j$ under $\varphi$ following $(b_{i-1})$. So by $(f_{k})$ for all $1\leq k\leq i-1$, we have $x_i\notin T_{j-1}$.  Since $\varphi_i$ is $(T^{i-1}(x_i)-x_i, D_{j-1},\varphi_{i-1})$-stable and $x_i\notin T_{j-1}$, we see that $\varphi_i$ is also $(T_{j-1}, D_{j-1},\varphi_{i-1})$-stable, and thus is a $T_{j-1}+f_{j-1}$ mod $\varphi_{i-1}$ coloring by (A2.5). Therefore, following Corollary~\ref{vstable}, $T^{i-1}(x_i)$ is an ETT with the same $(j-1)$ rungs under $\varphi_i$ as $T^{i-1}$ under $\varphi_{i-1}$. Since $T^i$ is a closure of $T^{i-1}(x_i)$ under $\varphi_i$ and $j-1\leq n$, we have $(b_i)$ by (A1) on $T^i$ and $(b_{i-1})$.

The remaining assertions follow as in the case \(i=1\), with the indices changed; we omit the details.

Now with the above claims, we see that we can always recursively construct $\varphi_{i}$ from $x_{i}$ if  $\eta\notin\phibar_{i-1}(T^{i-1})$ ($i\geq 2$) and the $(\eta,\tau)$-exit vertex of $T^{i-1}$ is in $P_1[r,w]-\{w\}$. But since we have a finite graph, at a certain point we have to stop at some $\varphi_i$ when $\eta\in\phibar_{i}(T^{i})$ or the $(\eta,\tau)$-exit vertex of $T^{i}$ is in $P_1[w,r']$; we still denote the corresponding coloring $\varphi_i$ by $\mu$. So $\mu$ satisfies all conditions $(a_i)$ to $(f_i)$.

Note that by $(c_i)$, the fact that $x\preceq x_i$ (see the paragraph above (6)) and the assumption $(iv)$ of $P$, we have

(9) $\tau\in\mubar(r)$, and therefore $\tau\in\mubar(T^{i})$. Moreover, $\mu(P_1)\subseteq\{\tau,\eta\}$.
  
We first consider the case that $\eta\in\mubar(T^{i})$.  Then in view of (9), we see that $w\in V(T^{i})$ by TAA along $P_1$. So $\alpha_i\in\mubar(w)$ following $(d_i)$, because $T^i$ is elementary by $(b_i)$ and $w\notin V(T^i(x_i))$. By $(c_i)$, both $u$ and $v$ are in $T^i$ by TAA. However, following $(d_i)$ and the fact that $u,v\notin V(T^i(x_i))$, we reach a contradiction to the elementariness of $T^i$ given by $(b_i)$. So we may assume that

(10) $\eta\notin\mubar(T^{i})$, and $P_1$ has an $(\eta,\tau)$-exit vertex $y_1$ of $T^{i}$ in $P_1[w,r']$. 

By hypothesis (iv) when \(i=1\), and by the preceding exit
condition otherwise, \(y_1\notin V(T^i(x_i))\). Thus

(11) $x_i\prec y_1$ along $T^i$.

We will then proceed in a similar fashion as earlier to recursively construct a series of colorings $\mu_l$ and corresponding claims with $l\geq 1$ using vertex $y_l$. Here is the detailed construction:

Let $\mu_0=\mu$, $y_0=x_i$, $\beta_0=\eta$ and $T^0_{\mu}=T^i$. Suppose we have already constructed $\mu_{l-1}$ and $T^{l-1}_{\mu}$, and suppose they satisfy the claims below (with index $l-1$) if $l\geq 2$. If $\alpha_i\in\mubar_{l-1}(T^{l-1}_{\mu}(y_l))$, we stop. Otherwise, we pick an arbitrary color $\beta_l$ in $\mubar_{l-1}(y_l)$ and let $\mu'=\mu_{l-1}/(G-T^{l-1}_{\mu},\beta_l,\tau)$. Then let $\mu_l'=\mu'/P_{y_l}(\beta_l,\beta_{l-1},\mu')$. Finally,  we construct $\mu_l$ from $\mu_l'$ by switching all the $\beta_l$-colored edges of $P\cup P_{ex}\cup D\cup P_1$ not in $P_r(\tau,\beta_l,\mu_l')$ to $\tau$ by Kempe changes. Let $T^l_{\mu}$ be a closure of $T^{l-1}_{\mu}(y_l)$ under $\mu_l$.

\begin{itemize}
	
	\item [$(A_l)$] $\mu_l$ is $(T^{l-1}_{\mu}(y_l)-y_l, D_{j-1},\mu_{l-1})$-stable, and if $h\geq 1$, then all the $(\gamma_h,\delta_h)$-ears of $T_h$ under $\mu_{l-1}$ are the same as under $\mu_l$.
	
	\item [$(B_l)$] $T^l_{\mu}$ is an elementary ETT with the same $(j-1)$ rungs under $\mu_l$ as $T^i$ under $\mu$.
	
	\item [$(C_l)$] $\mu_l(P \cup P_{ex}\cup D)\subseteq\mubar_l(T^l_{\mu}(y_{l}))\cup\{\alpha_i,\beta_{l}\}$, $\mu_l(P_1)\subseteq\mubar_l(T^l_{\mu}(y_{l}))\cup\{\beta_{l}\}$, and the set of $\alpha_i$-colored edges and the set of $\pi$-colored edges are the same under $\mu_l$ as under $\mu$.

	\item [$(D_l)$] $\mubar_l(w)\cap(\mubar_l(T^l_{\mu}(y_{l}))\cup\{\alpha_i\})\neq\emptyset$, $\mubar_l(u)\cap(\mubar_l(T^l_{\mu}(y_l))\cup\{\beta_l,\alpha_i\})\neq\emptyset$, and $\mubar_l(v)\cap(\mubar_l(T^l_{\mu}(y_l))\cup\{\beta_l,\alpha_i\})\neq\emptyset$.
	
	\item [$(E_l )$] All the $\beta_l$-colored edges of $P\cup P_{ex}\cup D\cup P_1$ under $\mu_l$ belong to $P_r(\tau,\beta_l,\mu_l)$ after the vertex $y_l$ along the direction from $r$.
	
	\item [$(F_l)$] $\beta_l\notin\mubar_l(T^l_{\mu})$ and $T^l_{\mu}$ has a $(\tau,\beta_l)$-exit vertex $y_{l+1}$ in $P_r(\tau,\beta_l,\mu_l)$ with $y_l\prec y_{l+1}$ along $T^l_{\mu}$, and the corresponding $(\tau,\beta_l)$-exit path does not share a vertex with $T^{l-1}_{\mu}$.
\end{itemize}

In this induction, we also show that
\(\{u,v,w\}\cap V(T_\mu^l(y_l))=\emptyset\), and that every
vertex of \(\{u,v,w\}\cap V(T_\mu^l)\) misses \(\alpha_i\).
For \(l=0\), this follows from the preceding observation,
\((d_i)\), and (10). For \(l\ge1\), the preceding stage and the
stopping rule give the first assertion. Using this exclusion
in the proof of \((F_l)\) below gives
\(\beta_l\notin\bar\mu_l(T_\mu^l)\), and \((D_l)\) then gives
the second assertion.

We then prove the above claims inductively in a similar fashion as earlier. Here are the details. Recall that $x\preceq x_i$ (see the paragraph above (6)).  By $(F_{k})$ for all $1\leq k\leq l-1$, (11), $(b_{i})$, $(B_{l-1})$, the choice of $j$ and (1), we see that 

(12) $x_i=y_0\prec y_1\prec...\prec y_l$ along $T^{l-1}_{\mu}$, $y_l\notin T_{h}$, and $y_l\notin T_{j-1}$. 

Since we did not stop at $\mu_{l-1}$, by the construction algorithm of $\mu_l$, we have $\alpha_i\notin\mubar_{l-1}(T_{\mu}^{l-1}(y_l))$ and thus $\alpha_i\neq\beta_l$. Moreover, $\beta_{l-1}\neq\alpha_i$ by (6) (when $l=1$) and  the construction algorithm of $\mu_{l-1}$ (when $l\geq 2$).
So following (6) on $\varphi_i=\mu_0$, (10), the elementariness of $T^{l-1}_{\mu}$ by $(B_{l-1})$ (when $l\geq 2$) and $(b_i)$ (when $l=1$), $(A_{k})$ for all $1\leq k\leq l-1$ (when $l\geq 2$), $(C_{l-1})$, $(F_{l-1})$ (when $l\geq 2$) and $(c_i)$ (when $l=1$), we have 

(13) $\pi,\tau\notin  S_h$, if $l\geq 2$ then $\beta_l, \beta_{l-1}, \pi,\alpha_{i},\eta, \gamma_h$ (when $h\geq 1$)  and $\tau$ are distinct, and if $l\geq 1$ then  $\beta_1, \pi,\alpha_{i},\eta=\beta_0, \gamma_h$ (when $h\geq 1$)  and $\tau$ are distinct.

Note that $T^{l-1}_{\mu}$ is an ETT with the same $(j-1)$ rungs under $\mu_{l-1}$ as $T_j$ under $\varphi$ by $(B_{l-1})$ (when $l\geq 2$) and $(b_{i})$ (when $l=1$). Clearly $\mu'$ is $(T^{l-1}_{\mu},D_{j-1},\mu_{l-1})$-stable as $T^{l-1}_{\mu}$ is closed for both $\beta_l$ and $\tau$ under $\mu_{l-1}$ and $\mu'=\mu_{l-1}/(G-T^{l-1}_{\mu},\beta_l,\tau)$. By $(F_{l-1})$ (when $l\geq 2$) and (10) (when $l=1$), we see that $P_{y_l}(\beta_l,\beta_{l-1},\mu')\cap T^{l-1}_{\mu}=\{y_l\}$. So $\mu_l'$ is $(T^{l-1}_{\mu}(y_l)-y_l,D_{j-1},\mu_{l-1})$-stable. As $T^{l-1}_{\mu}$ is closed for both $\beta_l$ and $\tau$ under $\mu_{l-1}$, all the $(\tau,\beta_l)$-chains intersecting $T^{l-1}_{\mu}$ under $\mu_l'$ are contained in $G[V(T^{l-1}_{\mu})]$ except for $P_r(\tau,\beta_l,\mu_l')$ that contains $y_l$, which is the only $(\tau,\beta_l)$-exit of $T^{l-1}_{\mu}$ under $\mu_l'$, and all vertices in $P_r(\tau,\beta_l,\mu_l')$ from $r$ to $y_l$ are contained in $T^{l-1}_{\mu}$. Moreover, since $\beta_l\notin\{\tau,\beta_{l-1},\alpha_i\}$ by (13), $y_{l-1}\prec y_{l}$ by (12), and $T^{l-1}_{\mu}$ is elementary under $\mu_{l-1}$ by $(B_{l-1})$ (when $l\geq 2$) and $(b_i)$ (when $l=1$), we have $\beta_l\notin (\mubar_{l-1}(T^{l-1}_{\mu}(y_{l-1}))\cup\{\beta_{l-1},\alpha_i\})$. So $\beta_l\notin\mu_{l-1}(P\cup P_{ex}\cup D \cup P_1)$ by $(C_{l-1})$ (when $l\geq 2$) and $(c_i)$ (when $l=1$).  Thus $\tau$ is not used on edges of $P\cup P_{ex}\cup D\cup P_1-G[V(T^{l-1}_{\mu})]$ and $\beta_l$ is not used on edges of $(P\cup P_{ex}\cup D\cup P_1)\cap G[V(T^{l-1}_{\mu})]$ under $\mu'$ and $\mu_l'$. Therefore, since all vertices in $P_r(\tau,\beta_l,\mu_l')$ from $r$ to $y_l$ are contained in $T^{l-1}_{\mu}$, none of the $\beta_l$-colored edges of $P\cup P_{ex}\cup D\cup P_1$ under $\mu_l'$ is in the subpath of $P_r(\tau,\beta_l,\mu_l')$ from $r$ to $y_l$. Consequently, we have $(E_l)$ as we switch all the $\beta_l$-colored edges of $P\cup P_{ex}\cup D\cup P_1$ not in $P_r(\tau,\beta_l,\mu_l')$ to $\tau$ by Kempe changes, and none of these $(\tau,\beta_l)$-chains  intersect $T^{l-1}_{\mu}$ under $\mu_l'$. Furthermore,  

(14) $\mu_l$ is also $(T^{l-1}_{\mu}(y_l)-y_l,D_{j-1},\mu_{l-1})$-stable.

To prove $(A_l)$, by (14), we may assume $h\geq 1$. Since $\beta_l\neq\gamma_h$, $\beta_{l-1}\neq\gamma_h$ and $\tau\notin S_h$ (see (13)), by $(A_{l-1})$ (when $l\geq 2$), $(a_i)$ (when $l=1$), (14) and the construction of $\mu_l$, we have $(A_l)$ if $\beta_l\neq\delta_h$ and $\beta_{l-1}\neq\delta_h$. If $\beta_l=\delta_h$, then all the $(\gamma_h,\delta_h)$-ears of $T_h$ under $\mu_{l-1}$ are contained in $G[V(T^{l-1}_{\mu})]$ by TAA, so still we have $(A_l)$ from $(A_{l-1})$ (when $l\geq 2$), $(a_i)$ (when $l=1$) and (14) as $P_{y_l}(\beta_l,\beta_{l-1},\mu')\cap T^{l-1}_{\mu}=\{y_l\}$ and none of the $(\tau,\beta_l)$-chains we switched by Kempe changes from $\mu_l'$ to $\mu_l$ intersects $T^{l-1}_{\mu}$ (see the paragraph above (14)). So we assume $\beta_{l-1}=\delta_h$. In this case, we have $\beta_{l}\neq\delta_h$ and all the $(\gamma_h,\delta_h)$-ears of $T_h$ under $\mu_{l-1}$ are the same as under $\mu'$. Let's first consider the case that $l\geq 2$. Then all the $(\gamma_h,\delta_h)$-ears of $T_h$ are contained in $G[V(T^{l-2}_{\mu})]$ under $\mu_{l-2}$ as $\beta_{l-1}=\delta_h\in\mubar_{l-2}(T^{l-2}_{\mu})$, and they are the same as under $\mu_{l-1}$ by $(A_{l-1})$. Since  $P_{y_l}(\beta_l,\beta_{l-1},\mu')$ is the same path as the $(\tau,\beta_{l-1})$-exit path of $T^{l-1}_{\mu}$ under $\mu_{l-1}$, it does not share a vertex with $T^{l-2}_{\mu}$ by $(F_{l-1})$. So all the $(\gamma_h,\delta_h)$-ears of $T_h$ under $\mu_{l-1}$ are the same as under $\mu_l'$. Finally, since  $\beta_{l}\neq\delta_h$ ($\beta_{l-1}=\delta_h$) and $\tau\notin S_h$ (see (13)), we see that all the $(\gamma_h,\delta_h)$-ears of $T_h$ under $\mu_{l}$ are the same as under $\mu_{l}'$. Thus we have $(A_l)$ by $(A_{l-1})$ and (14). 

So we may assume that $l=1$. So $\beta_{0}=\eta=\delta_h$. In this case, $P_{y_1}(\beta_1,\delta_h,\mu')$ is the same path as the $(\tau,\delta_h)$-exit path of $T^{i}$, which is the subpath $P_1[y_1,r']$ of $P_1[w,r']$. Since  $T_{n+1}$ is $\RE$ finished under $\varphi$, all the $(\gamma_h,\delta_h)$-ears of $T_h$ are contained in $G[V(T_{n+1})]$. Since $P_1[w,r']$ does not intersect $T_{n+1}$ by assumption (iv) of $P$, $P_1[y_1,r']$ does not intersect any  $(\gamma_h,\delta_h)$-ear of $T_h$ under $\varphi$. Because  all the $(\gamma_h,\delta_h)$-ears of $T_h$ under $\varphi$ are the same as under $\varphi_i$ by $(a_k)$ for all $1\leq k\leq i$, we see that $P_1[y_1,r']$ does not intersect any  $(\gamma_h,\delta_h)$-ear of $T_h$ under $\varphi_i=\mu_0$. Thus all the $(\gamma_h,\delta_h)$-ears of $T_h$ under $\mu_{0}$ are the same as under $\mu_1'$ by the construction of $\mu_1'$. Since  $\beta_{1}\neq\delta_h$ ($\beta_{0}=\delta_h$) and $\tau\notin S_h$ (see (13)), we see that all the $(\gamma_h,\delta_h)$-ears of $T_h$ under $\mu_{1}$ are the same as under $\mu_{1}'$. Thus we have $(A_1)$ by (14).

By (A2), $(B_{l-1})$ (when $l\geq 2$), and $(b_{i})$ (when $l=1$) under $\mu_{l-1}$, we have

(15) if $h\geq 1$, then all the $\delta_h$-colored defective edges of $T_h$ under $\mu_{l-1}$ are contained in some $(\gamma_h,\delta_h)$-ear of $T_h$ except for the $(\gamma_h,\delta_h)$-exit edge.

By $(A_l)$ and (15), we see that $T_h$ has the same $(\gamma_h,\delta_h)$-exit vertex under $\mu_l$ as under $\mu_{l-1}$ if $h\geq 1$.  Recall that $y_l\notin T_{j-1}$ by (12).  Since $\mu_l$ is $(T^{l-1}_{\mu}(y_l)-y_l, D_{j-1},\mu_{l-1})$-stable, $y_l\notin T_{j-1}$, and $T^{l-1}_{\mu}$ has the same rung number under $\mu_{l-1}$ as $T_j$ under $\varphi$ by $(B_{l-1})$ (when $l\geq 2$) and $(b_{i})$ (when $l=1$),  we see that $\mu_l$ is also $(T_{j-1}, D_{j-1},\mu_{l-1})$-stable, and thus is a $T_{j-1}+f_{j-1}$ mod $\mu_{l-1}$ coloring by (A2.5). Therefore, following Corollary~\ref{vstable}, $T^{l-1}_{\mu}(y_l)$ is an ETT with the same $(j-1)$ rungs under $\mu_l$ as $T^{l-1}_{\mu}$ under $\mu_{l-1}$. Since $T^l_{\mu}$ is a closure of $T^{l-1}_{\mu}(y_l)$ under $\mu_l$ and $j-1\leq n$, we have $(B_l)$ by (A1) on $T^l_{\mu}$, $(B_{l-1})$ (when $l\geq 2$) and $(b_i)$ (when $l=1$).

Following assumption (iv) of $P$, the construction of $\mu_l$ and (13), we have $(C_l)$ and $(D_{l})$ by $(C_{l-1})$ and $(D_{l-1})$ when $l\geq 2$, and by $(c_i)$ and $(d_i)$ when $l=1$.

Note that $T^l_{\mu}(x)=T_j(x)$ following $(A_l)$ and $(B_l)$. We claim that $\beta_l\notin\mubar_{l}(T^l_{\mu})$. Otherwise, by $(C_l)$, we have that $V(P_1)\subseteq V(T^{l}_{\mu})$ by TAA. So $w\in V(T^{l}_{\mu})$. Because $T^{l}_{\mu}$ is elementary by $(B_l)$, we have $\alpha_i\in\mubar_l(w)$ by $(D_l)$ and the fact that $w\notin V(T_\mu^l(y_l))$. So $u$ and $v$, the two ends of $P$ are also in $T^{l}_{\mu}$ by $(C_l)$ and TAA. However, this is a contradiction to $(D_l)$ and the elementariness of $T^{l}_{\mu}$ and the fact that $u,v\notin V(T_\mu^l(y_l))$. So $\beta_l\notin\mubar_{l}(T^l_{\mu})$, and therefore $T^l_{\mu}$ must have a $(\tau,\beta_l)$-exit vertex $y_{l+1}$ in $P_r(\tau,\beta_l,\mu_l)$ by the elementariness of $T^{l}_{\mu}$. This proves the first half of $(F_l)$.

We then claim that  $P_r(\tau,\beta_l,\mu_l)$ contains a vertex in $T^{l}_{\mu}$ that is after $y_l$ along the direction from $r$.
Indeed, if $w\notin T^l_{\mu}$, then the two edge disjoint paths $P_1[r,w]$ and $P_1[w,y_1]$ provide two $\beta_l$-colored boundary edges of $T^l_{\mu}$ which are also contained in $P_r(\tau,\beta_l,\mu_l)$ after the vertex $y_l$ along the direction from $r$ by $(E_l)$ and $(C_l)$, and one of them provides a vertex in $T^{l}_{\mu}$ that is after $y_l$ along the direction from $r$. So we assume that $w\in T^l_{\mu}$. Then because $w\notin V(T_\mu^l(y_l))$, we have $\alpha_i\in\mubar_l(w)$ by $(D_l)$ and the elementariness of $T^{l}_{\mu}$ given by $(B_l)$. If one of $u$ and $v$, say $v$, is contained in $T^{l}_{\mu}$, then $\beta_l\in\mubar_l(v)$ by the elementariness of $T^{l}_{\mu}$, $(D_l)$ and the fact that $u,v\notin V(T_\mu^l(y_l))$. So both $u$ and $v$ are contained in $T^{l}_{\mu}$ by TAA and $(C_l)$, a contradiction to $(D_l)$ and the elementariness of $T^{l}_{\mu}$. Thus neither $u$ nor $v$ is contained in $T^{l}_{\mu}$, and in this case $P_2$ and $P_3$ (see (3)) provide two $\beta_l$-colored boundary edges of $T^l_{\mu}$ which are also contained in $P_r(\tau,\beta_l,\mu_l)$ after the vertex $y_l$ along the direction from $r$ by $(E_l)$ and $(C_l)$. As before, we have as claimed.

Now  by the claim above, $P_r(\tau,\beta_l,\mu_l)$ contains a $(\tau,\beta_l)$-exit vertex $y_{l+1}$ of $T^{l}_{\mu}$ such that $y_{l+1}$ is after $y_l$ along the direction starting from $r$. Because $P_r(\tau,\beta_l,\mu_l)$ and $T^{l-1}_{\mu}$ only share vertices on the subpath of $P_r(\tau,\beta_l,\mu_l)$ from $r$ to $y_l$ by the construction of $\mu_l$ ($T^{l-1}_{\mu}$ is closed for $\tau$ and it only has one boundary edge colored by $\beta_l$ which is the edge incident to $y_l$ under $\mu_l$), we see that $y_{l+1}\notin V(T^{l-1}_{\mu})$ and the corresponding $(\tau,\beta_l)$-exit path of $T^l_{\mu}$ under $\mu_l$ does not share a vertex with $T^{l-1}_{\mu}$. So $y_l\prec y_{l+1}$ along $T^l_{\mu}$ and $(F_l)$ is established. This finishes all six claims $(A_l)$ to $(F_l)$ above.

With the six claims above and the construction algorithm of $\mu_l$ earlier, we see that we can always recursively construct $\mu_{l+1}$ with a new vertex $y_{l+1}$ with $x_i=y_0\prec y_1\prec...\prec y_{l+1}$ along $T^l_{\mu}$ by (12), unless $\alpha_i\in\mubar_{l}(T^l_{\mu}(y_{l+1}))$. But since we have a finite graph, at a certain point we have to stop at some $\mu_l$ with $\alpha_i\in\mubar_{l}(T^l_{\mu}(y_{l+1}))$. For convenience, we still denote the corresponding coloring by $\mu_l$ and use index $l$ for everything we established earlier. Note that $l=0$ may happen.

So we have $\alpha_i\in\mubar_{l}(T^l_{\mu})$. By \((e_i)\) and \((C_l)\), all the \(\alpha_i\)-colored edges
of \(P\cup P_{\rm ex}\cup D\) belong to
\(P_z(\pi,\alpha_i,\mu_l)\).
Since both colors are missing on \(T_\mu^l\), this chain is
contained in \(G[V(T_\mu^l)]\) by TAA. Thus

(16) all the  $\alpha_i$-colored edges of $P\cup P_{ex}\cup D$ are contained in $G[V(T^l_{\mu})]$. 

By the observation above and elementariness, at least two of
\(u,v,w\) lie outside \(T_\mu^l\).
If \(u,v\) are outside, let \(Q_2=P_2\) and \(Q_3=P_3\).
Otherwise, say \(u\in V(T_\mu^l)\), so \(w,v\) are outside.
The paths \(P_1[r,w]\) and \(P_1[y_1,w]\) are edge disjoint.
Splicing \(P[v,w]\), from \(v\), at its first intersection
with \(P_1[r,y_1]\), and retaining the other subpath to \(w\),
gives two edge disjoint paths \(Q_2,Q_3\) from \(r,y_1\) to
\(w,v\). The case \(v\in V(T_\mu^l)\) is symmetric.
Denote the outside ends of \(Q_2,Q_3\) by \(a,b\).
Since \(P_1\) has no \(\alpha_i\)-colored edge, \((C_l)\),
\((E_l)\), and (16) apply to these paths as well; when \(l=0\),
use \((c_i)\) and hypothesis (iv). Thus following TAA, (16), \((C_l)\) (when \(l\ge1\)) and
\((c_i)\) together with hypothesis (iv) (when \(l=0\)), the
paths \(Q_2\) and \(Q_3\) contain two \(\beta_l\)-colored
boundary edges \(f\) and \(g\) of \(T_\mu^l\), where
\(f\in E(Q_2)\) is incident with \(u'\in V(T_\mu^l)\) and
\(g\in E(Q_3)\) is incident with \(v'\in V(T_\mu^l)\), such
that \(Q_2[u',a]\) and \(Q_3[v',b]\) do not share edges with
\(G[V(T_\mu^l)]\). Note that all the \(\beta_l\)-colored
edges of \(Q_2\cup Q_3\) are contained in
\(P_r(\tau,\beta_l,\mu_l)\), by \((E_l)\) when \(l\ge1\),
and by \((c_i)\) and hypothesis (iv) when \(l=0\).
Therefore, since \(y_{l+1}\) is the
\((\tau,\beta_l)\)-exit vertex of \(T_\mu^l\) in
\(P_r(\tau,\beta_l,\mu_l)\), both \(u'\) and \(v'\) lie
in its subpath from \(r\) to \(y_{l+1}\).

We then let \(u^*\) and \(v^*\) be the vertices nearest
to \(a\) and \(b\), respectively, in the intersections of
\(Q_2[u',a]\) and \(Q_3[v',b]\) with this subpath.
Without loss of generality, we may assume that \(u^*\)
is not after \(v^*\) along
\(P_4=P_r(\tau,\beta_l,\mu_l)\) from \(r\). Now
$P_5=P_4[r,u^*]\cup Q_2[u^*,a]$ and
$P_6=P_4[y_{l+1},v^*]\cup Q_3[v^*,b]$
are two edge disjoint paths from \(r\) to \(a\) and from
\(y_{l+1}\) to \(b\), respectively. Moreover, by (13),
(16), \((C_l)\) when \(l\ge1\), and \((c_i)\) and
hypothesis (iv) when \(l=0\), we have


(17) $\mu_{l}(P_5\cup P_6)\subseteq (\mubar_l(T^l_{\mu}(y_{l}))\cup \{\beta_l\})\setminus\{\alpha_i\}$. 

Finally, we consider the coloring $\phi=\mu_l/(G-T^{l}_{\mu},\beta_{l+1},\tau)$, where $\beta_{l+1}$ is an arbitrary color in $\mubar_{l}(y_{l+1})$. Note that $\beta_{l+1}=\alpha_i$ may happen. It is now routine to check  that $T^l_{\mu}$ remains an ETT with the same $(j-1)$ rungs under $\phi$ as $T_j$ under $\varphi$. Moreover, $P_{y_{l+1}}(\beta_l,\beta_{l+1},\phi)$ is the $(\beta_l,\beta_{l+1})$-exit path of $T^{l}_{\mu}$. By (17), no matter whether $\beta_{l+1}=\alpha_i$ or not, by the elementariness of $T^{l}_{\mu}$ under $\mu_l$ and $\phi$, we see that $\beta_{l+1}\notin \mu_{l}(P_5\cup P_6)$. Thus $\phi(P_5\cup P_6)\subseteq (\overline{\phi}(T^l_{\mu}(y_{l+1}))\cup \{\beta_l\})$, $\beta_{l+1}\notin\phi((P_5\cup P_6)\cap G[V(T^l_{\mu})])$, and $\tau\notin\phi((P_5\cup P_6)- G[V(T^l_{\mu})])$.  
In addition, \(\tau\notin\{\beta_{l+1}\}\cup S_h\).
By \((D_l)\), or \((d_i)\) when \(l=0\), and
\(\alpha_i\in\bar\mu_l(T_\mu^l(y_{l+1}))\), each of \(a,b\)
misses a color in
\(\bar\mu_l(T_\mu^l(y_{l+1}))\cup\{\beta_l\}\).
The exterior interchange permutes two colors missing on this
segment, so condition (ii) of Lemma~4.5 holds under $\phi$. So by applying  Lemma~\ref{A3.1} with $P_{y_{l+1}}(\beta_l,\beta_{l+1},\phi)$ in place of $P_{ex}$, $r$ in place of $z$, $\beta_l$ in place of $\delta$, $\beta_{l+1}$ in place of $\alpha$, and $\tau$ in place of $\pi$, we have a contradiction between $(II)$ and $(III)$, because $P_5$ and $P_6$ have different ends $a$ and $b$.

It remains to verify the exceptional hypothesis of Lemma~\ref{A3.1}. Under the substitutions above, this hypothesis is relevant when $\beta_l=\delta_h$. If $l\geq 1$, all $(\gamma_h,\delta_h)$-ears of $T_h$ are contained in $G[V(T^{l-1}_{\mu})]$ by TAA and $(A_l)$, while the corresponding $(\tau,\beta_l)$-exit path is disjoint from $T^{l-1}_\mu$ by $(F_l)$. After interchanging $\tau$ and $\beta_{l+1}$ outside $T^{l}_\mu$ to obtain $\phi$, the resulting $(\beta_l,\beta_{l+1})$-exit path remains disjoint from these ears. The case $l=0$ follows from (10) and hypothesis $(iv)$ of the current Lemma.

\end{proof}	
\begin{LEM}(A3.5)\label{A3.5}
	Let $n$ be a nonnegative integer. Suppose (A1), (A2), (A2.5), and (A3) hold for  all  ETTs with at most $n$ rungs. Let $T_{n+1}$ be a closed ETT with ladder $T_0\subset \dots \subset T_n \subset T_{n+1}$. Assume $T_{n+1}$ could be extended further with $F_{n+1}=\{f_{n+1}\}$ and $S_{n+1}=\{\gamma_{n+1},\delta_{n+1}\}$.  Let $\alpha,\beta$ be two colors such that $\alpha\in\phibar(T_{n+1})\setminus\{\gamma_{n+1}\}$ and $\beta$ is arbitrary. For any $(\alpha,\beta)$-path $P$, if $P$ does not intersect $T_{n+1}$, then the coloring $\varphi'=\varphi/P$ is a $T_{n+1}+f_{n+1}$ mod $\varphi$ coloring.
	
\end{LEM}	

\begin{proof}
	Let $T_{n+1}$, $\alpha$, $\beta$, and $P$ be as described above. So $\varphi'=\varphi/P$ is $(T_{n+1},D_{n+1},\varphi)$-stable.  So Lemma~\ref{A3.5} follows immediately from (A3) and (A2.5) if $T_{n+1}$ is $\RE$ finished. Thus, we assume that $\Theta_{n+1}=\RE$ with initial index $h$ such that $1\leq h\leq n$. In this case, following (A2.5), we only need to confirm that $T_h$ has the same $(\gamma_h,\delta_h)$-exit vertex under $\varphi'$ as under $\varphi$, where $\gamma_h=\gamma_{n+1}$ and $\delta_h=\delta_{n+1}$.
	
	Note that $\delta_h$ is a defective color of $T_{n+1}$ under $\varphi$, because $\Theta_{n+1}=\RE$ and $S_{n+1}=\{\gamma_{n+1},\delta_{n+1}\}$. So $\alpha\neq\delta_h$ and $\alpha\neq\gamma_h$.  Since $T_{h}$ is elementary by (A1) and $|\phibar(T_1)|\geq 5$ by Lemma~\ref{size}, there exists a color $\eta\in\phibar(T_{h})\setminus\{\alpha,\beta, \gamma_h\}$. Moreover, $\eta\neq\delta_h$ as  $\delta_h$ is a defective color of $T_{n+1}$ under $\varphi$. Thus $\{\alpha,\eta\}\cap\{\gamma_h,\delta_h\}=\emptyset$. Let $\pi=\varphi/(G-T_{n+1},\alpha,\eta)$. Clearly, $\pi$ is $(T_{n+1},D_{n+1},\varphi)$-stable and $T_h$ has the same $(\gamma_h,\delta_h)$-exit vertex under $\pi$ as under $\varphi$. So $\pi$ is a $T_{n+1}+f_{n+1}$ mod $\varphi$ coloring. Now $P$ is an $(\eta,\beta)$-path not intersecting $T_{h}$. So by (A3) on $T_h$ and $P$,  $T_h$ has the same $(\gamma_h,\delta_h)$-exit vertex under $\pi/P$ as under $\pi$ and $\varphi$. However, since $\{\alpha,\eta\}\cap\{\gamma_h,\delta_h\}=\emptyset$ and $P$ does not intersect $T_{n+1}$, the $\gamma_h$-colored edges and $\delta_h$-colored edges are the same under $\varphi'=\varphi/P$ as under $\pi/P$. So $T_h$ has the same $(\gamma_h,\delta_h)$-exit vertex under $\varphi'$ as under $\pi/P$ and $\varphi$, as desired.	
\end{proof}

\begin{section}{Elementariness}

In this section, we will prove the elementariness of $T$ if $T$ is an ETT with ladder $T_0\subset  T_1 \subset 	\dots \subset T_n \subset T$ under a $k$-triple $(G,e,\varphi)$ (A1). Let $T-T_{n}$ be the subsequence of $T$ after removing $T_n$. In the proof we will divide $T-T_n$ into a number of subsequences. We call the nested sequence 

{\bf (5.1)} $T_n=T_{n,0}\subset T_{n,1} \subset T_{n,2}\subset 	\dots \subset T_{n,q}\subset T:=T_{n,q+1}$ 

a {\bf hierarchy} with {\bf $q$ levels} of $T$ or simply a hierarchy of $T$. Let $T$ be a tree-sequence of $G$ and let $C$ be a subset of $[k]$. We say that $T$ is 
$C^-$-{\bf closed}
under $\varphi$ if it is closed for all colors in $(\phibar(T)-C)$ under $\varphi$.  For a color $\alpha$, denote by $v(\alpha, T)$ the first vertex of $T$ that misses the color $\alpha$ along $\prec$ if $\alpha\in\phibar(T)$ and the last vertex of $T$ if $\alpha\notin\phibar(T)$. If $T$ is clear, we may simply denote $v(\alpha,T)$ by $v(\alpha)$ or $v_{\alpha}$.
  Recall $D_n=\cup_{m\leq n}S_m-\phibar(T_n)=\{\delta_1,...,\delta_{n}\}-\phibar(T_n)$, where $S_m=\{\gamma_m,\delta_m\}$.  Thus $|D_n|\leq n$. Define $D_{n,j}=D_{n}-\phibar(T_{n,j})$ for each $0\leq j\leq q+1$, so $|D_{n,j}|\leq n$ as well. Reserve $x_e$ for the first vertex in $T$, so $x_e$ is an end of the uncolored edge $e$ and $|\phibar(x_e)|\geq 2$. 

\begin{DEF} \label{R2}
{\rm A hierarchy (5.1) of $T$ is called {\bf good}\label{goodhierarchy} under $\varphi$ if for all $j$ with $0 \le j \le q$ and 
	all $\delta_m\in D_{n,j}$, there exists a $2$-color subset $\Gamma^{j}_m=\{\gamma^{j}_{m_1},\gamma^{j}_{m_2}\}\subseteq\phibar(T_{n,j})\setminus\{\gamma_n,\gamma,\tau\}$ with $\{\gamma,\tau\}\subseteq\phibar(x_e)$ and $\gamma\neq\tau$, such that
	\begin{itemize}
		\vspace{-2mm}
		\item[(i)] $\Gamma^j_m \cap\varphi( T_{n,j+1}(v_{\delta_m})-T_{n,j})=\emptyset$ for $0 \le j \le q$.
		(so colors in $\Gamma^{j}_m$ can only be used on edges in $T_{n,j+1}-T_{n,j}$ after a vertex missing the color $\delta_m$); 
		\vspace{-2mm}
		\item[(ii)] $\Gamma^{j}_{g} \cap \Gamma^{j}_{m}=\emptyset$ whenever $\delta_{g}$ and $\delta_{m}$ are two distinct colors in $D_{n,j}$;
		\vspace{-2mm}
		\item[(iii)] $\Gamma^{j}-\Gamma^{j-1}\subseteq\phibar(T_{n,j}-T_{n,j-1})$ for each $1\leq j\leq q$, where $\Gamma^{j}=\cup_{\delta_m\in D_{n,j}}\Gamma^{j}_m$  and $\Gamma^{j-1}=\cup_{\delta_m\in D_{n,j-1}}\Gamma^{j-1}_m$.
		\vspace{-2mm}
		\item[(iv)] $T_{n,j}$ is $(\cup_{\delta_m\in D_{n,j}}\Gamma^{j-1}_m)^-$-closed under $\varphi$ 
		for all $j$ with $1\le j \le q$.
		\vspace{-2mm}
	\end{itemize}
	\noindent The case $\gamma_n\in\{\gamma,\tau\}$ could happen. The sets $\Gamma^{j}_m$ are referred to as $\Gamma$-{\bf sets}\label{gammasets} for the hierarchy of $T$ under $\varphi$. Let $\Gamma^{j}=\cup_{\delta_m\in D_{n,j}}\Gamma^{j}_m$. We say a hierarchy of $T=T_{n,q+1}$ is {\bf good up to itself} under $\varphi$ if $T_{n,q+1}$ is  $(\cup_{\delta_m\in D_{n,q+1}}\Gamma^{q}_m)^-$-closed. For convenience, we say $T_n$ has a good hierarchy of $(-1)$ levels up to itself under $\varphi$.}
\end{DEF}


\begin{REM}\label{r21}
	Since switching the color $\delta_m$ with another color on a Kempe chain usually creates a non-stable coloring, we may use colors in $\Gamma^{j}_m$ as stepping stones to swap colors while keeping the coloring stable in later proofs. Thus, we may consider the set $\Gamma^{j}_m$ as a color set reserved for $\delta_m$  and (i) as a condition to ensure that such color changes will keep the TAA property for ETTs.  We also notice that (i) actually involves $T_{n,q+1}$ for $j=q$ while (ii), (iii) and (iv) only involve $T_{n,q}$.  	
		
\end{REM}

We divide the proof of (A1) into two statements.

\begin{LEM}(A1)\label{tec}
	Let $n$ be a positive integer. Suppose (A1), (A2), (A2.5),(A3), and (A3.5) hold for ETTs with at most $n-1$ rungs.
If $T$ is an ETT with ladder $T_1 \subset T_2 \subset \dots \subset T_n \subset T\subseteq T_{n+1}$ under $\varphi$, then the following statements A and B hold, which imply $T$ is elementary.
\begin{enumerate}
	
	\item[A.]  Every ETT with $n$ rungs that has a good hierarchy is elementary.
	\item[B.] If A holds, then under $\varphi$, there exists a closed ETT $T'$ having a good hierarchy with $V(T_{n+1})=V(T')$ and ladder $T_1 \subset T_2 \subset \dots \subset T_n \subset T'$.
	
\end{enumerate}	

\end{LEM}	



We prove Statement B first as it is much shorter than the proof of Statement A. 
\subsection{Proof of Statement B in Lemma~\ref{tec}}


We assume Statement A holds.
Let $T$ be an ETT with ladder $T_1 \subset T_2 \subset \dots \subset T_n \subset T\subseteq T_{n+1}$ under $\varphi$.
We will construct an ETT $T'$ with ladder $T_1 \subset T_2 \subset \dots \subset T_n\subset T'$ under $\varphi$ such that $V(T)\subseteq V(T')$ with a good hierarchy $T_n=T_{n,0} \subset T_{n,1}\subset \dots \subset T_{n,q} \subset T':=T_{n,q+1}$. 

Since $T_{n}$ is elementary under $\varphi$ by (A1) and $T_i$ is closed for each $1\leq i\leq n$, each $T_i$ has an odd number of vertices, and therefore $|T_{i}-T_{i-1}|\geq 2$. So by Lemma~\ref{size}, we have $|\phibar(T_{n})|\geq 3+2n$. Since $|D_{n}|\leq|\{\delta_1,\ldots,\delta_n\}|\leq n$, $|D_{n,0}|=|D_{n}-\phibar(T_n)|\leq n$. This allows us to construct $T_{n,1}$ using the following algorithm:


Choose pairwise disjoint sets
\(\Gamma_m^0=\{\gamma_{m1}^0,\gamma_{m2}^0\}\)
for distinct \(\delta_m\in D_{n,0}\) from
\(\phibar(T_n)\setminus\{\gamma_n,\gamma,\tau\}\),
where \(\{\gamma,\tau\}\subseteq\phibar(x_e)\).
 Let $T_{n,1}=T_{n,0}+f_n$, where $f_n$ is the connecting edge of $T_n$.  While there exists an edge $f\in\partial(T_{n,1})$ such that $\varphi(f)\in\phibar(T_{n,1})-\cup_{\delta_m\in D_{n,1}}\Gamma^{0}_m$ where $D_{n,1}=D_n-\phibar(T_{n,1})$, we augment $T_{n,1}$ by updating $T_{n,1}:=T_{n,1}+f$. 

Then the resulting $T_{n,1}$ obtained from this algorithm is $(\cup_{\delta_m\in D_{n,1}}\Gamma^{0}_m)^-$-closed. Moreover, since $\varphi(f_n)=\delta_n\notin\phibar(T_{n})$, the choice of $f$ in the above algorithm guarantees $\Gamma^0_m \cap\varphi( T_{n,1}(v_{\delta_m})-T_{n,0})=\emptyset$. It follows that $T_{n,1}$ satisfies all requirements in Definition~\ref{R2}.

Suppose $T_{n,j-1}$ is defined for some $j\geq 2$. If $T_{n,j-1}$ is closed, then $V(T_{n+1})=V(T_{n,j-1})$ and we let $T_{n,j-1}=T'$. So we assume $T_{n,j-1}$ is not closed.  By (A1) and Statement A, both $T_{n,j-2}$ and $T_{n,j-1}$ are elementary under $\varphi$. Following Definition~\ref{R2} (iv) and the choice of $\Gamma$ sets, $\gamma_n$ is closed for both $T_{n,j-2}$ and $T_{n,j-1}$ under $\varphi$. So $|T_{n,j-2}|$ and $|T_{n,j-1}|$ are both odd, and therefore $|\phibar(T_{n,j-1})-\phibar(T_{n,j-2})|=|\phibar(T_{n,j-1}-T_{n,j-2})|\geq 2$. Now since $T_{n, j-1}$ is $(\cup_{\delta_m\in D_{n,j-1}}\Gamma_m^{j-2})^-$-closed, there exists 
an edge $g\in\partial(T_{n,j-1})$ such that $\varphi(g)\in\Gamma^{j-2}_{k}$ for some $\delta_k\in D_{n,j-1}$.  Now we are ready to construct $T_{n,j}$ using the following algorithm:

 Let $\Gamma^{j-1}_{k}$ be a set of two colors from $\phibar(T_{n,j-1}-T_{n,j-2})$,  and keep $\Gamma^{j-1}_{m}=\Gamma^{j-2}_{m}$ for all other $\delta_{m}\in D_{n,j-1}$ with $\delta_{m}\neq\delta_k$. Let $T_{n,j}=T_{n,j-1}+g$,  where $g$ is defined above.   While there exists $f\in\partial(T_{n,j})$ such that $\varphi(f)\in\phibar(T_{n,j})-\cup_{\delta_m\in D_{n,j}}\Gamma^{j-1}_m$ where $D_{n,j}=D_n-\phibar(T_{n,j})$, we augment $T_{n,j}$ by updating $T_{n,j}:=T_{n,j}+f$. 
 
 Then the resulting  $T_{n,j}$ obtained from this algorithm is $(\cup_{\delta_m\in D_{n,j}}\Gamma^{j-1}_m)^-$-closed. In addition, it is routine to verify that the choices of $g$,  the $\Gamma$ sets, and $f$ give (i), (ii), and (iii) in Definition~\ref{R2}. So $T_{n,j}$ satisfies all requirements in Definition~\ref{R2}. If $T_{n,j}$ is still not closed, we will continue in this fashion to construct $T_{n,j+1}$ incrementally and eventually obtain a closed ETT $T'$ as desired. \qed

\subsection{Proof of Statement A in Lemma~\ref{tec}}

\begin{proof}
We prove Statement A by induction on $q$ which is the number of levels. Denote $T$ by $T_1\subset  T_2 \subset 	\dots \subset T_n:=T_{n,0} \subset T_{n,1}\subset 	\dots \subset T_{n,q}\subset T:=T_{n,q+1}$. When $q=0$, we have $T_{n,q}=T_{n,0}=T_n$. Since $T_n$ is an ETT under $\varphi$ with $n(T_n)=n-1$, it is elementary by (A1) with $(n-1)$ rungs.  Now we assume $T_{n,q}$ is elementary and show $T_{n,q+1}=T$ is. Write $T=(T_{n,q},e_1,y_1,e_{2},...,e_p,y_p)$ following the order $\prec$. We define the {\bf path number $p(T)$} of $T$ as the smallest index $i\in\{1,...,p\}$ such that the sequence $y_iT:=(y_i,e_{i+1},...,e_p,y_p)$ corresponds to a path in $G$.
Suppose on the contrary that  $T$ is a counterexample to the theorem, i.e., $T$ has a good hierarchy with $q$ levels under $\varphi$, but $V(T)$ is not elementary.  Furthermore, we assume that among all counterexamples, the following two conditions hold:
\begin{itemize}
	\item[$(i)$]   $p(T)$ is minimum,
	\item[$(ii)$]   $|T-T_{n,q}|$ is minimum subject to $(i)$, i.e. $p$ is minimum subject to $(i)$.
\end{itemize}

By our choice of counterexamples, 

{\bf (5.0)} $V(T(y_{p-1}))$ is elementary, where $y_0$ is defined as the last vertex of $T_{n,q}$ and $T(y_{0})=T_{n,q}$.

{\bf To simplify notation, we let $\Gamma^{q}_m=\{\gamma_{m1},\gamma_{m2}\}$ for $\delta_m\in D_{n,q}$, and reserve $\gamma,\tau$ for the two colors in $\phibar(x_e)$ in Definition~\ref{R2}. 
 We may often mention that an ETT has a good hierarchy under another coloring in the proof, and unless specified otherwise, we always mean that the hierarchy is good for the same levels and $\Gamma$-sets in $T$.}

\subsubsection{A few basic properties}
\begin{CLA}\label{close}
	Let $\alpha\in \phibar(T_{n,q})$. Then $\alpha$ is closed for $T_{n,j}$, where $j$ is the smallest index with $\alpha\in\phibar(T_{n,j})$. If $r$ is the largest index with $r\leq q$ such that $\alpha$ is closed for $T_{n,r}$, then $\alpha\notin\varphi(T_{n,q}-T_{n,r})$ and $\alpha\in \cup_{\delta_m\in D_{n,k}}\Gamma^{k-1}_m$ for each $k$ with $r< k\leq q$.
\end{CLA}
\begin{proof}
	Let $j$ be the smallest index such that $\alpha\in \phibar(T_{n,j})$. If $j=0$, then $\alpha$ is closed for $T_{n,0}=T_{n}$. In the case that $j\geq 1$,  we have $\alpha\notin\phibar(T_{n,j-1})$. By Definition~\ref{R2}, $\alpha\notin\Gamma^{j-1}$, and therefore $\alpha$ is closed for $T_{n,j}$ by Definition~\ref{R2} (iv). So there exists $r$ as the largest index with $r\leq q$ such that $\alpha$ is closed for $T_{n,r}$. The rest of the current claim holds trivially if $r=q$, so we may assume $r<q$. Since $r$ is the largest index with $r\leq q$ such that $\alpha$ is closed for $T_{n,r}$, for each $k$ with $r< k\leq q$, we have $\alpha\in \cup_{\delta_m\in D_{n,k}}\Gamma^{k-1}_m$ by Definition~\ref{R2} (iv). So following Definition~\ref{R2} (i), we see that $\alpha\notin \varphi(T_{n,k}-T_{n,k-1})$, and therefore $\alpha\notin\varphi(T_{n,q}-T_{n,r})$.
	
\end{proof}	
\begin{CLA}\label{r2stable} Let $\pi$ be $(T_{n,j},D_n,\varphi)$-stable with $0\leq j\leq q$ and $T_n+f_n$ mod $\varphi$. Then $T_{n,j}$ is an ETT that has a good hierarchy up to itself under $\pi$.
\end{CLA}	
\begin{proof} For \(j=0\), the assertion follows from \(T_n+f_n\) mod
	\(\varphi\) and the \((-1)\)-level convention.
	Thus assume \(j\ge1\). Since $\pi$ is $T_n+f_n$ mod $\varphi$, any closure of $T_n+f_n$ is an ETT with $n$ rungs under $\pi$. Because  $\pi$ is also $(T_{n,j}, D_{n},\varphi)$-stable, $T_{n,j}$ can be obtained by TAA from $T_n+f_n$ under $\pi$, and therefore is an ETT with $n$ rungs. The hierarchy property and the fact that $T_{n,j}$ is $(\cup_{\delta_m\in D_{n,j}}\Gamma^{j-1}_m)^-$-closed when $j\geq 1$ follow from Definition~\ref{R2} of $T_{n,q}$.
	\end{proof}

\begin{REM}\label{R2check}
	Let $T$ be an ETT under $\varphi$ with $n$ rungs and a good hierarchy of $j$ levels with $j\leq q$, and let $\pi$ be $T_{n}+f_n$ mod $\varphi$. It is worth pointing out that under $\pi$, if $T_{n,j}$ has a good hierarchy up to itself of $(j-1)$ levels and $T$ is still an ETT, then to check whether $T$ or any tree-sequence obtained from $T$ by TAA has a good hierarchy of $j$ levels under $\pi$,  we only need to check Definition~\ref{R2} (i) for $T_{n,j+1}$ which is $T$ or a tree-sequence obtained from $T_{n,j}$ by TAA, because  Definition~\ref{R2} (ii), (iii) and (iv) do not involve $T_{n,j+1}$. So if $\pi$ is also $(T,D_n,\varphi)$-weakly stable, then following Lemma~\ref{wstable}, $T$ also has a good hierarchy, as it satisfies Definition~\ref{R2} (i) under $\pi$ the same way as under $\varphi$. 
\end{REM}	
\begin{CLA}\label{bstep1}
	Let $T_{n,j}$ be an ETT under $\varphi$ with a good hierarchy $T_n:=T_{n,0} \subset 	\dots\subset T_{n,j-1} \subset T_{n,j}$ with $0\leq j\leq q$. If $T_{n,j}$ has a good hierarchy up to itself when $j\geq 1$, then the following hold:
	\begin{itemize}
		\item[(a)] $\gamma$, $\tau$ and $\gamma_n$ are closed for $T_{n,i}$ under $\varphi$ with $0\leq i\leq j$ (the case $\gamma_n\in\{\gamma,\tau\}$ may happen).
		\item[(b)]  For two colors $\alpha$ and $\beta$, if $\alpha\in\phibar(T_{n,j})\setminus\{\gamma_n\}$ and is closed in $T_{n,j}$, then $\alpha$ and $\beta$ are interchangeable for $T_{n,j}$ under $\varphi$.
	\end{itemize}	
\end{CLA}
\begin{proof}
	Let $T_{n,j}$ be as described above and assume $T_{n,j}$ is 

	(1) $(\cup_{\delta_m\in D_{n,j}}\Gamma^{j-1}_m)^-$-closed when $j\geq 1$.

	We first prove (a). Since $\gamma_n,\gamma,\tau\in\phibar(T_n)$ and $T_{n}$ is closed, the colors $\gamma$, $\tau$ and $\gamma_n$ are all closed for $T_{n,0}$ under $\varphi$. Moreover, $\gamma_n,\gamma,\tau\notin \Gamma_{i-1}$ by the choice of $\Gamma$ sets in Definition~\ref{R2} for each $1\leq i\leq j$. Since $T_{n,i}$ is 
	$(\cup_{\delta_m\in D_{n,i}}\Gamma^{i-1}_m)^-$-closed following Definition~\ref{R2} (iv) and (1), all three colors $\gamma$, $\tau$ and $\gamma_n$  are closed for $T_{n,i}$ under $\varphi$ as well.
	
	We then prove (b). The proof of (b) is actually very similar to the proof of interchangeability from Lemma~\ref{A2} (A2). Among all counterexamples, including all admissible colorings, choose $j$ to be the smallest index such that there exist two $(\alpha,\beta)$-paths $Q_1$ and $Q_2$ intersecting the ETT $T_{n,j}$ as described above, where $\alpha\in\phibar(T_{n,j})\setminus\{\gamma_n\}$ and is closed in $T_{n,j}$.  Since $T_{n,j}$ is elementary by induction hypothesis, paths $Q_1$ and $Q_2$ together yield at least two $(\alpha,\beta)$-exit paths, say $P_1$ with exit vertex $b_1$ and $P_2$  with exit vertex $b_2$, respectively. Moreover, if $\beta\notin\phibar(T_{n,j})$, we have an additional $(\alpha,\beta)$-exit path $P_3$ with exit vertex $b_3$. Renaming subscripts if necessary, we may assume that $b_1\prec b_2$ (and $b_2\prec b_3$ if $P_3$ exists) along $T_{n,j}$.  Note that we have $j\geq 1$, because the case that $j=0$ follows directly from (A2) on $T_n=T_{n,0}$. Let $h$ be the initial index of $T_{n+1}$. So $\gamma_h=\gamma_n$.  For any color $\pi\in\phibar(T_{n,j})$  that is closed for $T_{n,j}$ with $\pi\neq\gamma_h$, let $\varphi_{\pi}=\varphi/(G-T_{n,j},\alpha,\pi)$. We claim that
	
	(2) $\varphi_{\pi}$ is $T_{n}+f_n$ mod $\varphi$, and $T_{n,j}$ is an ETT with $n$ rungs under $\varphi_{\pi}$ that has a good hierarchy up to itself.
	
Since $\alpha$ and $\pi$ are closed for $T_{n,j}$, $\varphi_{\pi}$ is $(T_{n,j},D_n,\varphi)$-stable. Note that $\pi$ and $\alpha$ are different from $\gamma_h=\gamma_n$. So if they are all different from $\delta_h$, clearly $T_{h}$ has the same $(\gamma_{h},\delta_{h})$-exit vertex under $\varphi_{\pi}$ as $\varphi$. If one of them is $\delta_h$, then since  $\alpha$, $\pi$ and $\gamma_h$ (see (a))  are closed for $T_{n,j}$, all the $(\gamma_h,\delta_h)$-chains intersecting $T_h$ are contained in $G[V(T_{n,j})]$, so $T_{h}$ still has the same $(\gamma_{h},\delta_{h})$-exit vertex under $\varphi_{\pi}$ as $\varphi$.  Thus either way, $\varphi_{\pi}$ is $T_n+f_n$ mod $\varphi$ by (A2.5) because $\varphi_{\pi}$ is also $(T_{n}, D_{n},\varphi)$-stable. We then have (2) from Claim~\ref{r2stable}.  

	Without loss of generality, we assume one of $\gamma$ and $\tau$, say $\gamma$, is different from $\gamma_h$. We claim that
	
	(3) We may assume that $b_2\notin V(T_{n,j-1})$, and $P_2$ is a subpath of a $(\gamma,\beta)$-path $P$ not intersecting $T_{n,j-1}$ under $\varphi_{\gamma}$.
	
By (a) and (2), $\varphi_{\gamma}=\varphi/(G-T_{n,j},\alpha,\gamma)$ is $T_{n}+f_n$ mod $\varphi$, and $T_{n,j}$ remains an ETT with a good hierarchy under $\varphi_{\gamma}$. Since $\gamma\neq \gamma_h$, we see that there is at most one $(\gamma,\beta)$-path intersecting $T_{n,j-1}$ by the minimality of $j$. 
If $\beta\in\phibar(T_{n,j})$, then $P_2$ is a subpath of a $(\gamma,\beta)$-path $P$ not intersecting $T_{n,j-1}$ under $\varphi_{\gamma}$, and thus $b_2\notin T_{n,j-1}$ as desired. Now if  $\beta\notin\phibar(T_{n,j})$, then at least one of $P_2$ and $P_3$ ($P_3$ exists when $\beta\notin\phibar(T_{n,j})$) is a subpath of a $(\gamma,\beta)$-path $P$ not intersecting $T_{n,j-1}$ under $\varphi_{\gamma}$. Relabel $P_3$ as $P_2$ and $b_3$ as $b_2$ if necessary. Then (3) holds. Note that we will not use $b_3$ after this point, so $b_3\prec b_2$ after relabeling will not be an issue. 

	We call the tuple $(\varphi,T_{n,j}, \alpha, \beta, P_1,P_2)$ a {\em counterexample} and use ${\cal K}$ 
	to denote the set of all counterexamples to (b)
	with this minimum \(j\), before the additional path choice in (3).  
	For $i=1,2$, let $a_i$  be the end of $P_i$ outside of $T_{n,j}$, 
	and $f_i$ be the boundary edge of $T_{n,j}$ in $P_i$. Define $Q'(T_{n,j})=0$ if $\delta_n=\delta_h\in\phibar(T_{n,j}(b_2)-b_2)$, and $Q'(T_{n,j})=1$ otherwise.
	With a slight abuse of notation, let $(\varphi, T_{n,j}, \alpha, \beta, P_1,P_2)$ be a counterexample in ${\cal K}$ such that:
	\begin{itemize}
		
		\item [(i)] $Q'(T_{n,j})$ is minimum, and
		\item [(ii)] $|P_1|+|P_2|$ is minimum subject to (i).
	\end{itemize}	

Let $\eta$ be a color in $\phibar(b_2)$. 	Depending on whether $\beta$ is $\delta_h$ or not, we consider the following two cases.

{\flushleft {\bf Case I:}}  $\beta\neq\delta_h$.

Since $T_{n,j}$ is elementary by the induction hypothesis and $\gamma_h\in\phibar(T_{h})$, by (3), we have

(4) $\eta$ is different from $\gamma_n=\gamma_h$. 

Let $\varphi_{\eta}=\varphi/(G-T_{n,j},\alpha,\eta)$. Note that $P_2$ is an $(\eta,\beta)$-path under $\varphi_{\eta}$ with $\eta\in\overline{\varphi}_{\eta}(b_2)$, and $f_i$ is colored 
by $\beta$ under $\varphi_{\eta}$ for $i=1,2$.  
Consider $\sigma=\varphi_{\eta}/P_2$. So $\beta \in \overline{\sigma}(b_2)$, and $f_1$ is still colored 
by $\beta$ under $\sigma$.  We claim that 

(5) $\sigma$ is $T_{n}+f_n$ mod $\varphi_{\eta}$, and $T_{n,j}(b_2)$ is an ETT with $n$ rungs and a good hierarchy of $(j-1)$ levels under $\sigma$.

Since $\sigma$ is $(T_{n,j}(b_2)-b_2,D_n,\varphi_{\eta})$-stable, $T_{n,j}(b_2)$ is an ETT with $n$ rungs under $\sigma$ by (2), (4), and  Corollary~\ref{vstable},  once we show that $\sigma$ is $T_{n}+f_n$ mod $\varphi_{\eta}$. Indeed, because $\beta$ is not closed for $T_{n,j}$ under $\varphi$, we have $\beta\neq\gamma_h$ by (a). So by (4), both $\eta$ and $\beta$ are different from $\gamma_h$. Recall that we have $\beta\neq\delta_h$ in the current case. So if $\eta\neq\delta_h$, then $T_{h}$ has the same $(\gamma_{h},\delta_{h})$-exit vertex under $\sigma$ as $\varphi_{\eta}$, and thus $\sigma$ is $T_{n}+f_n$ mod $\varphi_{\eta}$ by (A2.5).  If $\eta=\delta_h$, then since $b_2\notin T_{n,j-1}$ by (3) and the elementariness of $T_{n,j}$ under $\varphi$, we have $\delta_h\notin\phibar(T_{n,j-1})$. Thus $\delta_h,\gamma_h\notin\Gamma^{j-1}$ by Definition~\ref{R2}. So all the $(\gamma_h,\delta_h)$-ears of $T_h$ under $\varphi_{\eta}$ are contained in $G[V(T_{n,j})]$ by Definition~\ref{R2} (iv), (2) and (4). Thus by the construction of $\sigma$, they are colored the same under $\sigma$. Since $T_h$ only has one $(\gamma_h,\delta_h)$-exit path under $\varphi_{\eta}$ by (A2),  $T_{h}$ still has the same $(\gamma_{h},\delta_{h})$-exit vertex under $\sigma$ as $\varphi_{\eta}$, and thus $\sigma$ is $T_{n}+f_n$ mod $\varphi_{\eta}$ by (A2.5). It is routine to see that $T_{n,j}(b_2)$ has a good hierarchy of $(j-1)$ levels; see Remark~\ref{R2check}.

So $T_{n,j}(b_2)$ is elementary under $\sigma$ following (5) and the induction hypothesis. Thus $\beta\notin\sigmabar(T_{n,j}(b_2)-b_2)$ by the construction of $\sigma$. Moreover, we have
 
 (6) $\beta\notin\phibar(T_{n,j}(b_2))$, and consequently $\beta\notin\Gamma^{j-1}$ by Definition~\ref{R2} and (3).
 
 
 Let $T'=T_{n,j}(b_2)+f_1$ under $\sigma$. Because $\beta\notin\Gamma^{j-1}$ by (6), we see that $T'$ satisfies Definition~\ref{R2} (i), and thus $T'$ has a good hierarchy of $(j-1)$ levels; see Remark~\ref{R2check}. Let $T'_{n,j}$ be a tree-sequence obtained from $T'$ by the following algorithm:

while there exists $f\in\partial(T'_{n,j})$ such that $\sigma(f)\in\sigmabar(T'_{n,j})-\cup_{\delta_m\in D_{n,j}}\Gamma^{j-1}_m$ where $D_{n,j}=D_n-\sigmabar(T'_{n,j})$, we augment $T'_{n,j}$ by updating $T'_{n,j}:=T'_{n,j}+f$. 

Note that this is exactly the same algorithm as in the proof of Statement B. It is routine to check that 

(7) $T'_{n,j}$ is an ETT with a good hierarchy of $(j-1)$ levels under $\sigma$ and $T'_{n,j}$ is $(\cup_{\delta_m\in D_{n,j}}\Gamma^{j-1}_m)^-$-closed.

Observe that neither $a_1$ nor $a_2$ is contained in $T'_{n,j}$, for otherwise, let $a_i \in V(T'_{n,j})$ for some $i$
with $1\le i \le 2$. By (7) and the induction hypothesis, we see that $T'_{n,j}$ is elementary under $\sigma$. Since $\{\beta,\eta\}\cap \overline{\sigma}(a_i) \ne \emptyset$ and $\beta \in \overline{\sigma}(b_2)$, 
we obtain $\eta \in \overline{\sigma}(a_i)$. Since $\eta\in\phibar(b_2)$ and $b_2\notin T_{n,j-1}$ by (3), we have $\eta\notin\phibar(T_{n,j-1})$. So $\eta\notin \Gamma^{j-1}$ as $\Gamma^{j-1}\subseteq\phibar(T_{n,j-1})$ following Definition~\ref{R2}. Recall that we also have $\beta\notin\Gamma^{j-1}$ by (6). So $P_1,P_2$ and $a_1, a_2$ are all entirely contained in $G[V(T'_{n,j})]$, as  $T'_{n,j}$ is $(\cup_{\delta_m\in D_{n,j}}\Gamma^{j-1}_m)^-$-closed by (7) and $(\cup_{\delta_m\in D_{n,j}}\Gamma^{j-1}_m)\subseteq \Gamma^{j-1}$.
However, because $\{\beta,\eta\}\cap \overline{\sigma}(a_k) \ne \emptyset$  for $k=1,2$, we see that $V(T'_{n,j})$ is not 
elementary with respect to $\sigma$, a contradiction to the elementariness we obtained earlier. So each $P_i$ contains a subpath $L_i$ with $i=1,2$, which is an exit path of  $T'_{n,j}$
under $\sigma$. Since $f_1\in T'_{n,j}$, we obtain $|L_1|+|L_2|<|P_1|+|P_2|$. Moreover, we observe that $Q'(T'_{n,j})\leq Q'(T_{n,j})$, and $\beta\neq\gamma_n$ as $\gamma_n\in\phibar(T_{n,j-1})$ (see (6)).
Thus the existence of the counterexample $(\sigma, T'_{n,j}, \beta, \eta, L_1,L_2)$ with $\beta\neq\gamma_n$ violates the minimality 
assumption on $(\varphi, T_{n,j}, \alpha, \beta, P_1,P_2)$.

{\flushleft {\bf Case II:}} $\beta=\delta_h$.

In this case, we have $h\geq 1$. As in the previous case, we have $\eta\neq\gamma_h$, and we consider $\varphi_{\eta}=\varphi/(G-T_{n,j},\alpha,\eta)$. Since $\eta\neq\gamma_h$, in view of (2), we have 

(8) $\varphi_{\eta}$ is $T_{n}+f_n$ mod $\varphi$, and therefore $T_{n,j}$ is an ETT with a good hierarchy of $(j-1)$ levels under $\varphi_{\eta}$.

Observe that $P_2$ is an $(\eta,\beta)$-path under $\varphi_{\eta}$ with $\eta\in\overline{\varphi}_{\eta}(b_2)$. As in the previous case, we let $\sigma=\varphi_{\eta}/P_2$. We claim that 

(9) $T_h$ has the same $(\gamma_h,\delta_h)$-exit vertex under $\sigma$ as $\varphi_{\eta}$, and therefore $\sigma$ is  $T_{n}+f_n$ mod $\varphi_{\eta}$ (by (A2.5) on $T_n$).

Note that following (9), with the same argument as showing (6), we have $\beta\notin\phibar(T_{n,j}(b_2))$ (so $Q'(T_{n,j})=1$ in this case), $\beta\notin\Gamma^{j-1}$, and $T_{n,j}(b_2)$ is an ETT with a good hierarchy of $(j-1)$ levels, and therefore, as before, neither $a_1$ nor $a_2$ is in $T'_{n,j}$, which is also an ETT with a good hierarchy of $(j-1)$ levels ($T'_{n,j}$ is constructed the same way as in Case I). So $P_1$ and $P_2$ give us two $(\beta,\eta)$-paths intersecting $T'_{n,j}$, and they provide two $(\beta,\eta)$-exit paths $L_1$ and $L_2$ of $T'_{n,j}$, respectively. Moreover, $L_1$ provides an exit vertex in $T'_{n,j}$ after $b_2$ along $\prec$ of $T'_{n,j}$, as $T_{n,j}(b_2)+f_1\subseteq T'_{n,j}$. Since $\delta_h=\beta\in\overline{\sigma}(b_2)$,  $(\sigma, T'_{n,j}, \beta, \eta, L_1,L_2)$ with $\beta\neq\gamma_n$ provides a counterexample with $Q'(T'_{n,j})=0$, a contradiction.

Following (a) and (3), the proof of (9) is almost identical to the proof of (7) in Lemma~\ref{A2} with $T_{n,j}$ in place of $T_{n+1}$. Here we give a summary of the actual algorithm used here. Suppose that $T_h$ has a different $(\gamma_h,\delta_h)$-exit vertex $v_1$ under $\sigma$ from the $(\gamma_h,\delta_h)$-exit vertex $v_2$ under $\varphi_{\eta}$, where $\sigma=\varphi_{\eta}/P_2$. In the case that $v_1\prec v_2$, we let $n'$ be the smallest index such that $v_2\in T_{n'}$ and $T_{n'}$ is $\RE$ finished. After the same exterior interchange at \(v_2\) as in the
proof of (7) in Lemma~\ref{A2}, Lemma~\ref{A3.0} gives a contradiction
with the new color missing at \(v_2\) in place of \(\gamma_h\).


In the case that $v_2\prec v_1$, we let $n'$ be the smallest index such that $v_1\in T_{n'}$ and $T_{n'}$ is $\RE$ finished. Let $h'$ be the initial index of $T_{n'}$. Pick $\alpha'\in\phibar(v_1)$.  Recall $\gamma\neq\gamma_h$ (see the assumption before (3)). Following (3),  we let $\mu$ be a coloring obtained from $\sigma$ by interchanging the labels $\gamma$ and $\tau$ if $\gamma\neq\gamma_{h'}$, and  let $\mu=\sigma/(G-T_{h},\gamma,\tau)$ if $\gamma=\gamma_{h'}$. Either way, we have $\tau\neq\gamma_{h'}$, and $\mu(P)\subseteq\{\tau,\eta,\delta_h\}$. We then let $\mu_1=\mu/(G-T_{h},\alpha',\gamma_h)$ if $\tau\neq\gamma_h$, and $\mu_1=\mu/(G-T_{h},\alpha',\gamma)$ otherwise. Let $\mu_2$ be obtained from $\mu_1$ by switching all the $\eta$-colored edges of $P$ not in $P_{x_e}(\eta,\tau,\mu_1)$ to $\tau$ by Kempe changes. Then under $\mu_2$, we reach a contradiction to Lemma~\ref{A3.2} with $T_{n'}$ ($n'\leq n$) in place of $T_{n+1}$, $\alpha'$ in place of $\alpha$, $\delta_h$ in place of $\delta$, $x_e$ in place of $r$, $v_1$ in place of $x$, $v_2$ in place of $y$, and $b_2$ in place of $w$.
\end{proof}

\begin{CLA}\label{colorchange}
	Let $\alpha\in\phibar(T_{n,j})$ for some $0\leq j\leq q$ with $\alpha\neq\gamma_n$, and $\beta$ be any color in $[k]$. Suppose $P$ is an $(\alpha,\beta)$-path not intersecting $T_{n,j}$. If $\pi=\varphi/P$, then $\pi$ is $T_{n}+f_n$ mod $\varphi$. 
\end{CLA}

\begin{proof} Let $h$ be the initial index for $\Theta_{n}$ and $\pi=\varphi/P$. Then $\gamma_n=\gamma_h$, $\delta_n=\delta_h$, and $\pi$ is $(T_n,D_n,\varphi)$-stable. We propose to show
		
	(1)	$T_h$ has the same $(\gamma_h,\delta_h)$-exit vertex under $\pi$ as under $\varphi$.
	
	Note that by assuming (1), we immediately see that $\pi$ is $T_{n}+f_n$ mod $\varphi$ by (A2.5) for $T_{n+1}$ in place of $T_{n+2}$. 
	
	In the case that $\alpha\neq\delta_h$ and $\beta\notin S_h=\{\gamma_h,\delta_h\}$, (1) holds trivially.  If $\alpha=\delta_h$, let $r\le j$ be the smallest index such that
	$\delta_h\in\overline{\varphi}(T_{n,r})$. Following Claims~\ref{close} and~\ref{bstep1}(a),
	both $\delta_h$ and $\gamma_h=\gamma_n$ are closed for $T_{n,r}$. Thus all
	the $(\gamma_h,\delta_h)$-chains intersecting $T_h$ are contained in
	$G[V(T_{n,r})]\subseteq G[V(T_{n,j})]$. Since $P$ does not intersect
	$T_{n,j}$, (1) holds.
	

	Thus we may assume $\alpha\notin S_h$ and $\beta\in S_h$. In this case, we may assume one of $\gamma$ and $\tau$, say $\gamma\neq\gamma_h$. Moreover, since $\gamma\in\phibar(x_e)$, $\gamma\neq\delta_h$. Following Claim~\ref{close}, $\alpha$ is closed for some $T_{n,j'}$ with $0 \leq j'\leq j$. With a slight abuse of notation with $j$ denoting $j'$, we may assume that $\alpha$ is closed for $T_{n,j}$. By Claim~\ref{bstep1} (a), $T_{n,j}$ is also closed for $\gamma$. So $\sigma=\varphi/(G-T_{n,j},\alpha,\gamma)$ is also $(T_{n,j}, D_n,\varphi)$-stable.  As $\gamma,\alpha\notin S_h$,  $T_h$ has the same $(\gamma_h,\delta_h)$-exit vertex under $\sigma$ as under $\varphi$.  So $\sigma$ is also $T_{n}+f_n$ mod $\varphi$ by (A2.5). Note that under $\sigma$, $P$ is a $(\gamma,\beta)$-path not intersecting $T_{n,j}$. So by (A3.5) for $T_n$, $\pi_1=\sigma/P$ is $T_{n}+f_n$ mod $\sigma$, and therefore is  $T_{n}+f_n$ mod $\varphi$ too. Since $\pi=\pi_1/(G-T_{n,j},\alpha,\gamma)$, we similarly see that $\pi$ is  $T_{n}+f_n$ mod $\pi_1$ as $\gamma,\alpha\notin S_h$. So $\pi$ is  $T_{n}+f_n$ mod $\varphi$, as desired. 
	
\end{proof}

\begin{CLA}\label{b9n}
	For any $y\in V(T(y_{p-1})-T_{n,q})$, $|(\phibar(T(y))\setminus\{\gamma_n\})\setminus\varphi(T(y) - T_{n,q})|\geq 2+2n$.
	Furthermore, for $i=1,2,3$, there exist distinct colors $\beta_i\in\phibar(T(y))\setminus\{\gamma_n\}$ with $\beta_i\notin \varphi(T(y)-T_{n,q})$ such that either $\beta_i\notin\Gamma^q\cup D_{n,q}$ or $\beta_i\in\Gamma^q_r$ for some $\delta_r\in D_{n,q}\cap\phibar(T(y))$ (here $\delta_r$ might be different for different $i$). Consequently, for any $y=y_k$ with $0\leq k\leq p-2$, there exists a color $\beta\in\phibar(T(y))$ with $\beta\notin \varphi(T(y_{k+2})-T_{n,q})$ such that either $\beta\notin\Gamma^q\cup D_{n,q}\cup\{\gamma_n\}$ or $\beta\in\Gamma^q_r$ for some $\delta_r\in D_{n,q}\cap\phibar(T(y))$.
	
\end{CLA}

\begin{proof}  Since under $\varphi$, $T_n$ is elementary following (5.0) and $T_{l}$ is closed for each $l\in[n]$, $|T_l|$ is odd, and thus $|\phibar(T_n)\setminus\{\gamma_n\}|\geq 2n+2$ following Lemma~\ref{size}. So $|\phibar(T_{n,q})\setminus\{\gamma_n\}|\geq 2n+3$ when $q\geq 1$. Because $|V(T(y)-T_{n,q})|=|E(T(y)-T_{n,q})|$,  following (5.0), we have $|\phibar (T(y) -T_{n,q})| \ge |\varphi(T(y) - T_{n,q})|$. Note that in the case that $q=0$, the connecting edge $f_n$ is colored by $\delta_n\notin\phibar(T_{n})$, so $|\varphi(T(y) - T_{n,q})\cap\phibar(T(y))|<|\varphi(T(y) - T_{n,q})|$ unless $\delta_n\in\phibar(T(y))$. Therefore,
	$ |(\overline{\varphi}(T(y))\setminus\{\gamma_n\}) \setminus\varphi(T(y) -T_{n,q}) | \ge |\phibar(T_{n,q})\setminus\{\gamma_n\}|  \ge |\phibar(T_n)\setminus\{\gamma_n\}|\geq 2n+2$, and equality holds only if $q=0$ and $\delta_n\in\phibar(T(y))$.

	Now let us show the ``furthermore" part. Denote $\overline{\varphi}(T(y))\setminus\{\gamma_n\}$ by $C$. Now let  $|C\setminus ( \Gamma^{q} \cup D_{n,q} \cup\varphi(T(y) - T_{n,q}))|=d$. Clearly we have as desired if $d\geq 3$, so we may assume $d<3$.
	Since \[
	C\setminus \varphi(T(y)-T_{n,q})  = (C  \setminus (\Gamma^{q}\cup D_{n,q} \cup \varphi(T(y) - T_{n,q}))) \cup
	(((\Gamma^q\cup D_{n,q})\cap C) \setminus (\varphi(T(y) - T_{n,q})\cap C)),
	\]
	we have
	\[
	|((\Gamma^{q}\cup D_{n,q})\cap C) \setminus (\varphi(T(y) - T_{n,q})\cap C)| \ge |C\setminus \varphi(T(y)-T_{n,q})| - |C\setminus (\Gamma^{q} \cup D_{n,q} \cup\varphi(T(y) - T_{n,q}))| \ge 2n + 2-d,   
	\]
	and equality holds only if $q=0$ and $\delta_n\in\phibar(T(y))$. Since $|D_{n,q}|\leq n$ and $d<3$, by the Pigeonhole Principle with $|D_{n,q}|$ many pigeonholes, there are at least $(3-d)$ distinct colors $\beta\in\Gamma^q_r\cap C$ such that  $\beta\notin\varphi(T(y) - T_{n,q})$ with $\delta_r\in D_{n,q}\cap\phibar(T(y))$ (different $\beta$ can have different $\delta_r$). Together with $|C\setminus (\Gamma^{q} \cup D_{n,q} \cup\varphi(T(y) - T_{n,q}))|=d$, the ``furthermore" part holds.
	
	Finally we prove the ``consequently" part of the current claim. Since there are only two edges in $T(y_{k+2})-T(y_k)$, clearly it holds by the ``furthermore" part if $k\geq 1$. So we assume $k=0$, that is, $y=y_0$ being the last vertex of $T_{n,q}$. Note that if $q\geq 1$, we have $|\phibar(T_{n,q})\setminus\{\gamma_n\}|\geq 2n+3$. Since there are only two edges in $T(y_2)-T_{n,q}$, $D_{n,q}\cap\phibar(T_{n,q})=\emptyset$, and $|\Gamma^q|\leq 2n$, we have a color $\beta\in\phibar(T_{n,q})$ with $\beta\notin \varphi(T(y_{2})-T_{n,q})$ such that $\beta\notin\Gamma^q\cup D_{n,q}\cup \{\gamma_n\}$, as desired. So $q=0$, and we have $|\phibar(T_{n,q})\setminus\{\gamma_n\}|\geq 2n+2$ in this case. However, in this case the first edge on $T(y_{2})-T_{n,q}$ is colored by $\delta_n\notin\phibar(T_{n})$. So again, since  there are only two edges in $T(y_2)-T_{n,q}$, $D_{n,q}\cap\phibar(T_{n,q})=\emptyset$, and $|\Gamma^q|\leq 2n$, we also have a desired color $\beta\in\phibar(T_{n,q})$ with $\beta\notin \varphi(T(y_{2})-T_{n,q})$ such that $\beta\notin\Gamma^q\cup D_{n,q}\cup \{\gamma_n\}$.
\end{proof}

\begin{CLA}\label{bchange}
	Let $\alpha$ and $\beta$ be colors in $\phibar(T(y_{p-1}))$  with
	$v(\alpha) \prec v(\beta)$ and $\alpha\notin\varphi(T(v(\beta)) -  T_{n,q})$ ($T(v(\beta))-T_{n,q}=\emptyset$ if $v(\beta)\in T_{n,q}$).
	If $\alpha\in \phibar(T_{n,q})-\{\gamma_n\}$, or $\alpha, \beta\notin D_{n,q}\cup\{\gamma_n\}$, or $\alpha,\beta$ are both closed for $T_{n,q}$, then $P_{v(\alpha)}(\alpha,\beta,\varphi)=P_{v(\beta)}(\alpha,\beta,\varphi)$. Additionally, if $\alpha\in \phibar(T_{n,q})$ and $\alpha$ is closed for $T_{n,q}$, then
	$P_{v(\alpha)}(\alpha,\beta,\varphi)$ is the only $(\alpha,\beta)$-{\bf path}  intersecting $T_{n,q}$.
\end{CLA}

\begin{proof}
	Let $v(\alpha)=u$ and $v(\beta)=w$. We consider the following few cases.
	{\flushleft \bf Case I: 	$u,w\in T_{n,q}$. }

	If $T_{n,q}$ is closed for both $\alpha,\beta$, then no boundary edges of $T_{n,q}$ are colored by $\alpha$ or $\beta$. Thus we have $P_{u}(\alpha,\beta,\varphi)=P_{w}(\alpha,\beta,\varphi)$, because $T_{n,q}$ is elementary by (5.0). Moreover, this is the only $(\alpha,\beta)$-path intersecting $T_{n,q}$, so Claim~\ref{bchange} holds in this case. 
	
	Now we suppose $T_{n,q}$ is closed for $\alpha$ or $\beta$ but not both. We first consider the case $\alpha,\beta\neq \gamma_n$. By Claim~\ref{bstep1} (could be applied on $\beta$ in place of $\alpha$ when $\beta$ is closed for $T_{n,q}$), there is at most one $(\alpha,\beta)$-path intersecting $T_{n,q}$. Thus $P_{u}(\alpha,\beta,\varphi)= P_{w}(\alpha,\beta,\varphi)$ is the unique $(\alpha,\beta)$-path intersecting $T_{n,q}$, as desired. We then assume $\beta=\gamma_n$. As $\gamma_n$ is closed for $T_{n,q}$ by Claim~\ref{bstep1}, $\alpha$ is not closed for $T_{n,q}$ by assumption. So we only need to show $P_{u}(\alpha,\beta,\varphi)= P_{w}(\alpha,\beta,\varphi)$. By Claim~\ref{close}, $\alpha$ is closed for $T_{n,j}$, where $j$ is the smallest index with $\alpha\in\phibar(T_{n,j})$. Thus both $\alpha$ and $\beta=\gamma_n$ are closed for $T_{n,j}$ by Claim~\ref{bstep1}, and therefore $P_{u}(\alpha,\beta,\varphi)=P_{w}(\alpha,\beta,\varphi)$ following (5.0).

	 We now assume neither $\alpha$ nor $\beta$ is $T_{n,q}$-closed. As $\gamma_n$ is closed for $T_{n,q}$ by Claim~\ref{bstep1}, we have both $\alpha,\beta\neq \gamma_n$.  In this case, we only need to show that $P_{u}(\alpha,\beta,\varphi)=P_{w}(\alpha,\beta,\varphi)$. By Claim~\ref{close}, $\beta$ is closed for some $T_{n,j}$ with $\beta\in\phibar(T_{n,j})$ and $0\leq j\leq q$. Because $u\prec w$, we have $u,w\in T_{n,j}$. Again by Claim~\ref{bstep1}, we can see that $P_{u}(\alpha,\beta,\varphi)= P_{w}(\alpha,\beta,\varphi)$ which is the unique $(\alpha,\beta)$-path intersecting $T_{n,j}$.

	{\flushleft {\bf Case II: $w\notin T_{n,q}$} and $u\in T_{n,q}$. }
	
	In this case $\alpha\notin\varphi(T(w) - T_{n,q})$, $\beta\neq\gamma_n$, and $\beta$ is not closed for $T_{n,q}$. So we have $\alpha\neq\gamma_n$. We first assume that  $\alpha$ is closed in $T_{n,q}$. By Claim~\ref{bstep1}, there is exactly one $(\alpha,\beta)$-path intersecting $T_{n,q}$, which is $P_{u}(\alpha,\beta,\varphi)$. Assume otherwise that $P_{u}(\alpha,\beta,\varphi)\neq P_{w}(\alpha,\beta,\varphi)$. Then $P_{w}(\alpha,\beta,\varphi)$ does not intersect $T_{n,q}$. Thus $\pi=\varphi/P_{w}(\alpha,\beta,\varphi)$ is $(T_{n,q}, D_n,\varphi)$-stable. Moreover, it is $T_n+f_n$ mod $\varphi$ by Claim~\ref{colorchange}, since $\alpha\neq\gamma_n$. Therefore, by Claim~\ref{r2stable}, $T_{n,q}$ is an ETT with a good hierarchy up to itself under $\pi$. Note that following TAA, the only case that $\beta\in\varphi(T(w) - T_{n,q})$ is  $\beta=\delta_n$ and $q=0$. Thus because $\pi$ is $(T_{n,q}, D_n,\varphi)$-stable and $\alpha\notin\varphi(T(w) - T_{n,q})$, $\pi$ is also $(T(w),D_n,\varphi)$-weakly stable. So $T(w)$ remains an ETT under $\pi$ by Lemma~\ref{wstable}.
		Following Remark~\ref{R2check}, $T(w)$  also has a good hierarchy under $\pi$. However, this is a contradiction to (5.0) as $\alpha\in\pibar(w)\cap \pibar(T(u))$.

	Thus $\alpha$ is not closed in $T_{n,q}$. In this case we only need to prove $P_{u}(\alpha,\beta,\varphi)=P_{w}(\alpha,\beta,\varphi)$. Assume otherwise. Following Claim~\ref{close}, we have $\alpha\notin\varphi(T_{n,q}-T_{n,r})$, where $r$ is the largest index with $r\leq q$ such that $\alpha$ is closed for $T_{n,r}$, and 
	
	(1) $\alpha\in \cup_{\delta_m\in D_{n,k}}\Gamma^{k-1}_m$ for each $k$ with $r< k\leq q$.
	
	 By the same argument as above for $T_{n,r}$ using Claim~\ref{bstep1}, we see that  $P_{u}(\alpha,\beta,\varphi)$ is the only $(\alpha,\beta)$-path intersecting $T_{n,r}$, and $\pi=\varphi/P_{w}(\alpha,\beta,\varphi)$ is $(T_{n,r}, D_n,\varphi)$-stable. Similarly following  Claim~\ref{colorchange} and Claim~\ref{r2stable}, $T_{n,r}$ is an ETT with a good hierarchy up to itself under $\pi$. Since
	 \(\alpha,\beta\notin\varphi(T(w)-T_{n,r}-f_n)\)
	 and \(\pi(f_n)=\varphi(f_n)\), we again see that
	 \(\pi\) is \((T(w),D_n,\varphi)\)-weakly stable. So the hierarchy of $T(w)$ satisfies Definition~\ref{R2} (i), (ii) and (iii) under $\pi$ the same way as it is under $\varphi$. Following (1), we see that $T(w)$ also satisfies Definition~\ref{R2} (iv), so $T(w)$ has a good hierarchy of $q$ levels.  Again, this is a contradiction to (5.0) as $\alpha\in\pibar(w)\cap \pibar(u)$.

	{\flushleft \bf Case III: 	$u,w\notin T_{n,q}$.}
	
	In this case,  by (5.0), we have  $\alpha\notin\varphi(T(w) -  T_{n,q})$ and $\alpha,\beta \notin D_{n,q}\cup\phibar(T_{n,q})$ . So $\alpha,\beta\notin\varphi(T(w))$.    Suppose on the contrary that $P_{u}(\alpha,\beta,\varphi)\neq P_{w}(\alpha,\beta,\varphi)$. Then it is easy to see that $\pi=\varphi/P_{w}(\alpha,\beta,\varphi)$ is $(T_{n,q},D_n,\varphi)$-stable and $(T(w), D_n,\varphi)$-weakly stable.  Moreover, $T_h$ has the same $(\gamma_h,\delta_h)$-exit vertex under $\pi$ as under $\varphi$, where $h$ is the initial index of $\Theta_n$. So $\pi$ is $T_n+f_n$ mod $\varphi$. Thus following Lemma~\ref{wstable}, Claim~\ref{r2stable} and Remark~\ref{R2check}, we see that $T(w)$ has a good hierarchy with $q$ levels, reaching a contradiction to (5.0) as $\alpha\in\pibar(w)\cap \pibar(T(u))$.
\end{proof}

\begin{CLA}\label{bstablechange}
	For any two colors $\alpha,\beta\in\phibar(T(y_{p-1}))$, the following two statements hold.
	\begin{itemize}
		\item [(i)]	Let $P$ be an $(\alpha,\beta)$-path other than $P_{v(\alpha)}(\alpha,\beta,\varphi)$. If $\alpha\in\phibar(T_{n,q})-\{\gamma_n\}$ and $\beta\neq\gamma_n$, or both $\alpha$ and $\beta$ are closed for $T_{n,q}$, then  $\pi=\varphi/P$ is $T_n+f_n$ mod $\varphi$, $(T_{n,q},D_n,\varphi)$-weakly stable, and $T_{n,q}$ has a good hierarchy of $(q-1)$ levels up to itself under $\pi$.
		\item [(ii)] If $T_{n,q}\prec v(\alpha) \prec v(\beta)$, $\alpha\notin \phiv(T(v(\beta))-T_{n,q})$ and $\alpha, \beta \notin D_{n,q}$, then $\pi=\phiv/P$ is $(T_{n,q},D_n,\varphi)$-stable and $T_n+f_n$ mod $\varphi$ for any $(\alpha, \beta)$-chain $P$. Moreover, $T$ is an ETT with a good hierarchy of $q$ levels under $\pi$. 
	\end{itemize}
\end{CLA}

\begin{proof}
	We first prove (i).
	If one of $\alpha$ and $\beta$ is closed in $T_{n,q}$, then $P$ does not intersect $T_{n,q}$ by Claim~\ref{bstep1}(b). So $\pi=\varphi/P$ is $(T_{n,q},D_n,\varphi)$-stable. Moreover, it is $T_n+f_n$ mod $\varphi$ by Claim~\ref{colorchange}. So $T_{n,q}$ has a good hierarchy up to itself by Claim~\ref{r2stable}. 
	
We then assume neither $\alpha$ nor $\beta$ is closed in $T_{n,q}$. Because $\alpha\in\phibar(T_{n,q})$, by Claim~\ref{close}, there exists a largest index $r$ such that $\alpha$ or $\beta$ is closed in $T_{n,r}$.  Thus by Claim~\ref{bstep1} (b) (could be applied on $\beta$ in place of $\alpha$ when $\beta$ is closed for $T_{n,r}$), $P$ does not intersect $T_{n,r}$, and $\pi=\varphi/P$ is $(T_{n,r},D_n,\varphi)$-stable. Therefore, it is $T_n+f_n$ mod $\varphi$ by Claim~\ref{colorchange}, and $T_{n,r}$ has a good hierarchy up to itself by Claim~\ref{r2stable}. We claim that  
	
	(1) $\alpha,\beta\notin\varphi(T_{n,q}-T_{n,r})$, unless $r=0$, $\beta=\delta_n$ with $\beta\notin\phibar(T_{n,q})$, and $f_n$ is the only edge colored by $\beta$ in $E(T_{n,q}-T_{n,r})$ in this case.
	
	Indeed, if $\beta\in\phibar(T_{n,q})$, then we have (1) immediately by Claim~\ref{close} by our choice of $r$. So $\beta\notin\phibar(T_{n,q})$. So following Algorithm 3.1 and TAA, the only way that $\beta\in\varphi(T_{n,q}-T_{n,r})$ happens is $r=0$ and $\beta=\delta_n$ with $\beta\notin\phibar(T_{n,q})$, and in this case $f_n$ is the only edge colored by $\beta$ in $E(T_{n,q}-T_{n,r})$. Note that we also have $\alpha \notin\varphi(T_{n,q}-T_{n,r})$ by Claim~\ref{close} with our choice of $r$. So (1) is established.
	
	Since $\pi$ is $(T_{n,r},D_n,\varphi)$-stable, we see that $\pi$ is $(T_{n,q},D_n,\varphi)$-weakly stable by (1). Because neither $\alpha$ nor $\beta$ is closed for $T_{n,k}$ for each $r<k\leq q$ by our choice of $r$, it is then routine to check that $T_{n,q}$ indeed has a good hierarchy of $(q-1)$ levels up to itself under $\pi$.
	
	Now we prove (ii). Let $u=v(\alpha)$ and $w=v(\beta)$. So  $\alpha\notin\varphi(T(w) -  T_{n,q})$ and $\alpha,\beta \notin D_{n,q}\cup\phibar(T_{n,q})$. Thus $\alpha,\beta\notin\varphi(T(w))$ and $\alpha,\beta\notin \Gamma^{q}$.  By Claim~\ref{bchange}, $P_w(\alpha,\beta,\varphi)=P_u(\alpha,\beta,\varphi)$. For any $(\alpha,\beta)$-chain $P$, it is easy to see that  $\pi=\varphi/P$ is $(T_{n,q},D_n,\varphi)$-stable and $(T(u)-u, D_n,\varphi)$-weakly stable. Moreover, $T_h$ has the same $(\gamma_h,\delta_h)$-exit vertex under $\pi$ as under $\varphi$, where $h$ is the initial index of $\Theta_n$. So $\pi$ is $T_n+f_n$ mod $\varphi$. Thus $T_{n,q}$ has a good hierarchy up to itself under $\pi$ by Claim~\ref{r2stable}. Whether $P\neq P_u(\alpha,\beta,\varphi)$ or
	$P= P_u(\alpha,\beta,\varphi)$, since  $\alpha,\beta\notin\varphi(T(w))$, $T$ can still be obtained from $T_{n,q}$ by TAA under $\pi$, even though colors of some edges in $T-T(w)$ may be switched between $\alpha$ and $\beta$. Thus $T$ is an ETT under $\pi$. Because $\alpha,\beta\notin \Gamma^{q}$, we still see that $T=T_{n,q+1}$ satisfies Definition~\ref{R2} (i). So following Remark~\ref{R2check}, $T$ remains an ETT with a good hierarchy of $q$ levels under $\pi$. 
 \end{proof}

It is worth pointing out that, in the rest of the proof, after applying Claim~\ref{bstablechange},  to check if an ETT has a good hierarchy we only need to check Definition~\ref{R2} (i) following Remark~\ref{R2check}. If it is straightforward to check Definition~\ref{R2} (i), we may simply say it is routine to see that an ETT has a good hierarchy without mentioning Remark~\ref{R2check} and the details.  
\subsubsection{Case Verification}

Since $V(T)$ is not elementary, there exists a color $\alpha\in\overline{\varphi}(y_p)\cap\overline{\varphi}(v)$ for some $v\in V(T(y_{p-1}))$. Note that $\gamma_n\notin \Gamma^q\cup D_{n,q}$; we use this fact many times in our proof without mentioning it for a clean presentation. {\bf Recall that to simplify notation, we let $\Gamma^{q}_m=\{\gamma_{m1},\gamma_{m2}\}$ for $\delta_m\in D_{n,q}$, and reserve $\gamma,\tau$ for the two colors in $\phibar(x_e)$ in Definition~\ref{R2}. We assume without loss of generality that one of $\gamma$ and $\tau$ in $\phibar(x_e)$, say $\gamma$, is different from $\gamma_n$.} 

\begin{CLA}\label{ngamma}
	We may assume that $\alpha\neq\gamma_n$.
\end{CLA}	

\begin{proof}
 Assume $\alpha=\gamma_n$. Then we just let $\pi=\varphi/P_{y_p}(\gamma_n,\gamma,\varphi)$. By Claim~\ref{bstablechange} (i), $T_{n,q}$ is an ETT with a hierarchy of $(q-1)$ levels up to itself under the $T_n+f_n$ mod $\varphi$ coloring $\pi$. Moreover, since $\gamma_n,\gamma\in\phibar(T_{n})$ and $\gamma_n,\gamma\notin \Gamma^{q}$ following Definition~\ref{R2}, it is routine to see that $T$ remains an ETT under $\pi$ with a good hierarchy of $q$ levels and  $\gamma\in\pibar(y_p)\cap\pibar(T_{n})$, and therefore $T$ remains a minimum counterexample under $\pi$. So the claim is established by replacing $\varphi$ with $\pi$.
\end{proof}

\begin{CLA}\label{p>1}
	$p \geq 2$.
\end{CLA}
\proof
Suppose on the contrary $p=1$, that is, $T=T_{n,q}\cup\{e_1,y_1\}$ and $y_1$ is the end of $e_1$ outside of $T_{n,q}$. We consider two cases.

{\flushleft {\bf Case I:}  $q=0$. In this case $T_{n,q}=T_n$ is closed and $e_1=f_n$.}		

In this case, we have $\alpha\in\phibar(T_n)$ and $\alpha$ is closed for $T_{n}$. Since $\gamma_n$ is also closed for $T_n$ under $\varphi$, $P_{y_1}(\alpha,\gamma_n,\varphi)$ is different from $P_{v(\alpha)}(\alpha,\gamma_n,\varphi)=P_{v(\gamma_n)}(\alpha,\gamma_n,\varphi)$ by Claim~\ref{bstep1}(b). Let $\pi=\varphi/P_{y_1}(\alpha,\gamma_n,\varphi)$. By Claim~\ref{bstablechange} (i), $\pi$ is $T_n+f_n$ mod $\varphi$, and therefore $T_{n}+f_n$ is an ETT with $n$ rungs under $\pi$ with the same extension types, connecting colors, companion colors and connecting edges as under $\varphi$. Let $h$ be the initial index for $\Theta_{n}$. Consequently, $y_1$ belongs to a $(\gamma_h,\delta_h)$-ear of $T_h$ following Algorithm 3.1 ($h=n$ when $\Theta_n=\IE$). On the other hand, since $\gamma_n=\gamma_h\in\pibar(y_1)$, $f_n$ is no longer an edge contained in a $(\gamma_h,\delta_h)$-ear of $T_{h}$, a contradiction.

{\flushleft {\bf Case II:}  $q>0$. In this case $T_{n,q}$ is not closed although it is $(\cup_{\delta_h\in D_{n,q}}\Gamma^{q-1}_h)^-$ closed.}	

In this case $e_1$ is colored by $\beta\in\Gamma^{q-1}$. So $\beta\notin\{\gamma_n,\gamma\}$ and $\beta\in\phibar(T_{n,q})$ by Definition~\ref{R2}. Note that $\alpha\neq\gamma_n$ following Claim~\ref{ngamma}. If \(\alpha=\gamma\), take \(\pi=\varphi\) and apply the
same final argument below. Thus assume \(\alpha\ne\gamma\). Since $\gamma$ is closed for $T_{n,q}$ by Claim~\ref{close} and $\gamma\neq\gamma_n$,  $P_{x_e}(\gamma,\alpha,\varphi)=P_{v(\alpha)}(\gamma,\alpha,\varphi)$ is the only $(\gamma,\alpha)$-path intersecting $T_{n,q}$ by Claim~\ref{bstep1}. Thus $P_{y_1}(\gamma,\alpha,\varphi)$ is different from $P_{x_e}(\gamma,\alpha,\varphi)$. Following Claim~\ref{bstablechange}, $T_{n,q}$ is an ETT with a good hierarchy up to itself of $(q-1)$ levels under $\pi=\varphi/P_{y_1}(\gamma,\alpha,\varphi)$. Moreover, because $\beta\notin\{\gamma_n,\gamma\}$, $\pi$ is $(T,D_n,\varphi)$-weakly stable, and therefore it remains a counterexample under $\pi$. However, since $y_1$, $x_e$, and $v(\beta)$ induce at least two $(\beta,\gamma)$-paths intersecting $T_{n,q}$ under $\pi$ ($v(\beta)=x_e$ may happen), we reach a contradiction to Claim~\ref{bstep1}(b) under $\pi$.
\qed

Recall that the path number $p(T)$ of $T$ is the smallest index $i\in\{1,...,p\}$ such that the sequence $y_iT:=(y_i,e_{i+1},...,e_p,y_p)$ corresponds to a path in $G$, where $p\geq 2$ by Claim~\ref{p>1}. Depending on the value of $p(T)$, we divide the remainder of the proof into three cases. In these three cases, we are going to perform a lot of Kempe changes based on Claims~\ref{bchange} and~\ref{bstablechange} as above, where Claim~\ref{bchange} is used to show that a path $P$ is different from $P_{v(\alpha)}(\alpha,\beta,\varphi)$ to set up Claim~\ref{bstablechange}. We may skip mentioning some of these details when they are straightforward. By Claim~\ref{ngamma}, when applying Claims~\ref{bchange} and~\ref{bstablechange} to $\alpha$ later, we will assume $\alpha\neq\gamma_n$ by default without mentioning Claim~\ref{ngamma}, unless specified otherwise. 

\setcounter{case}{0}
\begin{case}\label{bcase1x}$p(T)=1$.
\end{case}			

\begin{CLA}\label{claimi}
	We may assume $\alpha\in\overline{\varphi}(y_i)\cap\overline{\varphi}(y_p)$ for some $1\leq i\leq p-1$.
\end{CLA} 	

\proof Suppose  $\alpha\in\phibar(y_p)\cap\phibar(T_{n,q})$.  We first consider the case  $\alpha\notin\varphi(T -  T_{n,q})$. Let $\beta\in\overline{\varphi}(y_{p-1})$. So $\beta\neq\gamma_n$ by (5.0). Note that $\beta\notin\varphi(T-T_{n,q}-f_n)$. Let $\pi=\varphi/P_{y_p}(\alpha,\beta,\varphi)$. Following Claim~\ref{bchange} and Claim~\ref{bstablechange} (i), $P_{y_{p-1}}(\alpha,\beta,\varphi)=P_{v(\alpha)}(\alpha,\beta,\varphi)$, and $T_{n,q}$ is an ETT with a good hierarchy up to itself under the $T_n+f_n$ mod $\varphi$ coloring $\pi$. Moreover, since $\alpha\notin \varphi(T-T_{n,q})$, it is routine to check that $\pi$ is $(T,D_n,\varphi)$-weakly stable and $T$ remains a counterexample with a good hierarchy under $\pi$. But now Claim~\ref{claimi} holds with $\beta\in\pibar(y_p)\cap\pibar(y_{p-1})$.

We now consider the case $\alpha\in\varphi(T - T_{n,q})$.
Following $\prec$, let $e_j$ be the first edge in $T -T_{n,q}$ such that $\alpha=\varphi(e_j)$. 
We first consider the case $j\geq 2$. Then $\alpha\notin\varphi(T({y_{j-1}})-T_{n,q})$. Let $\beta \in \overline{\varphi}(y_{j-1})$. So $\beta\neq\gamma_n$ and $\beta\notin \Gamma^{q}$ in view of (5.0) and Definition~\ref{R2}. Let $\pi=\varphi/P_{y_p}(\alpha,\beta,\varphi)$. Again by Claim~\ref{bchange} and Claim~\ref{bstablechange} (i), $P_{y_{j-1}}(\alpha,\beta,\varphi)=P_{v(\alpha)}(\alpha,\beta,\varphi)$, and it is routine to check that $T(y_{j})$ is an ETT with a good hierarchy under $\pi$ which is $(T(y_{j}),D_n,\varphi)$-weakly stable and $T_n+f_n$ mod $\varphi$. Note that if $\alpha\in\Gamma^{q}_m$ for some $\delta_m\in D_{n,q}$, then $\delta_m\in\phibar(T(y_{j-1}))$ following Definition~\ref{R2} (i). Moreover, since   $\alpha\notin\varphi(T(y_{j-1})-T_{n,q})$, $\beta\notin\varphi(T(y_j)-T_{n,q}-f_n)$ and $\beta\notin \Gamma^{q}$, $T$ is still an ETT satisfying Definition~\ref{R2} (i) as $T_{n,q+1}$. Thus it is routine to check that $T$ remains a counterexample with a good hierarchy under $\pi$. But now Claim~\ref{claimi} holds with $\beta\in\pibar(y_p)\cap\pibar(y_{j-1})$.

Now we assume that $j=1$. In this case $q>0$, since when $q=0$ we have $\alpha=\delta_n\notin\phibar(T_{n})$. Therefore, $\alpha=\varphi(e_1)$ where $\alpha\in\Gamma^{q-1}$ following Definition~\ref{R2} (iv). Note that then $\alpha\notin\Gamma^{q}$ by Definition~\ref{R2} (i). Note that we also have $\gamma\notin\Gamma^q$ by Definition~\ref{R2} (i). Let $\pi=\varphi/P_{y_p}(\alpha,\gamma,\varphi)$. Following Claim~\ref{bstep1}(b) and Claim~\ref{bstablechange} (i), since $\alpha,\gamma\notin\Gamma^q$ and $\alpha,\gamma\in\phibar(T_{n,q})$ with $\gamma$ closed for $T_{n,q}$, we have $P_{x_e}(\alpha,\gamma,\varphi)=P_{v(\alpha)}(\alpha,\gamma,\varphi)$, and it is routine to check that $T$ remains a counterexample with a good hierarchy under $\pi$ and $\pi(e_1)=\varphi(e_1)=\alpha$. If \(\gamma\notin\pi(T-T_{n,q})\), we return to the first case.
Otherwise, \(\pi(e_1)=\alpha\ne\gamma\), so we return to the
case \(j\ge2\).
\qed

\begin{CLA}\label{bp-1}
	We may further assume that $\alpha\in\overline{\varphi}(y_{p-1})\cap\overline{\varphi}(y_p)$.
\end{CLA}
\proof  Following Claim~\ref{claimi},  we assume that $i<p-1$ but $i$ is already maximum among all counterexamples. Let $\beta\in\phibar(y_{i+1})$. So $\beta\neq\gamma_n$. We first consider the case $\alpha,\beta\notin D_{n,q}$. Then $\alpha,\beta\notin\phibar(T_{n,q})\cup D_{n,q}$. Since $p(T)=1$, $y_{i}$ is an end of both $e_i$ and $e_{i+1}$. So both $e_i$ and $e_{i+1}$ are not colored by $\alpha$ as $\alpha\in\phibar(y_i)$. Since $T$ is obtained by TAA, we have $\alpha\notin\varphi(T(y_{i+1})-T_{n,q})$. Thus by Claim~\ref{bchange}, $P_{y_p}(\alpha,\beta,\varphi)$ is different from $P_{y_i}(\alpha,\beta,\varphi)=P_{y_{i+1}}(\alpha,\beta,\varphi)$. Therefore following Claim~\ref{bstablechange},   $T$ is a counterexample with a good hierarchy under $\pi=\varphi/P_{y_p}(\alpha,\beta,\varphi)$.   But $\beta\in\pibar(y_{i+1})\cap\pibar(y_p)$, a contradiction to the maximality of $i$. 

We now consider the case $\beta=\delta_m$ with $\delta_m\in D_{n,q}$. So $\gamma_{m1}\notin\varphi(T(y_{i+1})-T_{n,q})$ following Definition~\ref{R2} (i).
By Claim~\ref{bchange} with $\gamma_{m1}$ in place of $\alpha$, $P_{v(\gamma_{m1})}(\alpha,\gamma_{m1},\varphi)=P_{y_{i}}(\alpha,\gamma_{m1},\varphi)$, and $P_{y_{p}}(\alpha,\gamma_{m1},\varphi)$ is different from the path above. Let $\pi=\varphi/P_{y_{p}}(\alpha,\gamma_{m1},\varphi)$. Note that now $\alpha\notin\varphi(T(y_{i+1})-T_{n,q}-f_n)$ by the construction of $T$ with a similar argument as earlier.  Therefore, following Claim~\ref{bstablechange},  we see that $T$ remains a counterexample with a good hierarchy under $\pi$ with $\gamma_{m1}\notin\pi(T(y_{i+1})-T_{n,q})$. Then by Claim~\ref{bchange} again, $P_{v(\gamma_{m1})}(\delta_m,\gamma_{m1},\pi)=P_{y_{i+1}}(\delta_m,\gamma_{m1},\pi)$, and $P_{y_{p}}(\delta_m,\gamma_{m1},\pi)$ is different from the path above. Let $\sigma=\pi/P_{y_{p}}(\delta_m,\gamma_{m1},\pi)$.  As before, by Claim~\ref{bstablechange},  it is routine to check that $T$ remains a counterexample with a good hierarchy under $\sigma$. But now $\delta_m\in\sigmabar(y_p)\cap\sigmabar(y_{i+1})$, a contradiction to the maximality of $i$.

Finally we consider the case that $\alpha=\delta_m$ with $\delta_m\in D_{n,q}$. In this case, we may assume that one of $\gamma_{m1}$ and $\gamma_{m2}$, say $\gamma_{m1}$, is not used on $e_{i+1}$ under $\varphi$. Consequently, we see $\gamma_{m1}\notin\varphi(T(y_{i+1})-T_{n,q})$ by Definition~\ref{R2} (i). As in the previous case, we let $\pi=\varphi/P_{y_p}(\gamma_{m1},\delta_m,\varphi)$ and let $\sigma=\pi/P_{y_p}(\gamma_{m1},\beta,\pi)$. As above, it is routine to check that $T$ remains a counterexample with a good hierarchy under $\sigma$ with $\beta\in\sigmabar(y_p)\cap\sigmabar(y_{i+1})$, a contradiction to the maximality of $i$. We omit the details as it is similar to the previous case.\qed

Now we have $\alpha\in\overline{\varphi}(y_{p-1})\cap\overline{\varphi}(y_p)$.  Let $\varphi(e_{p})=\theta\in\phibar(T(y_{p-2}))$. 
 Since $\alpha\in\overline{\varphi}(y_p)\cap\overline{\varphi}(y_{p-1})$, we can recolor $e_{p}$ by $\alpha$. 
Let the resulting coloring be $\pi$. Clearly $\pi$ is $(T(y_{p-1}),D_n,\varphi)$-weakly stable and $(T_{n,q},D_n,\varphi)$-stable and $\theta\in\pibar(y_{p-1})\cap\pibar(T(y_{p-2}))$.  We claim 

(1) $\pi$ is also $T_n+f_n$ mod $\varphi$.

Note that by (1), following Claim~\ref{r2stable} and Remark~\ref{R2check}, we have a smaller counterexample $T(y_{p-1})$ with a good hierarchy of $q$ levels, contradicting our assumption on $T$.

To see (1), let $h$ be the initial index of $\Theta_n$. Since $\pi$ is $(T_n,D_n,\varphi)$-stable, following (A2.5), it is sufficient to show that $T_h$ has the same $(\gamma_h,\delta_h)$-exit vertex under $\pi$ as under $\varphi$. Clearly this holds when $\alpha,\theta\notin S_h$. Note that by (A2) on $T_h$, 

(2) Under $\varphi$, all the $\delta_h$-colored boundary edges of $T_h$ belong to some $(\gamma_h,\delta_h)$-ear of $T_h$ except for the only $(\gamma_h,\delta_h)$-exit edge.

Since $\pi$ is also $(T_h,D_h,\varphi)$-stable and $T_h$ is $\RE$ finished, again by (A2),

(3) $T_h$ is an ETT with $(h-1)$ rungs under $\pi$.

 We first consider the case $\theta\in S_h$. In this case, since we only changed the color of one edge from $\varphi$ to $\pi$, by (2), the only way for $T_h$ to have a different $(\gamma_h,\delta_h)$-exit vertex is for $e_p$ to be an edge in some $(\gamma_h,\delta_h)$-ear of $T_h$. However, this creates two more $(\gamma_h,\delta_h)$-exit paths of $T_h$ under $\pi$, a contradiction to (A2) on $T_h$ under $\pi$ in view of (3).

Finally we consider the case $\alpha\in S_h$. Since $\alpha\in \phibar(y_p)\cap\phibar(y_{p-1})$,  neither $y_p$ nor $y_{p-1}$ belongs to any $(\gamma_h,\delta_h)$-ear of $T_h$ under $\varphi$. So all the $(\gamma_h,\delta_h)$-ears of $T_h$ under $\varphi$ are colored the same under $\pi$. In view of (2), we see that $T_h$ still has the same $(\gamma_h,\delta_h)$-exit vertex under $\pi$ as under $\varphi$, as desired. This completes the proof of Case~\ref{bcase1x}.

\begin{case}\label{b2}
	$p(T)=p$. 
\end{case}
In this case, $y_{p-1}$ is not incident to $e_p$. Let $\theta =\varphi(e_p)$. Recall that $p\geq 2$ by Claim~\ref{p>1}, and we defined $T(y_{0})$ to be $T_{n,q}$. Let us define the following two tree-sequences for this case.



$\bullet$ $T^-=(T_{n,q}, e_1, y_1, e_2, \ldots, e_{p-2}, y_{p-2}, e_p, y_p)$ and

$\bullet$ $T^* \hskip 0.6mm =(T_{n,q}, e_1, y_1, e_2, \ldots ,y_{p-2}, e_p, y_p, e_{p-1}, y_{p-1})$.
Note that $T^-$ is obtained from $T$ by deleting $e_{p-1}$ and $y_{p-1}$ and $T^*$ arises from $T$ by interchanging
the order of $(e_{p-1}, y_{p-1})$ and $(e_p, y_p)$. 
\begin{subcase}\label{bcase611}
	$\alpha\in\overline{\varphi}(y_p)\cap\overline{\varphi}(y_{p-1})$ and $\alpha=\delta_m\in D_{n,q}$.
	\end{subcase}
Since $\delta_{m}\in\overline{\varphi}(y_{p})$, we have $\theta\neq\delta_{m}$. Note that $\theta\in D_{n,q}$ may occur. Moreover, by TAA, (5.0) and Definition~\ref{R2} (i), we have

(1) $\gamma_{m1},\gamma_{m2}\notin \varphi(T(y_{p-1})-T_{n,q})$ and $\delta_m\notin \varphi(T-T_{n,q}-f_n)$ ($\delta_m$ can only be used on $f_n$ in $T-T_{n,q}$, and it happens only when $q=0$ and $\delta_m=\delta_n$).

Following (1) and Claim~\ref{bchange}, we have 

(2) $P_{v_{\gamma_{mj}}}(\delta_m, \gamma_{mj},\varphi)=P_{y_{p-1}}(\delta_m,\gamma_{mj},\varphi)$ for $j=1,2$. 
\begin{subsubcase}\label{bcase611a}
	$\theta\notin\overline{\varphi}(y_{p-1})$.
\end{subsubcase}
We first consider the case $\theta\notin\Gamma^{q}$. In this case, it is easy to check that $T^*$ is also a counterexample with a good hierarchy of $q$ levels under $\varphi$.  Following (1), we have $\gamma_{m1},\gamma_{m2},\delta_m\notin \varphi(T^*-T_{n,q}-f_n)$. Thus by Claim~\ref{bchange} on $T^*$, we have $P_{v(\gamma_{mj})}(\delta_m, \gamma_{mj},\varphi)=P_{y_{p}}(\delta_m,\gamma_{mj},\varphi)$ for $j=1,2$, a contradiction to (2).

Now we assume $\theta\in\Gamma^{q}$. Without loss of generality, we say $\theta=\gamma_{k1}$ with  $\delta_k\in D_{n,q}$.  Note that if $\delta_k\notin\phibar(y_{p-1})$, then  $T^*$ is also a counterexample with a good hierarchy of $q$ levels under $\varphi$, and therefore we can get a contradiction to (2) again as above. So we may assume $\delta_k\in\phibar(y_{p-1})$. So following Definition~\ref{R2} (i),

(3) $\gamma_{k2}\notin\varphi(T-T_{n,q})$.

Note that $\delta_k=\delta_m$ may occur. Nonetheless, following (1), (3), Claims~\ref{bchange} and~\ref{bstablechange}, it is routine to check that $T$ is a counterexample with a good hierarchy under the $(T,D_n,\varphi)$-weakly stable coloring $\pi=\varphi/P_{y_p}(\delta_m,\gamma_{k2},\varphi)$ with $\gamma_{k_2}\in\pibar(y_p)\cap\pibar(T_{n,q})$ and $\gamma_{k2}\notin\pi(T-T_{n,q})$.  By Claim~\ref{b9n}, there exists a color $\beta\in\pibar(T(y_{p-2}))\setminus\{\gamma_n\}$ with $\beta\notin\pi(T-T_{n,q})$ such that either $\beta\notin \Gamma^q$ or $\beta\in\Gamma^q_r$ for some $\delta_r\in D_{n,q}\cap \phibar(T(y_{p-2}))$.  We then let $\pi_1=\pi/P_{y_p}(\gamma_{k2},\beta,\pi)$ and $\pi_2=\pi_1/P_{y_p}(\beta,\gamma_{k1},\pi_1)$. 
Since \(\gamma_{k2},\beta\notin\pi(T-T_{n,q})\) and
\(\beta\ne\gamma_n\), the assumption on \(\beta\) and repeated
applications of Claims~\ref{bchange} and~\ref{bstablechange} to \(\pi\) and \(\pi_1\)
show that \(T\) remains a counterexample with a good hierarchy
under \(\pi_2\), with
\(\pi_2(e_p)=\beta\in\phibar(T(y_{p-2}))\) and
\(\pibar_2(y_p)\cap\pibar_2(T_{n,q})\ne\emptyset\).
However, since $\delta_r\in\phibar(T(y_{p-2}))$ when $\beta\in\Gamma^q_r$, we see that $T^-$ is also a counterexample under $\pi_2$, which is a contradiction to the minimum assumption on $T$.

\begin{subsubcase}\label{bcase611b}
	$\theta \in\overline{\varphi}(y_{p-1})$. 
\end{subsubcase}
In this case $\theta\neq\delta_m$ and $\theta\neq\gamma_n$ (by (5.0)). Let $\pi=\varphi/P_{y_{p}}(\delta_m,\gamma_{m1},\varphi)$. Following (1), (2) and Claim~\ref{bstablechange}, it is routine to check that $T$ is a counterexample with a good hierarchy under the $(T,D_n,\varphi)$-weakly stable coloring $\pi$ and $\gamma_{m1}\notin\pi(T-T_{n,q})$. So again by Claims~\ref{bchange} and~\ref{bstablechange} with $\theta\neq\gamma_n$,  it is routine to check that $T$ remains a counterexample with a good hierarchy under the $(T(y_{p-1}),D_n,\pi)$-weakly stable coloring $\pi_1=\pi/P_{y_{p}}(\gamma_{m1},\theta, \pi)$. Note that under $\pi_1$, 	$\theta\in\pibar_1(y_p)\cap\pibar_1(y_{p-1})$ and $\pi_1(e_p)=\gamma_{m1}$. Now if $\theta\in D_{n,q}$, then under $\pi_1$ we have Case~\ref{bcase611a}. So we may assume $\theta\notin D_{n,q}$, which will be handled in Case \ref{bcase612a} below with $\gamma_{m1}$ in place of $\theta$.

\begin{subcase}\label{bcase612}
	$\alpha\in \phibar(y_p)\cap \phibar(y_{p-1})$ and $\alpha\notin D_{n,q}$.
\end{subcase}
 In this case,  because $\alpha\notin D_{n,q}$, following TAA we have 
 
 (4) $\alpha\notin\varphi(T - T_{n,q})$.

\begin{subsubcase}\label{bcase612a}
	$\theta  \notin\overline{\varphi}(y_{p-1})$.
\end{subsubcase}
Note that in this case, $T^*$ is also a counterexample with a good hierarchy, unless $\theta\in\Gamma^{q}_m$ for some $\delta_m\in D_{n,q}$ with $\delta_m\in\phibar(y_{p-1})$.
We first assume that there is no $m$ such that $\theta\in\Gamma^{q}_m$ with $\delta_m\in\ D_{n,q}$ and $\delta_m\in\phibar(y_{p-1})$. 
By Claim~\ref{b9n},  we have $|(\overline{\varphi}(T(y_{p-2}))\setminus\{\gamma_n\})\setminus \varphi(T(y_{p-2})-T_{n,q})|\geq 2n+2$. Recall that $|D_{n,q}|\leq n$ and $n\geq 1$. Since there are only two edges in $T-T(y_{p-2})$, there exists  a color $\beta\in\phibar (T(y_{p-2})) \setminus ( D_{n,q}\cup\{\gamma_n\})$ such that $\beta\notin\varphi(T - T_{n,q})$. Note that $\beta\notin\overline{\varphi}(y_{p})$. Otherwise, $T^-$ is a smaller counterexample, giving a contradiction to the minimum assumption of $T$.   Now since $\alpha,\beta\notin D_{n,q}$, $\alpha,\beta\neq\gamma_n$ and  $\beta\notin\varphi(T - T_{n,q})$, we have $P_{v(\beta)}(\alpha,\beta,\varphi)=P_{y_{p-1}}(\alpha,\beta,\varphi)$ by applying Claim~\ref{bchange} on $T$.  On the other hand, by applying Claim~\ref{bchange} on $T^*$,  we see that $P_{v(\beta)}(\alpha,\beta,\varphi)=P_{y_{p}}(\alpha,\beta,\varphi)$, a contradiction.  

We then assume  $\theta=\gamma_{m1}$ with $\delta_m\in D_{n,q}$ and $\delta_{m}\in\overline{\varphi}(y_{p-1})$. In this case we may also assume $\delta_m\notin\phibar(y_p)$, for otherwise we have Case~\ref{bcase611a}. Let $\pi=\varphi/P_{y_{p}}(\alpha,\gamma_{m2},\varphi)$ and $\pi_1=\pi/P_{y_{p}}(\gamma_{m2},\delta_m, \pi)$. Note that $\gamma_{m2}\notin\varphi(T - T_{n,q})$ and $\delta_m\notin\varphi(T-T_{n,q}-f_n)$ by Definition~\ref{R2} (i) and TAA. 
So with (4) and $\alpha\neq\gamma_n$, again by repeatedly applying Claims~\ref{bchange} and~\ref{bstablechange} on $\varphi$ and $\pi$, it is routine to check that $T$ remains a counterexample with a good hierarchy under the $(T, D_n,\varphi)$-weakly stable coloring $\pi_1$. Note that $\delta_m\in\pibar_1(y_p)\cap\pibar_1(y_{p-1})$ and $\gamma_{m1}=\pi_1(e_p)\notin\pibar_1(y_{p-1})$, so we are back to Case \ref{bcase611a}.

\begin{subsubcase}\label{bcase612b}
	$\theta\in\overline{\varphi}(y_{p-1})$.
\end{subsubcase}
We first assume  $\theta=\delta_{m}$ for some $\delta_m\in D_{n,q}$. Let $\pi=\varphi/P_{y_{p}}(\alpha,\gamma_{m1},\varphi)$ and $\pi_1=\pi/P_{y_{p}}(\delta_{m},\gamma_{m1},\pi)$. Since $\gamma_{m1}\notin\varphi(T -  T_{n,q})$ by Definition~\ref{R2} (i) and $\alpha\neq\gamma_n$, following (4) by repeatedly applying Claims~\ref{bchange} and~\ref{bstablechange} on $\varphi$ and $\pi$, it is again routine to check that $T$ remains a counterexample with a good hierarchy under the $(T(y_{p-1}), D_n,\varphi)$-weakly stable coloring $\pi_1$ with $\gamma_{m2},\alpha\notin\pi(T - T_{n,q})$. Note that $\delta_m\in\pibar_1(y_p)\cap\pibar_1(y_{p-1})$ and $\gamma_{m1}=\pi_1(e_p)\notin\pibar_1(y_{p-1})$, so we are back to Case \ref{bcase611a}.

We then assume $\theta\notin D_{n,q}$. Note that $\theta\neq\gamma_n$ by (5.0) as $\theta\in\phibar(y_{p-1})$.
By Claim~\ref{b9n}, there exists a color  $\beta\in\overline{\varphi}(T({y_{p-2}}))$ with $\beta\notin\varphi(T - T_{n,q})$ such that either $\beta\notin\Gamma^{q}\cup  D_{n,q}\cup\{\gamma_n\}$ or $\beta\in\Gamma^{q}_r$ with $\delta_{r}\in\phibar(T(y_{p-2}))$.   Let $\pi=\varphi/P_{y_{p}}(\alpha,\beta,\varphi)$ and $\pi_1=\pi/P_{y_{p}}(\theta,\beta,\pi)$. Since $\alpha,\beta\notin\varphi(T - T_{n,q})$ (see (4)), $\alpha,\beta,\theta\neq\gamma_n$, and if $\beta\in\Gamma^{q}_r$ then $\delta_{r}\in\phibar(T(y_{p-2}))$, by repeatedly applying Claims~\ref{bchange} and~\ref{bstablechange} on $\varphi$ and $\pi$, it is again routine to check that $T$ remains a counterexample with a good hierarchy under the $(T(y_{p-1}), D_n,\varphi)$-weakly stable coloring $\pi_1$. Note that we have $\theta\in\pibar_1(y_p)\cap\pibar_1(y_{p-1})$, $\pi_1(e_p)=\beta\notin\pibar_1(y_{p-1})$, so we are back to Case \ref{bcase612a}.

\begin{subcase}
	$\alpha\in\overline{\varphi}(y_p)\cap\overline{\varphi}(v)$ for a vertex  $v\prec y_{p-1}$.
\end{subcase}

In this case, we may assume that $\alpha\notin D_{n,q}$. Otherwise, say $\alpha=\delta_m$ for some $\delta_m\in D_{n,q}$. Then by Definition~\ref{R2} (i), $\gamma_{m1}\notin\varphi(T(v(\delta_{m}))-T_{n,q})$. So following Claims~\ref{bchange} and~\ref{bstablechange}, it is routine to check that $T$ remains a counterexample with a good hierarchy under $\pi=\varphi/P_{y_p}(\delta_m,\gamma_{m1},\varphi)$. Since $\gamma_{m1}\in\pibar(y_p)\cap \pibar(T_{n,q})$, we have as assumed by considering $T$ under $\pi$ with $\gamma_{m1}$ for $\alpha$.

\begin{CLA}\label{balpha} We may further assume
	$\alpha\notin\varphi(T - T_{n,q})$ such that either $\alpha\notin D_{n,q}\cup\Gamma^{q}\cup\{\gamma_n\}$, or 
	$\alpha\in \Gamma^q_m$  with $\delta_m\in D_{n,q}$ and $\delta_m\in\phibar(T({y_{p-2}}))$.
\end{CLA}

\proof  By Claim \ref{b9n}, there exists a color  $\beta\in\overline{\varphi}(T({y_{p-2}}))$ with $\beta\notin\varphi(T - T_{n,q})$ such that either $\beta\notin\Gamma^{q}\cup  D_{n,q}\cup\{\gamma_n\}$ or $\beta\in\Gamma^{q}_m$ with $\delta_{m}\in\phibar(T(y_{p-2}))$.  If $\beta\in\phibar(y_p)$, we are done. 
Hence we assume $\beta\notin\phibar(y_p)$. Let $P:=P_{y_p}(\alpha, \beta, \varphi)$.  We will show one of the following two statements holds.
\begin{itemize}
	\item[(a)] \label{bStatA} $\pi = \varphi/P$ is $(T,D_n,\varphi)$-weakly stable and $T$ is a counterexample with a good hierarchy.
	\item[(b)] \label{bStatB} Under $\varphi$, there exists a non-elementary ETT $T'$ with a good hierarchy of $q$ levels but  $p(T') < p(T)$.
\end{itemize}

Note that statement (a) gives Claim~\ref{balpha} while statement (b) gives a contradiction to the minimum assumption on $T$. Now if $V(P)\cap V(T(y_{p-1})) = \emptyset$, since $\alpha,\beta\notin D_{n,q}$, $\alpha,\beta\neq\gamma_n$, and if $\beta\in \Gamma^q_m$ then $\delta_m\in\phibar(T(y_{p-2}))$, following Claim~\ref{bstablechange} (both cases), it is routine to check we have statement (a).

Hence we assume $V(P)\cap V(T(y_{p-1}))\neq \emptyset$. Along the order of $P$ from $y_p$, let $u$ be the first vertex in $V(T(y_{p-1}))$ and $P'$ be the subpath joining $u$ and $y_p$. Define
\begin{eqnarray*}
	T' & = & T(y_{p-2})\cup P' \quad \mbox{ if $u\ne y_{p-1}$, and }\\
	T' & = & T(y_{p-1}) \cup P' \quad \mbox{ if $u=y_{p-1}$. }
\end{eqnarray*}
 Clearly $T'$ is an ETT under $\varphi$ as it can be obtained by TAA from $T_{n,q}$, and $T_{n,q}$ has a good hierarchy of $(q-1)$ levels up to itself. Note that $p(T')<p(T)$. We proceed with the proof by considering three cases:  $\alpha \notin \Gamma^{q}$, $\alpha \in \Gamma^{q} - \varphi(T-T_{n,q})$ and $\alpha \in \Gamma^{q}\cap \varphi(T-T_{n,q})$.
{\flushleft \bf Case I: $\alpha\notin \Gamma^{q}$.}	
Since $\alpha, \beta \in \phibar(T(y_{p-2}))$ and $\delta_m\in\phibar(T(y_{p-2}))$ whenever $\beta\in \Gamma^q_m$, we see that $T'=T_{n,q+1}$ satisfies Definition~\ref{R2} (i). Hence it is routine to see that statement (b) holds, a contradiction.


{\flushleft \bf Case II: $\alpha\in \Gamma^{q}\cap \varphi(T-T_{n,q})$.}  Assume $\alpha = \gamma_{k1}\in\Gamma^q$ with $\delta_k\in D_{n,q}$. Since $\varphi(e_p)\neq\alpha$, $\alpha\in\varphi(T(y_{p-1})-T_{n,q})$. Therefore we must have  $\delta_{k}\in\phibar(T(y_{p-2}))$ following Definition~\ref{R2} (i). Moreover, if $\beta\in \Gamma^q_m$, then $\delta_m\in\phibar(T(y_{p-2}))$, so we again see that $T'=T_{n,q+1}$ satisfies Definition~\ref{R2} (i). Hence it is routine to see that statement (b) holds, a contradiction.

{\flushleft \bf Case III: $\alpha \in \Gamma^{q} - \varphi(T-T_{n,q})$.} Then $\alpha,\beta\notin\varphi(T-T_{n,q})$ and $\alpha,\beta\neq\gamma_n$. Let $\pi=\varphi/P$. Again following Claims~\ref{bchange} and~\ref{bstablechange}, it is routine to check that $T$ remains a counterexample with a good hierarchy, and statement (a) holds. \qed

Now let $\varphi$ and $\alpha$ be as in the claim above. We then consider two cases.
\begin{subsubcase}
	$\theta=\varphi(e_p)\notin\phibar(y_{p-1})$.
\end{subsubcase}
Note that unless $\theta\in\Gamma^q_k$  and $\delta_{k}\in\overline{\varphi}(y_{p-1})$ for some $\delta_k\in D_{n,q}$, we see $T^-$ is a counterexample with a good hierarchy under $\varphi$, giving a contradiction to the minimum assumption on $T$. Hence we may assume $\theta=\gamma_{k1}$  and $\delta_{k}\in\overline{\varphi}(y_{p-1})$ with $\delta_k\in D_{n,q}$. So $\delta_k\notin\varphi(T-T_{n,q}-f_n)$. Following Definition~\ref{R2} (i),  we have  $\gamma_{k1}\notin\varphi(T(y_{p-1}) - T_{n,q})$. Let $\pi=\varphi/P_{y_p}(\alpha,\gamma_{k1},\varphi)$ and $\pi_1=\pi/P_{y_p}(\gamma_{k1},\delta_{k},\pi)$.
Following our assumption on $\alpha$ in Claim~\ref{balpha}, by repeatedly applying Claims~\ref{bchange} and~\ref{bstablechange} on $\varphi$ and $\pi$, it is again routine to check that $T$ is a counterexample with a good hierarchy under the $(T(y_{p-1}),D_n,\varphi)$-weakly stable coloring $\pi_1$. Since $\delta_k\in\pibar_1(y_p)\cap\pibar_1(y_{p-1})$ and $\pi_1(e_p)=\alpha$, we are back to Case \ref{bcase611a}.

\begin{subsubcase}
	$\theta=\varphi(e_p)\in\overline{\varphi}(y_{p-1})$.
\end{subsubcase}
In this case, we have $\theta\neq\gamma_n$ by (5.0). We first consider the case $\theta=\varphi(e_p)=\delta_{k}$ with $\delta_k\in D_{n,q}$. Following Definition~\ref{R2} (i), we see $\gamma_{k1}\notin\varphi(T -  T_{n,q})$. Let $\pi=\varphi/P_{y_p}(\alpha,\gamma_{k1},\varphi)$ and $\pi_1=\pi/P_{y_p}(\gamma_{k1},\delta_{k},\pi)$. Again following our assumption on $\alpha$ in Claim~\ref{balpha}, by repeatedly applying Claims~\ref{bchange} and~\ref{bstablechange} on $\varphi$ and $\pi$, it is routine to check that $T$ is a counterexample with a good hierarchy under the $(T(y_{p-1}),D_n,\varphi)$-weakly stable coloring $\pi_1$. Since $\delta_k\in\pibar_1(y_p)\cap\pibar_1(y_{p-1})$ and $\pi_1(e_p)=\gamma_{k1}$, we are back to Case \ref{bcase611a}.

We then consider the case $\theta=\varphi(e_p)\notin D_{n,q}$. Since $\theta\neq\gamma_n$, following our assumption on $\alpha$ in Claim~\ref{balpha}, by Claims~\ref{bchange} and~\ref{bstablechange} (both cases), it is routine to check that $T$ is a counterexample with a good hierarchy under the $(T(y_{p-1}),D_n,\varphi)$-weakly stable coloring $\pi=\varphi/P_{y_{p}}(\alpha,\theta,\varphi)$. Now $\theta\in \pibar(y_p)\cap \pibar(y_{p-1})$ and $\pi(e_p)=\alpha\notin\pibar(y_{p-1})$, which is dealt with in Case \ref{bcase612a}. This completes the proof of Case~\ref{b2}.




\begin{case}\label{bcase4}
	$2\leq p(T)\leq p-1$.
\end{case}

Define $I_{\varphi}=\{i:\overline{\varphi}(y_p)\cap\overline{\varphi}(y_{i})\ne \emptyset\ \ with\ \ 1\leq i\leq p-1\}$. So $I_{\varphi} = \emptyset$ when $\{ v\in V(T(y_{p-1})):\phibar(y_p)\cap \phibar(v) \ne \emptyset\} \subseteq V(T_{n,q})$.  By convention,  we denote $\max{I_{\varphi}} = 0$ when $I_{\varphi} = \emptyset$. Let $j= p(T)$. By the assumption of the current case,   we have  $j\geq 2$ and $e_j\notin E(y_{j-1},y_j)$. We first consider the case $max(I_{\varphi})\geq j$. This case is similar to Case~\ref{bcase1x} and can be handled in the same fashion:  We first show that $\max(I_{\varphi}) = p-1$ the same way as in Claim~\ref{bp-1},  and replace color $\varphi(e_p)$ by $\alpha$ to get a smaller counterexample with a new coloring $\pi$. The proof of $\pi$ being $T_n+f_n$ mod $\varphi$ is also the same as the end in Case~\ref{bcase1x}. We omit the details.

Now we assume that $max(I_{\varphi})< j$. Recall that we define $T(y_0)=T_{n,q}$.

\begin{CLA}\label{bj-1}
		We may assume there exists $\alpha\in\overline{\varphi}(y_p)\cap\overline{\varphi}(T(y_{j-2}))$ such that either $\alpha\notin \Gamma^q\cup D_{n,q}\cup\{\gamma_n\}$ or $\alpha\in\Gamma^q_m$ for some $\delta_m\in D_{n,q}$ with $\delta_m\in\phibar(T(y_{j-2}))$.
	
\end{CLA}

\proof Following Claim~\ref{ngamma}, let $\alpha\in\overline{\varphi}(y_p)\cap\overline{\varphi}(T(y_{j-1}))$ with $\alpha\neq\gamma_n$. We consider the following two cases.
{\flushleft \bf Case I:} $\alpha\in\phibar(y_{j-1})$. 
 
 By Claim~\ref{b9n}, there exists a color $\beta\in\overline{\varphi}(T(y_{j-2}))$ with $\beta\notin\varphi(T(y_{j-1})- T_{n,q})$ and either $\beta\notin\Gamma^q\cup D_{n,q}\cup\{\gamma_n\}$ or $\beta\in \Gamma^q_m$  for some $\delta_m\in D_{n,q}\cap\phibar(T(y_{j-2}))$.

We first assume $\alpha=\delta_k \in D_{n,q}$. In this case we let $\pi=\varphi/P_{y_p}(\delta_k, \gamma_{k1},\varphi)$ and $\pi_1=\pi/P_{y_p}(\gamma_{k1},\beta,\pi)$. Following Definition~\ref{R2} (i), we have  $\gamma_{k1}\notin\varphi(T(y_{j-1}) - T_{n,q})$. With the assumption on $\beta$, by repeatedly applying Claims~\ref{bchange} and~\ref{bstablechange} on $\varphi$ and $\pi$, it is routine to check that $T$ remains a counterexample with a good hierarchy under the $(T(y_{j-1}),D_n,\varphi)$-weakly stable coloring $\pi_1$.  Now $\beta\in\pibar_1(y_p)\cap \pibar_1(v(\beta))$,  and we have as claimed by considering $\pi_1$ with $\beta$.

We then assume $\alpha \notin D_{n,q}$. Note that $\alpha\notin\Gamma^q$ following Definition~\ref{R2} as $y_{j-1}\notin T_{n,q}$. Moreover, in this case, since $\alpha\neq\gamma_n$, by the assumption on $\beta$ and Claims~\ref{bchange} and~\ref{bstablechange}, it is routine to check that $T$ remains a counterexample with a good hierarchy under the $(T(y_{j-1}),D_n,\varphi)$-weakly stable coloring $\pi=\varphi/ P_{y_p}(\alpha,\beta,\varphi)$. Again we have as claimed by considering $\pi$ with $\beta$.

{\flushleft \bf Case II:} $\alpha\in\phibar(T(y_{j-2}))$. 

In this case, we may further assume $\alpha\notin D_{n,q}$. Otherwise, say $\alpha=\delta_m\in D_{n,q}$. Then $v(\delta_m)\in T(y_{j-1})-T_{n,q}$. Following Definition~\ref{R2} (i), we have $\gamma_{m1}\notin\varphi(T(v(\delta_m))-T_{n,q})$. Consequently, by Claims~\ref{bchange} and~\ref{bstablechange}, it is routine to check that $T$ remains a counterexample with a good hierarchy under the $(T(v(\delta_m)),D_n,\varphi)$-weakly stable coloring $\pi=\varphi/ P_{y_p}(\delta_m,\gamma_{m1},\varphi)$, and we have as claimed by considering $\pi$ with $\gamma_{m1}\in\pibar(y_p)\cap\pibar(T(y_{j-2}))$ and $\delta_m\in\pibar(T(y_{j-2}))$.

Now if $\alpha\notin\Gamma^{q}$, we are done. Hence we may assume $\alpha=\gamma_{m1}\in\Gamma^{q}_m$ with $\delta_m\in D_{n,q}$. Let us first consider the case that $\delta_m\notin\phibar(T(y_{p-1}))$. Then $\gamma_{m1}\notin\varphi(T-T_{n,q})$ by Definition~\ref{R2} (i). Let $\beta\in\phibar(y_{p-1})$. Since $p-1\geq j$, $\beta\notin\overline{\varphi}(T-T_{n,q}-f_n)$ by TAA and the fact that $e_p\in E(y_{p-1},y_p)$. Moreover, $\beta\neq\gamma_n$ by (5.0). So following Claims~\ref{bchange} and~\ref{bstablechange}, it is routine to check that 
$T$ remains a counterexample with a good hierarchy under the $(T,D_n,\varphi)$-weakly stable coloring $\pi=\varphi/P_{y_p}(\gamma_{m1},\beta,\varphi)$. With $\beta\in\pibar(y_p)\cap\pibar(y_{p-1})$, we reach the case that $max(I_{\pi})\geq j$ and then proceed as in Case~\ref{bcase1x}.

Now we assume that $\delta_m\in\phibar(T(y_{p-1}))$. Note that $\delta_m\notin\phibar(T_{n,q})$ since $\delta_m\in D_{n,q}$. So $\delta_{m}\in\phibar(y_k)$ for some $1 \leq k\leq p-1$. If $k< j-1$, then we have Claim~\ref{bj-1} already, so we assume $k\geq j-1$. Following Definition~\ref{R2} (i), we have $\gamma_{m1}\notin\varphi(T(y_{k})-T_{n,q})$. So following Claims~\ref{bchange} and~\ref{bstablechange}, it is routine to check that $T$ is a counterexample with a good hierarchy under $\pi=\varphi/P_{y_p}(\delta_m, \gamma_{m1},\varphi)$ with $\delta_m\in\pibar(y_p)\cap\pibar(y_k)$. Now if $k=j-1$, we return to Case I. Otherwise, $k\geq j$ and we reach the case that $max(I_{\pi})\geq j$, and then proceed as in Case~\ref{bcase1x}. 
\qed

\begin{CLA}\label{beta}
	We may assume there exists $\alpha\in\overline{\varphi}(y_p)\cap\overline{\varphi}(T(y_{j-2}))$ with $\alpha\notin\varphi(T(y_{j})-T_{n,q})$ such that either $\alpha\notin \Gamma^q\cup D_{n,q}\cup\{\gamma_n\}$ or $\alpha\in\Gamma^q_m$ for some $\delta_m\in D_{n,q}$ with $\delta_m\in\phibar(T(y_{j-2}))$.
	
\end{CLA}

\proof Let us consider a coloring $\varphi$ satisfying Claim~\ref{bj-1}.
 By Claim~\ref{b9n}, there exists a color $\beta\in\overline{\varphi}(T(y_{j-2}))$ with $\beta\notin\varphi(T(y_{j})- T_{n,q})$ and either $\beta\notin\Gamma^q\cup D_{n,q}\cup\{\gamma_n\}$ or $\beta\in \Gamma^q_m$  for some $\delta_m\in D_{n,q}\cap\phibar(T(y_{j-2}))$. If \(\beta\in\phibar(y_p)\), we are done with \(\beta\)
 in place of \(\alpha\). Thus assume otherwise.
  Now we consider the path $P:=P_{y_{p}}(\alpha,\beta,\varphi)$. First we consider the case $V(P)\cap V(T(y_{j-1}))\neq\emptyset$. Along the order of $P$ from $y_p$, let $u$ be the first vertex in $V(T(y_{j-1}))$ and $P'$ be the subpath joining $u$ and $y_p$. Define
\begin{eqnarray*}
	T' & = & T(y_{j-2})\cup P' \quad \mbox{ if $u\ne y_{j-1}$, and }\\
	T' & = & T(y_{j-1}) \cup P' \quad \mbox{ if $u=y_{j-1}$. }
\end{eqnarray*}
Clearly $T'$ is an ETT under $\varphi$ as it can be obtained by TAA from $T_{n,q}$, and $T_{n,q}$ has a good hierarchy of $(q-1)$ levels up to itself. By the assumptions on $\alpha$ and $\beta$, $T'$ also satisfies Definition~\ref{R2} (i). So following Remark~\ref{R2check}, $T'$ is a counterexample under $\varphi$. Note that we have $p(T')<p(T)$, giving a  contradiction to the minimum assumption on $T$.

Therefore we have $V(P)\cap V(T(y_{j-1}))=\emptyset$. Let $\pi=\varphi/P$.  By the assumptions on $\alpha$ and $\beta$, with Claim~\ref{bstablechange}, we see that $T$ remains a counterexample with a good hierarchy under the $(T(y_{j-1}),D_n,\varphi)$-stable coloring $\pi$. With $\beta\in\pibar(y_p)\cap\pibar(T(y_{j-2}))$, we have as claimed by considering $\beta$ and $\pi$.
\qed

Finally we consider a coloring $\varphi$ satisfying Claim~\ref{beta}, and let $\beta\in\phibar(y_j)$. Note that $\beta\neq\gamma_n$ by (5.0). We consider the following two cases based on $\beta$.
\begin{subcase}
	$\beta\notin D_{n,q}$.
\end{subcase}
Let $\pi=\varphi/P_{y_p}(\alpha,\beta,\varphi)$. Following the assumption on $\alpha$ with $\beta\neq\gamma_n$, by Claims~\ref{bchange} and~\ref{bstablechange}, it is routine to check that $T$ is a counterexample with a good hierarchy under $\pi$ with $\beta\in\pibar(y_p)\cap\pibar(y_j)$. So we reach the case that $max(I_{\pi})\geq j$, and then proceed as in Case~\ref{bcase1x}. 
\begin{subcase}
	$\beta=\delta_k\in D_{n,q}$.
\end{subcase}
In this case, by Definition~\ref{R2} (i), we have $\gamma_{k1}\notin\varphi(T(y_j)-T_{n,q})$. Let $\pi=\varphi/P_{y_p}(\alpha,\gamma_{k1},\varphi)$ and $\pi_1=\pi/P_{y_p}(\gamma_{k1},\delta_k,\pi)$. Again following the assumption on $\alpha$, by repeatedly applying Claims~\ref{bchange} and~\ref{bstablechange} on $\varphi$ and $\pi$, we see that $T$ is a counterexample with a good hierarchy under $\pi_1$ with $\delta_k\in\overline{\pi_1}(y_p)\cap\overline{\pi_1}(y_j)$. So again we reach the case that $max(I_{\pi_1})\geq j$, and then proceed as in Case~\ref{bcase1x}. 
This completes the proof of Case \ref{bcase4} and Statement A.
\end{proof}

\end{section}

\begin{section}{Conclusion}

Although the proof of Conjecture~\ref{con:GAS} is presented in terms of critical edges, it indeed provides a polynomial-time algorithm to achieve the bound in Conjecture~\ref{con:GAS}. For a given graph \(G\) with \(\Delta(G)\ge3\), set
\(k=\Delta(G)+1\) and construct a maximal partial
\(k\)-edge-coloring greedily, assigning to an uncolored edge
a color missing at both ends whenever possible. Then if there's still an uncolored edge $e$ and $G'$ is the subgraph induced by colored edges, following Algorithm 3.1 and the algorithm in the proof of Statement B of Lemma~\ref{tec}, we construct an ETT with a good hierarchy from $e$ until it is non-elementary, or Algorithm 3.1 stops due to no suitable ears, or an elementary closed tree is not strongly closed, but the required interchangeability or ear property fails, or it is elementary and strongly closed. If it is elementary and strongly closed, we update $k=k+1$ and repeat the process. Otherwise, we use the algorithms derived from Theorem~\ref{tech} to find a coloring of $G'+e$.

Following~\cite{HochNS80} by Hochbaum et al. and~\cite{StiebSTF-Book} by Stiebitz et al., we may always assume that $k=O(\Delta)$, where $\Delta=\Delta(G)$ is the maximum degree of $G$. Then, each set $\phibar(x)$ can be found in $O(\Delta)$, and therefore, one can decide in time $O(\Delta)$ whether two vertices have a common missing color. More importantly, the coloring $\varphi/P$ or $\varphi/(G-T,\alpha,\beta)$ can be constructed in time $O(|V(G)|+\Delta)$. In the proof of (A1), we used two nested ``minimal" assumptions (first $p(T)$, the path number and then $|T-T_{n,q}|$). In the general proof, we also used two parameters $n,q$, where $n$ is the rung number and $q$ is the level of the hierarchy. Note that the elementary ETTs considered here have at most
\(\min\{\Delta,|V(G)|\}\) vertices. Thus, the four parameters $p(T)$, $|T-T_{n,q}|$, $n$ and $q$ are all bounded by $\min\{\Delta, |V(G)|\}$. Our proof is actually an algorithm to find a smaller non-elementary ETT by reducing the lexicographic order of $(n, q, p(T),|T-T_{n,q}|)$ through Kempe changes of types $\varphi/P$ and $\varphi/(G-T,\alpha,\beta)$. So it is not difficult to see that the complexity of our coloring algorithm is $O(|E|(|V|+\Delta)\Delta^6)$, where $E=E(G)$ and $V=V(G)$. Moreover, if $k$ is preconfigured by calculating $\max\{ \Gamma(G), \Delta+1\}$, the coloring process can be done in  $O(|E|(|V|+\Delta)\Delta^5)$. 

For future work, with this new proof of the Goldberg-Seymour conjecture, we now can start to find algorithms for all previous results based on the conjecture itself or the proof of the conjecture. For example, Cao, Chen, Ding, Zang, and the author~\cite{gupta} obtained some partial results towards Gupta's co-density conjecture, which is commonly known as a dual version of the Goldberg-Seymour conjecture on edge cover packing. This new approach could be used to develop a polynomial-time algorithm for finding the corresponding edge covers. 
 In her Ph.D. dissertation \cite{fchromatic}, Hao improved the proof of the Goldberg-Seymour conjecture to confirm a generalized version of the Goldberg-Seymour conjecture on the $f$-chromatic index; see Chapter $8$ of~\cite{StiebSTF-Book} for more details on the $f$-chromatic index conjecture. It might also be possible that this new approach could be revised for a polynomial-time algorithm for the $f$-chromatic index. 

\end{section}

\begin{section}{Acknowledgment}
The author wishes to thank his wife, Wenyi, for her constant support. This work wouldn't have been possible without her. The author also thanks Jonathan Farley for careful reading and many helpful comments and corrections on earlier versions of the manuscript.

The author used OpenAI's GPT-6 Pro model in ChatGPT during
August--September 2026 to assist in auditing the proofs of
Theorems~3.2 and~3.3. The audit focused on logical dependencies,
recursive constructions, Kempe-change arguments, omitted cases, and
internal cross-references. The model was used as a proof-checking and
editorial aid and did not provide independent peer review. The author
independently evaluated all model-generated comments, made all final
mathematical and editorial decisions, and assumes full responsibility
for the correctness and content of the manuscript.
	\end{section}
\bibliography{gsbib}
\bibliographystyle{plain}


\end{document}